\documentclass{conm-p-l}

\usepackage[latin1]{inputenc}
\usepackage[T1]{fontenc}
\DeclareFontFamily{T1}{wncyr}{}
\DeclareFontShape{T1}{wncyr}{m}{n}{%
  <5><6><7><8><9>gen*wncyr%
  <10><10.95><12><14.4><17.28><20.74><24.88>wncyr10}{}

\usepackage{calc}
\usepackage{multicol}
\usepackage{amssymb}	
\usepackage{wasysym}
\usepackage{verbatim}	
\usepackage{mathrsfs}	
\usepackage{dsfont}
\usepackage[dvipsnames]{xcolor}
\usepackage{pgf}
\usepackage{pgfpages}
\usepackage{tikz}
\usetikzlibrary{trees,shapes,%
  matrix,arrows,decorations.pathmorphing,backgrounds,positioning,fit}
\usepackage{indentfirst}
\usepackage{cancel}
\usepackage{fixltx2e} 
\usepackage{hyperref}
\theoremstyle{plain}

\newtheorem{thm}{Theorem}[section]
\newtheorem{lem}[thm]{Lemma}
\newtheorem{prop}[thm]{Proposition}

\newtheorem*{claim}{Claim}

\theoremstyle{definition}

\newtheorem{defn}[thm]{Definition}
\newtheorem{exm}[thm]{Example}

\theoremstyle{remark}
\newtheorem{rem}[thm]{Remark}

\newcommand{\on}{\operatorname}

\newcommand{\mc}{\mathcal}

\newcommand{\mf}{\mathfrak}

\newcommand{\id}{\ensuremath{\mathop{\rm id\,}\nolimits}}

\newcommand{\Z}{\mathbb{Z}}

\newcommand{\C}{\mathbb{C}}
\newcommand{\Q}{\mathbb{Q}}

\newcommand{\A}{\mathbb{A}}

\newcommand{\Sn}[1][n]{\mf S_{#1}}

\newcommand{\ve}{\varepsilon}
\newcommand{\om}{\omega}

\newcommand{\ol}{\overline}

\newcommand{\st}{\stackrel}
\newcommand{\mx}{\mbox}

\newcommand{\geqs}{\geqslant}
\newcommand{\leqs}{\leqslant}

\newcommand{\w}{\wedge}
\newcommand{\wh}{\widehat}
\newcommand{\wt}{\widetilde}

\newcommand{\sm}{\setminus}
\newcommand{\sha}{\mathbin{\textup{\usefont{T1}{wncyr}{m}{n}\char120 }}}

\newcommand{\ds}{\displaystyle}

\newcommand{\ra}{\rightarrow}
\newcommand{\lra}{\longrightarrow}

\newcommand{\HH}{\on{H}}

\newcommand{\CH}{\on{CH}}

\newcommand{\p}{\mathbb{P}}

\newcommand{\Sp}{\on{Spec}}

\newcommand{\dme}[1][]{\mc{DM}_{gm}^{eff}}
\newcommand{\dm}[1][]{\mc{DM}_{gm}}

\newcommand{\SmCorr}[1][]{\on{\SmCorr}_{#1}}

\newcommand{\pst}[1]{\p^1 #1 \setminus \{0,1,\infty\} }
\newcommand{\ps}{\pst{}}

\newcommand{\Nc}[1][X]{{\mc N}^{\bullet}_{#1}}
\newcommand{\cN}[1][1]{{\mc N}^{#1}_{X}}
\newcommand{\cNg}[2]{{\mc N}^{#2}_{#1}}

\newcommand{\Nge}[2]{{\mc N}^{eq, \,#2}_{#1}}

\newcommand{\Alt}{{\mc A}lt}
\newcommand{\Lc}{\mc L}
\newcommand{\Lcu}{\Lc^{1}}

\newcommand{\Lcb}{\mc L^{B}}
\newcommand{\Lcub}{\Lc^{1,B}}

\newcommand{\Zc}{\mc Z}
\newcommand{\Ze}[1][\bullet]{\mc Z_{eq}^{#1}}

\newcommand{\ap}{{a'}}
\newcommand{\bp}{{b'}}

\newcommand{\dN}{\partial}
\newcommand{\da}[1][]{\partial_{#1}}

\newcommand{\dc}{d_{cy}}

\newcommand{\Lie}{\on{Lie}}
\newcommand{\Ttr}{\mc T^{tri}}

\newcommand{\Fq}{\mc F_{\Q}^{\bullet}}

\newcommand{\Td}{\mc T^{dec,or}}

\newcommand{\Tdr}{\mc T^{||}}
\newcommand{\T}[1]{T_{#1^*}}

\newlength{\ledge}
\newlength{\sibdis}
\newlength{\ecan}
\newlength{\ecas}
\newlength{\ecae}
\newlength{\ecao}
\newlength{\eca}
\newlength{\lbullet}
\tikzset{deftree/.style={level distance=\ledge,sibling distance=\sibdis}}
\tikzset{edgesp/.style={level distance=#1*\ledge,sibling distance=#1\sibdis}}
\tikzset{labf/.style={mathsc,yshift=-\eca}}
\tikzset{labfs/.style={mathss,yshift=-\eca}}
\tikzset{labr/.style={mathsc,yshift=\eca}}
\tikzset{labrs/.style={mathss,yshift=\eca}}
\tikzset{math mode/.style = {execute at begin node=$, execute at end node=$}}
\tikzset{mathscript mode/.style =%
 {execute at begin node=$\scriptstyle , execute at end node=$}}
\tikzset{math/.style = {execute at begin node=$, execute at end node=$}}
\tikzset{mathsc/.style =%
 {execute at begin node=$\scriptstyle , execute at end node=$}}
\tikzset{mathss/.style =%
 {execute at begin node=$\scriptscriptstyle , execute at end node=$}}
\tikzset{root/.style={draw,circle,inner sep=1pt,execute at begin node=$\bullet,
    execute at end node=$}}
\tikzset{roottest/.style={draw,circle,inner sep=#1pt}}
\tikzset{roots/.style={draw,circle,inner sep=2pt}}
\tikzset{bull/.style={fill,circle,minimum size=2pt,inner sep=0pt}}
\tikzset{leaf/.style={minimum size=2pt,inner sep=1pt}}
\tikzset{leafb/.style={minimum size=0pt,inner sep=0pt}}
\tikzset{intvertex/.style={mathsc,fill,circle,minimum size=0.6ex, inner sep=0pt}}
\tikzset{Reda/.style={-,double distance=0.3ex, draw=black}}
\tikzset{N/.style={-,thin,draw=black}}
\tikzset{Nd/.style={-,dotted, thin,draw=black}}
\newcommand{\prodracine}{%
\begin{tikzpicture}[%
baseline={([yshift=-0.5ex]current bounding box.center)},scale=0.4]
\tikzstyle{every child node}=[mathscript mode,minimum size=0pt, inner sep=0pt]
\node[minimum size=0pt, inner sep=0pt]
{}
[level distance=1.5em,sibling distance=3ex]
child {node[fill, circle, minimum size=2pt,inner sep=0pt]{}[level distance=1.5em]
  child{ node{}}
  child{ node{}}
};\end{tikzpicture}
}
\DeclareMathOperator{\prac}{\prodracine}
\newcommand{\prodracinered}{%
\begin{tikzpicture}[%
baseline={([yshift=-0.5ex]current bounding box.center)},scale=0.4]
\tikzstyle{every child node}=[mathscript mode,minimum size=0pt, inner sep=0pt]
\node[minimum size=0pt, inner sep=0pt]
{}
[level distance=1.5em,sibling distance=3ex]
child {node[fill, circle, minimum size=2pt,inner sep=0pt]{}[level distance=1.5em]
  child{ node{}edge from parent [N]}
  child{ node{}edge from parent [N]}
edge from parent [Reda]};\end{tikzpicture}
}
\DeclareMathOperator{\pracd}{\prodracinered}
\newcommand{\arbreaa}[3][1]{%
\begin{tikzpicture}[%
baseline={(current bounding box.center)},scale=#1]
\tikzstyle{every child node}=[mathscript mode,minimum size=0pt, inner sep=0pt]
\node[minimum size=0pt, inner sep=0pt]
{}
[level distance=1.2em,sibling distance=3ex]
child {node[fill, circle, minimum size=2pt,inner sep=0pt]{}[level distance=1.5em]
  child{ node[leaf]{#2}}
  child{ node[leaf]{#3}}
};\end{tikzpicture}
}

\newcommand{\rootaa}[3][1]{
\begin{tikzpicture}[baseline=(current bounding box.center),scale=#1]
\tikzstyle{every child node}=[mathscript mode,minimum size=0pt, inner sep=0pt]
\node[roots](root) {}
[deftree]
child {node[fill, circle, minimum size=2pt,inner sep=0pt]
           {}[level distance=1.5em]
  child{ node[leaf]{#2}}
  child{ node[leaf]{#3}}
};
\fill (root.center) circle (1/#1*\lbullet) ;
\end{tikzpicture}
}

\newcommand{\roota}[2][1]{
\begin{tikzpicture}[baseline=(current bounding box.center),scale=#1]
\tikzstyle{every child node}=[mathscript mode,minimum size=0pt, inner sep=0pt]
\node[roots](root) {}
[deftree]
child {node[leaf](1){#2} 
};
\fill (root.center) circle (1/#1*\lbullet) ;
\end{tikzpicture}
}

\newcommand{\TLx}[2][0.6]{
\begin{tikzpicture}[baseline=(current bounding box.center),scale=#1]
\tikzstyle{every child node}=[intvertex]
\node[roots](root) {}
[deftree]
child {node(1){} 
}
;
\fill (root.center) circle (1/#1*\lbullet) ;
\node[mathsc, xshift=-1ex] at (root.west) {#2};
\node[labf] at (1.south){0};
\end{tikzpicture} %
}
\newcommand{\TLy}[2][0.6]{
\begin{tikzpicture}[baseline=(current bounding box.center),scale=#1]
\tikzstyle{every child node}=[intvertex]
\node[roots](root) {}
[deftree]
child {node(1){} 
}
;
\fill (root.center) circle (1/#1*\lbullet) ;
\node[mathsc, xshift=-1ex] at (root.west) {#2};
\node[labf] at (1.south){1};
\end{tikzpicture} %
}

\newcommand{\TLxy}[2][0.6]{
\begin{tikzpicture}[baseline=(current bounding box.center),scale=#1]
\tikzstyle{every child node}=[intvertex]
\node[roots](root) {}
[deftree]
child {node{} 
    child{ node(1){}}
    child{ node(2){}} 
}
;
\fill (root.center) circle (1/#1*\lbullet) ;
\node[mathsc, xshift=-1ex] at (root.west) {#2};
\node[labf] at (1.south){0};
\node[labf] at (2.south){1};
\end{tikzpicture} 
}
\newcommand{\TLxxy}[2][0.6]{
  \begin{tikzpicture}[baseline=(current bounding box.center),scale=#1]
 \tikzstyle{every child node}=[intvertex]
 \node[roots](root) {}
 [deftree]
 child {node{}
   child[edgesp=2] {node(1){}}
   child {node{} 
     child{ node(2){}}  
     child{ node(3){}} 
   }
}
;
\fill (root.center) circle (1/#1*\lbullet) ;
\node[mathsc, xshift=-1ex] at (root.west) {#2};
\node[labf] at (1.south){0};
\node[labf] at (2.south){0};
\node[labf] at (3.south){1};
 \end{tikzpicture}
}
\newcommand{\TLxyy}[2][0.6]{
\begin{tikzpicture}[baseline=(current bounding box.center),scale=#1]
\tikzstyle{every child node}=[intvertex]
\node[roots](root) {}
[deftree]
child {node{} 
  child{node{}
    child{ node(1){}}
    child{ node(2){}}}  
  child[edgesp=2]{ node(3){}} 
}
;
\fill (root.center) circle (1/#1*\lbullet) ;
\node[mathsc, xshift=-1ex] at (root.west) {#2};
\node[labf] at (1.south){0};
\node[labf] at (2.south){1};
\node[labf] at (3.south){1};
\end{tikzpicture}
}

\newcommand{\TLxxxy}[2][0.6]{
 \begin{tikzpicture}[baseline=(current bounding box.center),scale=#1]
\tikzstyle{every child node}=[intvertex]
\node[roots](root) {}
[deftree]
child{node{}
  child [edgesp=3]{node(1){}}
  child{node{}
    child [edgesp=2]{node(2){}}
    child {node{} 
      child{ node(3){}}  
      child{ node(4){}} 
    }
  }
};

\fill (root.center) circle (1/#1*\lbullet) ;
\node[mathsc, xshift=-1ex] at (root.west) {#2};
\node[labf] at (1.south){0};
\node[labf] at (2.south){0};
\node[labf] at (3.south){0};
\node[labf] at (4.south){1};
\end{tikzpicture}
}
\newcommand{\TLxxyy}[2][0.6]{
\begin{tikzpicture}[baseline=(current bounding box.center),scale=#1]
\tikzstyle{every child node}=[intvertex]
\node[roots](root) {}
[deftree]
child{node{}
  child[edgesp=3]{node(1){}}
  child {node{} 
    child{node{}
      child{ node(2){}}
      child{ node(3){}}
    }  
    child[edgesp=2]{ node(4){}} 
  }
};
\fill (root.center) circle (1/#1*\lbullet) ;
\node[mathsc, xshift=-1ex] at (root.west) {#2};
\node[labf] at (1.south){0};
\node[labf] at (2.south){0};
\node[labf] at (3.south){1};
\node[labf] at (4.south){1};
\end{tikzpicture} 
+
\begin{tikzpicture}[baseline=(current bounding box.center),scale=#1]
\tikzstyle{every child node}=[intvertex]
\node[roots](root) {}
[deftree]
child {node{} 
  child{node{}
    child[edgesp=2]{node(1){}}
    child{node{}
      child{ node(2){}}
      child{ node(3){}}
    }  
  }
  child[edgesp=3]{ node (4){}} 
};
\fill (root.center) circle (1/#1*\lbullet) ;
\node[mathsc, xshift=-1ex] at (root.west) {#2};
\node[labf] at (1.south){0};
\node[labf] at (2.south){0};
\node[labf] at (3.south){1};
\node[labf] at (4.south){1};
\end{tikzpicture} 
}
\newcommand{\TLxxyya}[2][0.6]{
\begin{tikzpicture}[baseline=(current bounding box.center),scale=#1]
\tikzstyle{every child node}=[intvertex]
\node[roots](root) {}
[deftree]
child{node{}
  child[edgesp=3]{node(1){}}
  child {node{} 
    child{node{}
      child{ node(2){}}
      child{ node(3){}}
    }  
    child[edgesp=2]{ node(4){}} 
  }
};
\fill (root.center) circle (1/#1*\lbullet) ;
\node[mathsc, xshift=-1ex] at (root.west) {#2};
\node[labf] at (1.south){0};
\node[labf] at (2.south){0};
\node[labf] at (3.south){1};
\node[labf] at (4.south){1};
\end{tikzpicture} 
}

\newcommand{\TLxyyy}[2][0.6]{
\begin{tikzpicture}[baseline=(current bounding box.center),scale=#1]
\tikzstyle{every child node}=[intvertex]
\node[roots](root) {}
[deftree]
child {node{} 
  child{node{}
    child{node{}
      child{node (1){}}
      child{ node (2){}}
    }
    child[edgesp=2]{ node (3){}}
  }  
  child[edgesp=3]{ node (4){}} 
};
\fill (root.center) circle (1/#1*\lbullet) ;
\node[mathsc, xshift=-1ex] at (root.west) {#2};
\node[labf] at (1.south){0};
\node[labf] at (2.south){1};
\node[labf] at (3.south){1};
\node[labf] at (4.south){1};
\end{tikzpicture}
}

\newcommand{\TLxyxyy}[2][0.6]{
\begin{tikzpicture}[baseline=(current bounding box.center),scale=#1]
\tikzstyle{every child node}=[intvertex]
\node[roots](root) {}
[deftree]
child{node{}
  child[edgesp=3]{node{}[edgesp=1]
    child{node(1){}}
    child{node(2){}}
  }
  child[edgesp=2] {node{} 
    child[edgesp=1]{node{}
      child{ node(3){}}
      child{ node(4){}}
    }  
    child{ node(5){}} 
  }
};
\fill (root.center) circle (1/#1*\lbullet) ;
\node[mathsc, xshift=-1ex] at (root.west) {#2};
\node[labf] at (1.south){0};
\node[labf] at (2.south){1};
\node[labf] at (3.south){0};
\node[labf] at (4.south){1};
\node[labf] at (5.south){1};
\end{tikzpicture} 
+
\begin{tikzpicture}[baseline=(current bounding box.center),scale=#1]
\tikzstyle{every child node}=[intvertex]
\node[roots](root) {}
[deftree]
child {node{} 
  child{node{}
    child{node{}
      child[edgesp=2]{node(1){}}
      child{node{}
	child{node(2){}}
	child{ node(3){}}
      }
    }
    child[edgesp=3]{ node(4){}}
  }  
  child[edgesp=4]{ node(5){}} 
};
\fill (root.center) circle (1/#1*\lbullet) ;
\node[mathsc, xshift=-1ex] at (root.west) {#2};
\node[labf] at (1.south){0};
\node[labf] at (2.south){0};
\node[labf] at (3.south){1};
\node[labf] at (4.south){1};
\node[labf] at (5.south){1};
\end{tikzpicture} 
+
\begin{tikzpicture}[baseline=(current bounding box.center),scale=#1]
\tikzstyle{every child node}=[intvertex]
\node[roots](root) {}
[deftree]
child {node{}
  child{node{}
    child[edgesp=3]{node(1){}}
    child{node{}
      child{node{}
	child{node(2){}}
	child{ node(3){}}
      }
      child[edgesp=2]{ node(4){}}
    }  
  }
  child[edgesp=4]{ node(5){}} 
};
\fill (root.center) circle (1/#1*\lbullet) ;
\node[mathsc, xshift=-1ex] at (root.west) {#2};
\node[labf] at (1.south){0};
\node[labf] at (2.south){0};
\node[labf] at (3.south){1};
\node[labf] at (4.south){1};
\node[labf] at (5.south){1};
\end{tikzpicture} 
}

\newcommand{\TLraxyxyy}[5][0.6]{
\begin{tikzpicture}[remember picture, baseline=(current bounding box.center),scale=#1]
\tikzstyle{every child node}=[intvertex]
\node[roots](#5root) {}
[deftree]
child{node{}
  child[edgesp=3]{node{}[edgesp=1]
    child{node(#5l1){}}
    child{node(#5l2){}}
  }
  child[edgesp=2] {node{} 
    child[edgesp=1]{node{}
      child{ node(#5l3){}}
      child{ node(#5l4){}}
    }  
    child{ node(#5l5){}
          edge from parent node(#3){}} 
  }
};
\fill (#5root.center) circle (1/#1*\lbullet) ;
\node[mathsc, xshift=-1ex] at (#5root.west) {#2};
\node[labf] at (#5l1.south){0};
\node[labf] at (#5l2.south){1};
\node[labf] at (#5l3.south){0};
\node[labf] at (#5l4.south){1};
\node[labf] at (#5l5.south){1};
\end{tikzpicture} 
+
\begin{tikzpicture}[remember picture,baseline=(current bounding box.center),scale=#1]
\tikzstyle{every child node}=[intvertex]
\node[roots](#5rrroot) {}
[deftree]
child {node{}
  child{node{}
    child[edgesp=3]{node(#5lll1){}}
    child{node{}
      child{node{}
	child{node(#5lll2){}}
	child{ node(#5lll3){}}
      }
      child[edgesp=2]{ node(#5lll4){}
                       edge from parent node(eb){}}
    }  
  }
  child[edgesp=4]{ node(#5lll5){}} 
};
\fill (#5rrroot.center) circle (1/#1*\lbullet) ;
\node[mathsc, xshift=-1ex] at (#5rrroot.west) {#2};
\node[labf] at (#5lll1.south){0};
\node[labf] at (#5lll2.south){0};
\node[labf] at (#5lll3.south){1};
\node[labf] at (#5lll4.south){1};
\node[labf] at (#5lll5.south){1};
\end{tikzpicture} 
+
\begin{tikzpicture}[remember picture,baseline=(current bounding box.center),scale=#1]
\tikzstyle{every child node}=[intvertex]
\node[roots](#5rroot) {}
[deftree]
child {node{} 
  child{node{}
    child{node{}
      child[edgesp=2]{node(#5ll1){}}
      child{node{}
	child{node(#5ll2){}}
	child{ node(#5ll3){}}
      }
    }
    child[edgesp=3]{ node(#5ll4){}}
  }  
  child[edgesp=4]{ node(#5ll5){}} 
};
\fill (#5rroot.center) circle (1/#1*\lbullet) ;
\node[mathsc, xshift=-1ex] at (#5rroot.west) {#2};
\node[labf] at (#5ll1.south){0};
\node[labf] at (#5ll2.south){0};
\node[labf] at (#5ll3.south){1};
\node[labf] at (#5ll4.south){1};
\node[labf] at (#5ll5.south){1};
\end{tikzpicture} 
}
\title{Multiple zeta value cycles in low weight}
\author{Isma{\"e}l Soud{\`e}res}
\thanks{}
\address{Fachbereich Mathematik \\
Universität Duisburg-Essen, Campus Essen \\
Universitätsstrasse 2\\
45117 Essen\\
Germany \\
ismael.souderes@uni-due.de}
\address{
Max-Planck-Insititut für Mathematik, Bonn\\
Vivatsgasse 7 \\ 53111 Bonn \\
Germany \\
souderes@mpim-bonn.mpg.de }
\date{\today}

\begin{document}
\thanks{This work has been partially supported by DFG grant SFB/TR45 and by
  Prof. Levine Humboldt Professorship. 
I would like to thank P. Cartier for his
  attention to my work and H. Gangl for all his help, his comments and his
 stressing demand for me using tree related structures which allowed me to
 understand the Lie-like underlying combinatorics. Finally, none of this would
 have been possible without M. Levine patience, his explanations and deep
 inputs. This paper has been finalized during my stay at the MPIM and I am
 grateful to the MPIM for providing ideal working conditions and support.}
\begin{abstract}
In a recent work, the author has constructed two families of algebraic cycles
in Bloch's cycle algebra over $\ps$ that are expected to correspond to multiple
polylogarithms in one variable and have a good specialization at $1$ related to
multiple zeta values.

This is a short presentation, by the way of toy examples in low weight ($\leqs
5$), of this construction and could serve as an introduction to the general
setting. Working in low weight also makes it possible to push (``by hand'') the
construction further. In particular, we will not only detail the construction of
the cycles but we will also associate to these cycles explicit elements in the
bar construction over the cycle algebra and make as explicit as possible the
``bottom-left'' coefficient of the Hodge realization period matrix. That is, in
a few relevant cases we will associated to each cycle an integral showing how the
specialization at $1$ is related to multiple zeta values. We will be
particularly interested in  a new weight $3$ example corresponding to $-\zeta(2,1)$.
 \end{abstract}
\maketitle

\tableofcontents

\section{Introduction}

The multiple polylogarithm functions were defined in \cite{PAGGon} by the power
series
\[
Li_{k_1, \ldots, k_m}(z_1,\ldots, z_m)=\sum_{n_1> \cdots > n_m>0} 
\frac{z_1^{n_1}}{n_1^{k_1}} \, \frac{z_2^{n_2}}{n_2^{k_2}} \cdots
\frac{z_m^{n_m}}{n_m^{k_m}}
\qquad (z_i \in \C, \, |z_i|<1).
\]
They admit an analytic continuation to a Zariski open subset of $\C^m$. The case
$m=1$ is nothing but the classical \emph{polylogarithm} functions. The case
$z_1=z$ and $z_2=\cdots=z_m=1$ gives a one variable version of multiple
polylogarithm functions 
\[
Li^{\C}_{k_1, \ldots , k_m}(z)=Li_{k_1, \ldots, k_m}(z,1,\ldots, 1)=\sum_{n_1> \cdots > n_m>0} 
\frac{z^{n_1}}{n_1^{k_1} n_2^{k_2} \cdots n_m^{k_m}}.
\] 
When $k_1$ is greater or equal to $2$, the series converge as $z$ goes to $1$ and
one recovers the multiple zeta value
\[
\zeta(k_1,\ldots, k_m) =Li^{\C}_{k_1, \ldots , k_m}(1) = 
Li_{k_1, \ldots, k_m}(1,\ldots, 1)=\sum_{n_1> \cdots > n_m>0} 
\frac{1}{n_1^{k_1} n_2^{k_2} \cdots n_m^{k_m}}.
\]

To the tuple of integers $(k_1, \ldots, k_m)$ of weight $n=\sum k_i$, we can
associate a tuple of $0$'s 
and $1$'s
\[
(\ve_n,\ldots, \ve_1):=(\underbrace{0,\ldots , 0}_{k_1-1\mx{ \scriptsize  times}}, 1,\ldots
  ,\underbrace{0,\ldots ,0}_{k_m-1\mx{ \scriptsize  times}},
  1)
\]
which allows to write multiple polylogarithms as iterated integrals
($z_i\neq 0$ for all $i$):
\[
Li_{k_1,\ldots, k_m}^{\gamma}(z_1, \ldots,z_m)=(-1)^m \int_{\Delta_{\gamma}} 
\frac{dt_1}{t_1-\ve_1 x_1} \w \cdots \w
\frac{dt_n}{t_n-\ve_n x_n}
\]
where $\gamma$ is a path from $0$ to $1$ in $\C \sm \{x_1, \ldots, x_n\}$, the
integration domain $\Delta_{\gamma}$ is the associated real simplex consisting
of all $n$-tuples of points $(\gamma(t_1), \ldots, \gamma(t_n))$ with $t_i<t_j$
for $i<j$ and where we have set %
$x_n=z_1^{-1}$, $x_1=(z_1\cdots z_m)^{-1}$ and where, for all $i$ such that 
 $k_1+\cdots+k_{l-1}+1\leqs i<
k_1+\cdots+k_l$, we have set $x_{n-i}=(z_1\cdots z_l)^{-1}$. Classically,
$\gamma$ is the straight path from $0$ to $1$ : $\gamma(t)=t$ and in this case
the superscript will be omitted..  

Bloch and Kriz in \cite{BKMTM} have constructed an
algebraic cycle avatar of the classical polylogarithm function.
More recently in \cite{GanGonLev05}, Gangl, Goncharov and Levin, using a
combinatorial approach, have 
built algebraic cycles corresponding to the multiple polylogarithm values
$Li_{k_1,\ldots, k_m}(z_1, \ldots, z_m)$ with parameters $z_i$ satisfying in
particular that all the $z_i$ but $z_1$ have to be different from $1$
and their methods do not give algebraic cycles corresponding to multiple zeta
values.

The goal of the article \cite{SouMPCC} was to develop a geometric construction for multiple
polylogarithm cycles removing the previous obstruction which will allow to have
multiple zeta cycles.

A general idea underlying this project consists of looking for cycles fibered over a larger
base and not just point-wise cycles for some fixed parameter $(z_1,
...,z_m)$.  Levine in \cite{LEVTMFG} shows that there exists a short
exact sequence relating the Bloch-Kriz Hopf algebra over $\Sp(\Q)$, its
relative version over $\ps$ and the Hopf algebra associated to Goncharov and
Deligne's motivic fundamental group over $\ps$ which contains motivic avatars of 
iterated integrals associated to the multiple polylogarithms in one variable. 

As this one variable version of multiple polylogarithms gives multiple zeta
values for $z=1$, it is natural to investigate first the case of the Bloch-Kriz
construction over $\ps$ in order to 
obtain algebraic cycles corresponding to multiple polylogarithms in one variable
with a ``good specialization'' at $1$.

This paper presents the main geometric tools in order to construct such
algebraic cycles and applies the general construction described in \cite{SouMPCC}
to concrete examples up to weight $5$. In these particular cases, one can
easily go further in the description, lifting the obtained cycles to the bar
constructions over the Bloch's cycle algebra, describing the corresponding Bloch-Kriz
motive and computing some associated integrals 
related to the Hodge realization. Those integrals give back multiple
polylogarithms in one variable and their specialization at $1$  give multiple
zeta values.

The structure of the paper is organized as follows. In section 2 we review shortly
the combinatorial context as it provides interesting relations for the bar
elements associated to the cycles and an interesting relation with Goncharov's 
motivic coproduct for motivic iterated integrals. Section 3 is devoted to the
geometric situation and to the construction of the cycles after a presentation
of the Bloch's cycle 
algebra. Section 4, presents a combinatorial representation of the
constructed cycles as parametrized cycles.

Section 5 recalls the definition of the bar construction over a commutative
differential graded algebra and associates 
elements in the bar constructions (and a corresponding motive in the Bloch-Kriz
construction) to the low weight examples of cycles. Finally in section 6, I
follow Gangl, Goncharov and Levin's 
algorithm associating  an integral to some of the low weight algebraic cycles
previously described.
\section{Combinatorial situation}
In this paper a \emph{tree} is a planar finite tree whose internal vertices have
valency $\geqs 3$ and where at each vertex a cyclic ordering of the incident
edges is given. A \emph{rooted} tree has a distinguished external vertex called
the \emph{root} and a \emph{forest} is a disjoint union of trees.

Trees will be drawn with the convention that the cyclic ordering of the edges
around an internal vertex is displayed in counterclockwise direction. The root
vertex in the case of a rooted tree is displayed at the top.
\subsection{Trees, Lie algebras and Lyndon words}
 
Let $\Ttr$ be the $\Q$-vector space generated by rooted trivalent trees with
leaves decorated by $0$ and $1$ modulo the relation
\[
 \begin{tikzpicture}[baseline=(current bounding box.center),scale=0.8]
\tikzstyle{every child node}=[mathscript mode, inner sep=0.5pt]%
\node[inner sep=0.5pt]
 {$\scriptstyle T_1$}
[level distance=1.5em,sibling distance=3ex]
child {node[fill, circle, minimum size=2pt,inner sep=0pt]{}[level distance=1.5em]
  child{ node{T_2}}
  child{ node{T_3}}
};
\end{tikzpicture}
=-
\begin{tikzpicture}[baseline=(current bounding box.center),scale=0.8]
\tikzstyle{every child node}=[mathscript mode, inner sep=0.5pt]%
\node[ inner sep=0.5pt]
{$\scriptstyle T_1$}
[level distance=1.5em,sibling distance=3ex]
child {node[fill, circle, minimum size=2pt,inner sep=0pt]{}[level distance=1.5em]
  child{ node{T_3}}
  child{ node{T_2}}
};\end{tikzpicture}
\]
where the $T_i$'s are subtrees (and $T_1$ contains the root of the global tree).
Note that in the above definition, the root is not decorated.

Define on $\Ttr$ the internal law $\prac$ by
\[
\ds \rootaa{T_1}{T_2} \prac \rootaa{T_3}{T_4}= 
\begin{tikzpicture}[baseline=(current bounding box.center),scale=0.75]
\tikzstyle{every child node}=[mathscript mode, inner sep=0pt]
\node[draw,circle,inner sep=0.5pt](root) 
{$\bullet$}
[level distance=1.5em,sibling distance=3ex]
child {
 node[bull]
    {}[level distance=1.5em,sibling distance=6ex]
  child{ node[bull]{}[level distance=1.5em,sibling distance=3ex]
           child{node[leaf]{T_1}}
           child{node[leaf]{T_2}}
       }
  child{ node[bull]{}[level distance=1.5em,sibling distance=3ex]
          child{node[leaf]{T_3}}
          child{node[leaf]{T_4}}
       }
};\end{tikzpicture}.
\]
and extend it by bilinearity. One remarks that by definition $\prac$ is
antisymmetric.  Identifying $\{0, 1\}$ with  $\{X_0 , X_1 \}$ by the
obvious morphism and using the correspondence $\prac \leftrightarrow [ , ]$,
this internal law allows us to identify the free Lie algebra $\Lie(X_0,X_1)$
with $\Ttr$ modulo the Jacobi identity. Thus one can identify the (graded) dual
of $\Lie(X_0,X_1)$ as a subspace of $\Ttr$
\[
\Lie(X_0,X_1)^* \subset \Ttr.
\]

A Lyndon word in $0$ and $1$ is a word in $0$ and $1$ strictly smaller than any
of its nonempty proper right factors for the lexicographic order with $0<1$ (for
more details, see \cite{ReuFLA93}). The
standard factorization $[W]$ of a Lyndon word $W$ is defined inductively by
$[0]=X_0$, $[1]=X_1$ and otherwise by  $[W]=[[U],[V]]$ with $W=UV$, $U$ and $V$ nontrivial
and such that $V$ is minimal. The sets of Lyndon brackets $\{[W]\}$, that is Lyndon
words in standard factorization, form a basis of $\Lie(X_0,X_1)$ which can then
be used to write the Lie bracket 
\[
[[U],[V]]=
\sum_{\substack{W\, \mx{{\scriptsize Lyndon}} \\ \mx{{\scriptsize words}} }}
\alpha_{U,V}^W[W].
\]
with $U<V$ Lyndon words.
\begin{exm}Lyndon words in letters $0<1$ in lexicographic order are up to weight
  $5$: 
\begin{multline*}
0<00001<0001<00011<001<00101<0011 <00111<\\ 01<01011<011<0111<01111<1
\end{multline*}
\end{exm}

The above identification of $\Lie(X_0,X_1)$ as a quotient of $\Ttr$ and the
basis of Lyndon brackets allows us to define a family of trees dual to the Lyndon bracket
basis beginning with $\T 0=\roota 0$ and $\T 1=\roota 1$ and then setting
\begin{equation}\label{def:TW}
\T{W}=\sum_{U<V} \alpha_{U,V}^W \T{U} \prac \, \T{V}.
\end{equation}

\begin{exm} We give below the corresponding dual trees in weight $1$, $2$ and $3$
\[
\T{0}=\roota{0} ,\quad \T{1} = \roota 1 ,\quad \T{01}=\TLxy{} ,\quad
\T{001}=\TLxxy{},\quad 
\T{011}=\TLxyy{}.
\]
In weight $4$ appears the first linear combination
\[
\T{0001}=\TLxxxy{}
,\quad 
\T{0011}=\TLxxyy{}
,\quad
\T{0111}=\TLxyyy{},
\]
due to the fact that both $[0]\w[011]$ and $[001]\w [1]$ are mapped onto 
$[0011]=[X_0,[[X_0,X_1],X_1]]$ under the bracket map. 

In weight $5$, we will concentrate our attention to the two following examples
\[
\T{00101}=
\begin{tikzpicture}[baseline=(current bounding box.center),scale=0.6]
\tikzstyle{every child node}=[intvertex]
\node[roots](root) {}
[deftree]
child{node{}
  child[edgesp=2]{node{}
    child[edgesp=2]{node(3){}
      edge from parent [N]}
    child[edgesp=1] {node{}  
       child {node(4){}  
        edge from parent [N]}
       child {node(5){} 
        edge from parent [N]}
    edge from parent [N]}
  edge from parent [N]
  }
  child[edgesp=3] {node{}  
     child[edgesp=1] {node(1){}  
      edge from parent [N]}
     child[edgesp=1] {node(2){} 
      edge from parent [N]}
  edge from parent [N]
  }
edge from parent [N]
}
;
\fill (root.center) circle (1/0.6*\lbullet) ;
\node[labf] at (3.south){0};
\node[labf] at (4.south){0};
\node[labf] at (5.south){1};
\node[labf] at (1.south){0};
\node[labf] at (2.south){1};
\end{tikzpicture} %
-
\begin{tikzpicture}[baseline=(current bounding box.center),scale=0.6]
\tikzstyle{every child node}=[intvertex]
\node[roots](root) {}
[deftree]
child{node{}
  child{node{} 
    child[edgesp=3]{node(2){}
      edge from parent [N]}
    child{node{}
      child[edgesp=2]{node(3){}
        edge from parent [N]}
      child[edgesp=1] {node{}  
         child[edgesp=1] {node(4){}  
          edge from parent [N]}
         child[edgesp=1] {node(5){} 
          edge from parent [N]}
      edge from parent [N]}
    edge from parent [N]}
  edge from parent [N]
  }
  child[edgesp=4]{node(1){}
    edge from parent [N]
  }
edge from parent [N]
}
;
\fill (root.center) circle (1/0.6*\lbullet) ;
\node[labf] at (1.south){1};
\node[labf] at (2.south){0};
\node[labf] at (3.south){0};
\node[labf] at (4.south){0};
\node[labf] at (5.south){1};
\end{tikzpicture}
\] 
and 
\[
\T{01011}=\TLxyxyy{}.
\]
\end{exm}

\subsection{Another differential on trees}
\renewcommand{\Td}{\Fq}

In \cite{GanGonLev05}, Gangl, Goncharov and Levin introduced a
  differential $\dc$ on trees which reflects the differential in the Bloch's cycle
  algebra $\mc N_{\Sp(\Q)}$ (see Section \ref{sec:cycles}). In their work 
they have shown that
 some particular linear combinations of trivalent trees attached to
 decompositions of polygons have decomposable
 differential. More precisely, the differential of these particular linear
 combinations of trees is a linear combination of products of the same type of
 linear combinations of trees. The elements $\T W$ have a similar behavior
 under $\dc$.

One begins by endowing trees with an extra structure.
\begin{defn}\begin{itemize}
\item  An \emph{orientation} $\om$ of a tree $T$  (or a forest) is a numbering of the
 edges. That is if $T$ has $n$ edges and if $E(T)$ denotes its set of edges ,
 $\om$ is a map $E(T) \lra \{1, \ldots, n\}$.
\item Let $e(T)$ denote the cardinality of $E(T)$, that is the number of edges of $T$,
  and let $we(T)$, the weight of $T$, be the number of leaves of $T$. The degree of $T$
  is defined by $deg(T)=2we(T)-e(T)$. We extend these definitions to forests by linearity.
\end{itemize}
\end{defn}

\begin{defn}~

\begin{itemize}
\item Let $\ds V^{t}$ be the $\Q$-vector space generated by a unit $1$ and
  oriented forests  of
  rooted trees $T$  with root vertex decorated by: $t$, $0$ or $1$
and  leaves decorated by $0$ or $1$.  
\item Let $\cdot$ denote the product induced by
the disjoint union of the trees and shift of the numbering for the orientation
of the second factor. That is the product of $(F_1, \om_1)$ and $(F_2,\om_2)$ is
the forest $F=F_1\sqcup F_2$ together with the numbering $\om$ satisfying
$\om|_{E(F_1)}=\om_1$ and $\om|_{E(F_2)}=\om_2+n_1$ where $n_1=e(F_1)$ is the number of edges
in $F_1$. Note that here, by convention, the empty tree is $0$ and the unit for
$\cdot$ is the extra generator $1$.
\item Define $\Td$ to be the  algebra $V^t$ modulo the relations:
\[
(T,\sigma(\om))=\ve(\sigma)(T,\om),\qquad 
\begin{tikzpicture}[baseline=(current bounding box.center),scale=0.6]
\tikzstyle{every child node}=[mathss]
\node[roottest=2](root) {}
[level distance=1.5em,sibling distance=3ex]
child {node[fill, circle, minimum size=2pt,inner sep=0pt]{}[level distance=1.5em]
  child{ node[leaf]{T_1}}
  child{ node[leaf]{T_2}}
};
\node[mathss, xshift=-1ex] at (root.west) {0};
\fill (root.center) circle (2pt) ;
\end{tikzpicture} =0\qquad \text{and} \qquad  
\begin{tikzpicture}[baseline=(current bounding box.center),scale=0.6]
\tikzstyle{every child node}=[mathss]
\node[roottest=2](root) {}
[level distance=2em,sibling distance=3ex]
child {node[fill, circle, minimum size=2pt,inner sep=0pt](1){}[level distance=1.5em]
};
\node[mathss, xshift=-1ex] at (root.west) {1};
\node[labfs] at (1) {0}; 
\fill (root.center) circle (2pt) ;
\end{tikzpicture} =0.
\]
for any permutation $\sigma$ and where $\ve(\sigma)$ denotes the usual signature
of the permutation $\sigma$. 
\end{itemize}
The algebra $\Td$ endowed with the product $\cdot$ is graded commutative because
the orientation introduces signs into the usual disjoint union. Note
that for any forest $F$ one has $(-1)^{\deg(F)}=(-1)^{e(F)}$.
\end{defn}
\begin{rem}
\begin{enumerate}
\item There is an obvious direction on the edges of a rooted tree: away from
  the root.
\item A rooted tree comes with a canonical numbering, starting from the root
  edge and induced by the cyclic ordering at each vertex.
\end{enumerate}
\end{rem}
\newcommand{\sca}{1.0}
\renewcommand{\sca}{1.5}
\begin{exm} With our convention, 
an example of this canonical
ordering is shown at Figure \ref{orderT}; we recall that by convention we draw
trees with the root at the top and the cyclic order at internal vertices
counterclockwise.

\begin{figure}[htb]
\begin{tikzpicture}[baseline=(current bounding box.center),scale=\sca]
\tikzstyle{every child node}=[intvertex]
\node[roots](root) {}
[deftree]
child {node{} 
  child{node{}
    child{ node(1){}
    edge from parent node[mathsc,left]{e_3}
    }
    child{ node(2){}
    edge from parent node[mathsc,right]{e_4}
    }
    edge from parent node[mathsc, left]{e_2}
  }  
  child[edgesp=2]{ node(3){}
    edge from parent node[mathsc,auto]{e_5}
  } 
    edge from parent node[mathsc,auto]{e_1}
}
;
\fill (root.center) circle (1/\sca*\lbullet) ;
\node[mathsc, xshift=-1ex] at (root.west) {t};
\node[labf] at (1.south){0};
\node[labf] at (2.south){1};
\node[labf] at (3.south){1};
\end{tikzpicture}
\caption{A tree with its canonical orientation, that is the canonical
  numbering of its edges.} 
\label{orderT}
\end{figure}
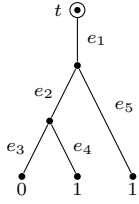
\end{exm}
\renewcommand{\sca}{1.0}
Now, we define on $\Td$ a differential
satisfying $d^2=0$ and the Leibniz rule
\[d(
(F_1,\om_1)\cdot (F_2,\om_2))=d((F_1,\om_1))\cdot (F_2,\om_2)
+(-1)^{e(F_1)}(F_1,\om_1)\cdot d((F_2,\om_2)).
\]

The set of rooted planar trees decorated as above endowed with their canonical
orientation forms a set of 
representatives for the permutation relation and it generates $\Fq$ as an
algebra. Hence, we will define this differential first on these trees and then
extend the definition by the Leibniz rule. 

The differential of an oriented tree $(T, \om)$ is a linear combination of oriented
forests where the trees appearing arise by contracting an edge  of $T$ and fall
into two types depending on whether the edge is internal or not. We will need the
notion of splitting.
\begin{defn}A \emph{splitting} of a tree $T$ at an internal vertex $v$  is the
  disjoint union of the trees which arise as $T_i \cup v$ where 
  the $T_i$ are the connected components of $T \sm v$. Moreover
\begin{itemize}
\item the planar structure of
  $T$ and its decorations of leaves induce on each  $T_i
  \cup v$ a planar structure  and decorations of leaves ;
\item an ordering of the edges of $T$ induces an orientation of the forest $\sqcup_i
(T_i\cup v)$;
\item if $T$ has a root $r$ then $v$ becomes the root for all $T_i\cup v$ which
  do not contain $r$, and if $v$ has a decoration then it keeps its decoration
  in all the $T_i \cup v$.
\end{itemize} 
 \end{defn}
\begin{defn} Let $e$ be an edge of a tree $T$. The contraction of $T$ along $e$
  denoted $T/e$ is given as follows:
\begin{enumerate}
\item If the tree consists of a single edge, its contraction is the empty tree.
\item If $e$ is an internal edge, then $T/e$ is the tree obtained from $T$ by
  contracting $e$ and identifying the incident vertices to a single vertex.
\item If $e$ is the edge containing the root vertex then $T/e$ is the forest
  obtained by first contracting $e$ to the internal incident vertex $w$ (which
  inherits the decoration of the root) and then by splitting at $w$;
  $w$ becoming the new root of all trees in the forest $T/e$.
\item If $e$ is an external edge not containing the root vertex then $T/e$ is the forest
  obtained as follows: first one  contracts $e$ to the internal incident vertex $w$ (which
  inherits the decoration of the leaf) and then one performs a splitting at $w$. 
\item If  $T$ is endowed with its  canonical
orientation  $\om$ there is a natural
orientation $i_e\om$ on $T/e$ given as follows :
\[\forall f \,\in E(T/e)\quad 
\begin{array}{ll}
i_e\om(f)=\om(f) & \mx{if }\om(f)<\om(e) \\
i_e\om(f)=\om(f)-1& \mx{if }\om(f)>\om(e).
\end{array}
\]
\end{enumerate}
\end{defn}
\begin{exm} Two examples are given below. In Figure \ref{rootcont}, one contracts
  the root vertex and in Figure \ref{leafcont}, a leaf is contracted. 
\newcommand{\scont}{1.5}
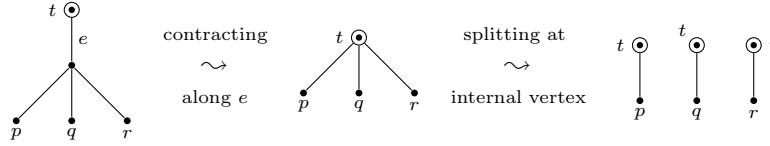
\begin{figure}[!htb]
\begin{tikzpicture}[baseline=(current bounding box.center),scale=\scont]
\tikzstyle{every child node}=[intvertex]
\node[roottest=2](root) {}
[deftree]
child {node{}
  child{node(1){}}
  child{node(2){}}
  child{node(3){}}
  edge from parent node[xshift=1ex, mathsc]{e}
};
\node[mathsc, xshift=-1ex] at (root.west) {t};
\node[labf] at (1.south) {p};
\node[labf] at (2.south) {q};
\node[labf] at (3.south) {r};
\fill (root.center) circle (1/\scont*\lbullet) ;
\end{tikzpicture} 
$\begin{array}{c}
\scriptstyle \textrm{contracting}\\ \leadsto \\\scriptstyle   \textrm{along }e\end{array}$
\begin{tikzpicture}[baseline=(current bounding box.center),scale=\scont]
\tikzstyle{every child node}=[intvertex]
\node[roottest=2](root) {}
[deftree]
  child{node(1){}}
  child{node(2){}}
  child{node(3){}} ;
\node[mathsc, xshift=-1ex] at (root.west) {t};
\node[labf] at (1.south) {p};
\node[labf] at (2.south) {q};
\node[labf] at (3.south) {r};
\fill (root.center) circle (1/\scont*\lbullet) ;
\end{tikzpicture} 
$\begin{array}{c}\scriptstyle{\textrm{splitting at }}\\
{\leadsto}\\
\scriptstyle \textrm{internal vertex}\end{array}$
\begin{tikzpicture}[baseline=(current bounding box.center),scale=\scont]
\tikzstyle{every child node}=[intvertex]
\node(ancre){};
\node[roottest=2,xshift=-5ex] (root) at (ancre){}
[deftree]
  child{node(1){}
};
\node[roottest=2,xshift=0ex] (roott) at (ancre){}
[deftree]
  child{node(2){}
};
\node[roottest=2,xshift=5ex] (roottt) at (ancre){}
[deftree]
  child{node(3){}
};
\node[mathsc, xshift=-1ex] at (root.west) {t};
\node[mathsc, yshift=0.8ex, xshift=-0.8ex] at (roott.130) {t};
\node[mathsc, xshift=1ex] at (roottt.east) {t};
\node[labf] at (1.south) {p};
\node[labf] at (2.south) {q};
\node[labf] at (3.south) {r};
\fill (root.center) circle  (1/\scont*\lbullet) ;
\fill (roott.center) circle  (1/\scont*\lbullet) ;
\fill (roottt.center) circle  (1/\scont*\lbullet) ;
\end{tikzpicture}
\caption{Contracting the root}  
\label{rootcont}
\end{figure}
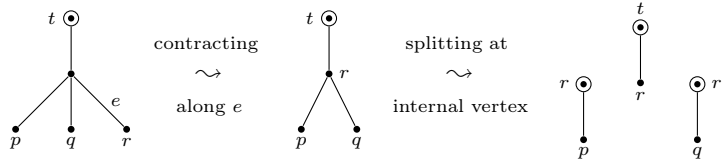
\begin{figure}[!htb]
\begin{tikzpicture}[baseline=(current bounding box.center),scale=\scont]
\tikzstyle{every child node}=[intvertex]
\node[roottest=2](root) {}
[deftree]
child {node{}
  child{node(1){}}
  child{node(2){}}
  child{node(3){}
        edge from parent node[xshift=1.5ex, mathsc]{e}}
};
\node[mathsc, xshift=-1ex] at (root.west) {t};
\node[labf] at (1.south) {p};
\node[labf] at (2.south) {q};
\node[labf] at (3.south) {r};
\fill (root.center) circle  (1/\scont*\lbullet) ;
\end{tikzpicture}$\begin{array}{c}\scriptstyle{\textrm{contracting }}\\
{\leadsto}\\
\scriptstyle \textrm{along }e\end{array}$ 
\begin{tikzpicture}[baseline=(current bounding box.center),scale=\scont]
\tikzstyle{every child node}=[intvertex]
\node[roottest=2](root) {}
[deftree]
child {node(3){}
  child{node(1){}}
  child{node(2){}}
};
\node[mathsc, xshift=-1ex] at (root.west) {t};
\node[labf] at (1.south) {p};
\node[labf] at (2.south) {q};
\node[mathsc,xshift=1ex] at (3.east) {r};
\fill (root.center) circle  (1/\scont*\lbullet) ;
\end{tikzpicture} $\begin{array}{c}\scriptstyle{\textrm{splitting at }}\\
{\leadsto}\\
\scriptstyle \textrm{internal vertex}\end{array}$
\begin{tikzpicture}[baseline=(current bounding box.center),scale=\scont]
\tikzstyle{every child node}=[intvertex]
\node(ancre){};
\node[roottest=2,xshift=-5ex] (root) at (ancre){}
[deftree]
  child{node(1){}
};
\node[roottest=2,xshift=5ex] (roott) at (ancre){}
[deftree]
  child{node(2){}
};
\node[roottest=2,yshift=5ex] (roottt) at (ancre){}
[deftree]
  child{node(3){}
};
\node[mathsc, xshift=-1ex] at (root.west) {r};
\node[mathsc, xshift=1ex] at (roott.east) {r};
\node[mathsc, yshift=1ex] at (roottt.north) {t};
\node[labf] at (1.south) {p};
\node[labf] at (2.south) {q};
\node[labf] at (3.south) {r};
\fill (root.center) circle  (1/\scont*\lbullet) ;
\fill (roott.center) circle  (1/\scont*\lbullet) ;
\fill (roottt.center) circle  (1/\scont*\lbullet) ;
\end{tikzpicture}  
\caption{Contracting a leaf} 
\label{leafcont} 
\end{figure}
\end{exm}

\begin{defn}
Let $(T,\om)$ be a tree endowed with its canonical orientation, one defines $\dc(T,w)$ as 
\[
\dc(T,\om)=\sum_{e \in E(T)} (-1)^{\om(e)-1}(T/e, i_e\om).
\] 
One extends $\dc$ to all oriented trees by the relation 
 $ 
\dc(T,\sigma \circ \om)=\ve(\sigma)\dc(T,\om)
$ 
and to $\Fq$ by linearity and the Leibniz rule.

In particular $\dc$ maps a tree with at most one edge to $0$ (which corresponds
by convention to the empty tree).
\end{defn}
As proved in \cite{GanGonLev05}, $\dc$, extended with the Leibniz rule, induces
a differential on $\Fq$.  
\begin{prop} The map $\dc : \Fq \lra \Fq$ makes $\Fq$ into a commutative
  differential graded algebra. In particular $\dc^2=0$.
\end{prop}

By an abuse of notation, for any Lyndon word $U$ the image of
  $\T{U}$ in $\Td$ with root vertex decorated by $t$ and canonical orientation
  is also denoted by $\T U$. The image of $\T U$ in $\Td$ with root vertex
  decorated by $1$ and canonical orientation is denoted by $\T U(1)$.

The main result of the combinatorial aspects is the following.
\begin{thm}  
Let $W$ be a Lyndon word.  Then the following equality holds in $\Td$:
\begin{equation} \tag{ED-T}\label{ED-T} 
\ds \dc(\T{W})=\sum_{U<V}\alpha_{U,V}^W\T{U}\cdot\T{V} 
+\sum_{U,V}\beta_{U,V}^W\T{U}\cdot\T{V}(1)
\end{equation}
 where the $\alpha_{U,V}^W$ are the ones from Equation
\eqref{def:TW}. In the above equation, coefficients $\alpha$'s and $\beta$'s are
in $\Z$.
\end{thm}
For a detailed proof, we refer to \cite{SouMPCC}. The theorem mainly follows by
induction from the combinatorics of the free Lie algebra $\Lie(X_0,X_1)$: 
\begin{itemize}
\item The terms in the first sum in the R.H.S of \eqref{ED-T} come from the contraction of
  the root edge which is nothing but the differential $d_{Lie}$ dual to the bracket
  of $\Lie(X_0,X_1)$.
\item Using the inductive definition of $\T W$ (cf. Equation \eqref{def:TW}), one shows
  by induction  that, as $d_{Lie}^2=0$, internal edges do not contribute. 
\item Terms in $\T{V}(1)$ arise from leaves decorated by $1$. The fact that
  terms arising from leaves decorated by $1$ can be regrouped as a product $\T U
  \cdot \T{V}(1)$ is due to a particular decomposition of some specific brackets
  in terms of the Lyndon basis. 
\end{itemize}

\begin{exm}\label{exm:dT}
As mentioned before, the trees are endowed with their canonical numbering. 
First on remarks that trees with only one edge are mapped to $0$ so 
\[
\dc(\T 0)=\dc(\TLx t)=0 \qquad \mx{and} \qquad 
\dc(\T 1)=\dc(\TLy t)=0.
\]

We recall
that a tree with root decorated by $0$ is $0$ in $\Td$. As applying an odd permutation
to the numbering changes the sign of the tree, the trivalency of the tree
$\T W$ shows that some trees coming from the computation of $\dc$ are $0$ in
$\Td$ because they contain a symmetric subtree; that is they contain a subtree of the form 
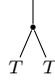
\[
 \arbreaa[1]{T}{T}
\]
where $T$ is a trivalent tree.
 
Using the fact that the
tree $\begin{tikzpicture}[baseline=(current bounding box.center),scale=1]
\tikzstyle{every child node}=[intvertex]
\node[roots](root) {}
[deftree]
child {node (1){} 
};
\fill (root.center) circle (1/1*\lbullet) ;
\node[mathsc, xshift=-1ex] at (root.west) {1};
\node[labf] at (1.south){0};
\end{tikzpicture}$ is $0$ in $\Fq$, one computes in weight $2$
\[
\dc(\T{01})=\dc(\TLxy[1] t)= \TLx[1.5] t \cdot \TLy[1.5] t= \T 0 \, \T 1.
\]
In weight $3$, one has 
\[
\dc(\T{001})=\dc(\TLxxy t)=\TLx[1.5] t \cdot \TLxy[1] t = \T 0 \, \T{01} 
\] 
and  
\begin{align*}
\dc(\T{011})&= \dc\left( 
\TLxyy t \right)= \TLxy[1] t \cdot \TLy[1.5] t + \TLy[1.5] t \TLxy[1] 1 \\
&=\T{01} \, \T 1 + \T 1 \, \T{01}(1). 
\end{align*}

In weight $4$, one can easily check that 
\[
\dc(\T{0001})=\T 0  \, \T{001} 
\qquad \mx{and} \qquad 
\dc(\T{0111})=\T{011}\, \T 1 + \T 1 \T{011}(1).
\]
The example of $\T{0011}$ is more interesting. 

\begin{multline} \label{exm:dT0011T}
\dc \left(\TLxxyy[0.8] t 
\right)
= 
\TLx[2] t \cdot \TLxyy[1] t +  \TLxxy[1] t \cdot \TLy[2] t \\
 +\TLy[2] t \cdot \TLxxy[1] 1 +  \TLxy[1.5] t \cdot \TLxy[1.5] 1
\end{multline}
That is:
\begin{equation}\label{exm:dT0011}
\dc(\T{0011})= \T 0 \, \T{011} + \T{001} \, \T 1 + 
\T 1 \, \T{001}(1) + \T{01} \T{01}(1)
\end{equation}
In the above equations, the term in $\T{01} \, \T{01}(1)$ is coming form the
last edge of the tree  
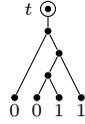
\[
\TLxxyya t
\]
appearing in $\T{0011}$. 

Computing $\dc^2(\T{0011})$ (which is $0$), the
differential $\dc(\T{01} \, \T{01}(1))$ cancels with the term in $\T 0 \, \T 1
\T{01}(1)$ arising from $\dc(\T 0 \, \T{011})$. It can be thought of as the
propagation of the weight $3$ correction term $\T 1 \T{01}(1)$ 
appearing in $\dc(\T{011})$.

We give below an example in weight $5$,  $\dc(\T{01011})$ :
\begin{multline*}
\dc\left(\TLraxyxyy[0.8]{t}{ea}{eb}{a}\right)= \\
 \TLxy[1.7]{t} \cdot \TLxyy[1.5]{t}
%
+ \left(\TLxxyy[1] {t} \right)\cdot \TLy[2.1]{t} \\
+ \TLy[2.1]{t} \cdot \left( \TLxxyy[1]{1} \right)
+\tikz[remember picture]{\coordinate(deux);}
2\TLxyy[1.5]{t} \cdot \TLxy[1.7]{1}
\begin{tikzpicture}[remember picture, overlay]
\node[mathsc] at (ea.east) {e};
\node[mathsc] at (eb.east) {f};
\end{tikzpicture} 
\end{multline*}
where the last term arises from the part of the differential associated to
edges $e$ and $f$. The above equation can be written as
\begin{equation}\label{exm:dT01011}
\T{01011}=\T{01}\cdot \T{011}+\T{0011}\cdot \T{1}
+\T{1}\cdot \T{0011}(1)+2\T{011}\cdot \T{01}(1)
\end{equation}
\end{exm}

\section{Algebraic cycles}\label{sec:cycles}
This section begins with the construction of the cycle complex (or cycle
algebra) as presented 
in \cite{BlochACHKT,BlochLMM,BKMTM, LevBHCG}. Then, we give some properties of
equidimensional cycles groups over $\A^1$ and build some algebraic cycles
corresponding to multiple polylogarithms in one variable. 

Here the base field is $\Q$ and the various structures
have $\Q$ coefficients. 
\subsection{Construction of the cycle algebra}
Let $\square^n$ be the algebraic $n$-cube
\[
\square^n=(\p^1\sm \{1\})^n.
\] 
Insertion morphisms $s_{i}^{\ve} :\square^{n-1} \lra \square^{n} $ are given by
the identification
\[
\square^{n-1}\simeq \square^{i-1}\times \{\ve\}\times \square^{n-i}
\]
for $\ve=0,\infty$.  A \emph{face} $F$ of codimension $p$ of $\square^n$
is given by the equation $x_{i_k}=\ve_k$ 
for $k$ in $\{1, \ldots, p\}$ and  $\ve_k$ in $\{0,\infty\}$ where $x_1, \ldots,
x_n$ are the usual affine coordinates on $\p^1$. In particular, codimension $1$
faces are given by the images of insertion morphisms.

Now, let $X$ be a smooth  irreducible quasi-projective variety over ${\Q}$.
\begin{defn} Let $p$ and $n$ be non-negative integers. Let $\Zc^p(X,n)$ be the
  free group generated by closed irreducible
  sub-varieties of $X \times \square^n$ of codimension $p$ which intersect all
  faces $X \times F$ properly (where $F$ is a face of $\square^n$). That is:
\[
\Z\left< Z \subset X \times \square^n \text{ such that} \left\{
\begin{array}{l}
Z \text{ is closed and irreducible} \\
\on{codim}_{X\times F}(Z\cap (X \times F))=p \\
 \text{or } Z \cap (X
\times F)= \emptyset
\end{array}
\right.
\right>
\]
\end{defn}
 A sub-variety $Z$ of $X \times \square^n$ as above is \emph{admissible}.
The insertion morphisms $s_i^{\ve}$ induce a well defined pull-back 
$s_i^{\ve \, *} : \Zc^p(X,n) \ra \Zc^p(X,n-1)$ and a differential: 
\[
\dN =
\sum_{i=1}^{n}(-1)^{i-1} (s_i^{0\,*}-s_i^{\infty\,*}) : 
\Zc^p(X,n) \lra \Zc^p(X,n-1).
\]
The permutation group $\Sn$ acts on $\square^n$ by permutation of the factors.
This action extends to an action of the semi-direct
product $G_n=(\Z / 2\Z)^n \rtimes \Sn$ where each $\Z/2\Z$ acts on $\square^1$ by
sending the usual affine coordinates $x$ to $1/x$.  The sign representation of
$\Sn$ extends to a sign representation $G_n \lra
\{\pm 1\}$. Let $\Alt_n \in \Q[G_n]$ be the corresponding projector;  when the
context is clear enough, we may drop
the subscript $n$.
\begin{defn} Let $p$ and $k$ be integers as above. One defines 
\[
\mc N_X^k(p)=\Alt_{2p-k}(\Zc^p(X,2p-k)\otimes \Q).
\]
We will refer to $k$ as the \emph{cohomological degree} and to $p$ as the \emph{weight}.
\end{defn}
For our purpose, we will not only need admissible cycles but cycles in $X
\times \square^n$ whose fibers over $X$ are also admissible. 

\begin{defn}[Equidimensionality]\label{def:equicycle}
Let $X$ be an irreducible smooth variety.
\begin{itemize}
\item Let $\Ze[p](X,n)$ denote the free abelian group
  generated by irreducible closed subvarieties $Z \subset X\times \square^n$ such
  that for any face 
  $F$ of $\square^n$, the intersection $Z\cap (X\times F) $ is empty or the
  restriction of $p_1 : X\times \square^n \lra X$ to  
\[
 Z\cap (X\times F) \lra X
\]
is equidimensional of relative dimension $\dim(F)-p$.
\item We say that elements of $\Ze[p](X,n)$ are \emph{equidimensional over $X$ with
    respect to any  face} or simply \emph{equidimensional}.
\item Following the definition of $\cNg{X}{k}(p)$, let $\Nge{X}{k}(p)$ denote 
\[
\Nge{X}{k}(p)=\Alt_{2p-k}\left(\Ze[p](X,2p-k )\otimes \Q\right).
\]
\item If $Z$ is an irreducible closed subvariety of $X \times \square^n$
  satisfying the above condition, $Z|_{t=x}$ will denote the fiber over the
  point $x \in X$ of $p_1$ restricted to $Z$ that is $Z\cap (\{x\}\times \square^n)$.  

Let $C=\Alt(\sum q_i Z_i)$ be an element in $\Nge{X}{\bullet}$ with the
  $Z_i$ as above and $q_i$'s in $\Q$. For a point $x \in X$, we will denote by
  $C|_{t=x}$ the   element of $\cNg{X}{\bullet}$ 
\[
C|_{t=x}=\Alt(\sum q_i Z_i|_{t=x})
\]
which is well defined in both $\cNg{X}{\bullet}$ and $\cNg{\{x\}}{\bullet}$ by
definition of the $Z_i$. 
\end{itemize}
\end{defn}
\begin{exm}\label{exmL0:A1} Consider the graph of the identity $\A^1 \st{t 
\mapsto t}{\lra} \A^1$
  restricted to $\A^1\times \A^1\sm\{ 1\}$. Let $\Gamma_0$ be its embedding in
  $\A^1 \times \square^1$. Then $\Gamma_0$ is of codimension $1$ in $\A^1
  \times \square^1$ and is 
admissible as the intersection with the face 
$x_1=\infty$ is empty and the intersection with the face $x_1=0$ is $\{0\}\times
\{0\}$ which is of codimension $1$ in $\A^1 \times \{0\}$.

However, $\Gamma_0$ is not equidimensional as 
\[
\Gamma_0 \cap\left( \A^1\times \{0\}\right)=\{0\}\times\{0\}
\]
is neither equidimensional over $\A^1$ nor empty as the condition would require.

Applying the projector $\Alt$ gives an element $\ol{\Lc_0}$ in
$\cNg{\A^1}{1}(1)$. Using the definition of $\Gamma_0$ as a graph, one obtains a
parametric representation (where the projector $\Alt$ is omitted):
\[
\ol{\Lc_0}=[t;t] \qquad \subset \A^1\times \square^1.
\]
In the above notation the semicolon separates the base space coordinates from
the cubical coordinates.
\end{exm}

The morphisms $s_i^{\ve\, *}$ induce morphisms $\dN_i^{\ve} : \cNg{X}{k}(p)
\lra \cNg{X}{k+1}(p)$ and the above differential
$\dN=\sum_i(-1)^{i-1}(\dN_i^0-\dN^{\infty}_i)$ gives a complex
\[
\mc N_X^{\bullet}(p) : \qquad \cdots \lra \mc N_X^{k}(p) \st{\dN}{\lra} \mc
N_X^{k+1}(p) \lra \cdots 
\]
\begin{defn}
One defines the cycle complex as 
\[
\mc N_X^{\bullet}=\bigoplus_{p \geqs 0}\mc N_X^{\bullet}(p)
=\Q \oplus \bigoplus_{p \geqs 1}\mc N_X^{\bullet}(p)
\]
and as the differential restricts to equidimensional cycles, one also defines
\[
 \Nge{X}{\bullet}=\bigoplus_{p\geqs 0}\Nge{X}{\bullet}(p).
\]
\end{defn}

The author refers sometimes to $\cNg{X}{\bullet}$ as the cycle algebra because of
another natural structure coming with this cubical cycle complex: the product structure.

Levine has shown in \cite{LevBHCG}[\S 5] or \cite{LEVTMFG}[Example 4.3.2]
 the following proposition.
\begin{prop} Concatenation of the cube factors and pull-back by the diagonal
\[
X\times \square^n \times X \times \square^m \st{\sim}{\ra} 
X\times X \times \square^n \times  \square^m \st{\sim}{\ra} 
X\times X \times \square^{n+m}\st{\Delta_X}{\longleftarrow}
 X \times \square^{n+m}
\]
induce, after applying the $\Alt$ projector, a well-defined product:
\[
\mc N_X^k(p)\otimes \mc N_X^l(q) \lra \mc N_X^{k+l}(p+q)  
\]
 denoted by $\cdot$

The complex $\Nge{X}{\bullet}$ is stable under this product law.
\end{prop}
\begin{rem} The smoothness hypothesis on $X$  allows us to consider the
  pull-back by the diagonal $\Delta_X : X \lra X \times X$ which is, in this case,
  of local complete intersection. 
\end{rem}
One has the following theorem (stated in \cite{BKMTM, BlochLMM} for
$X=\Sp({\Q})$).
\begin{thm}[{\cite{LevBHCG}}]
The cycle complex $\mc N_X^{\bullet}$ is an Adams graded, commutative differential graded 
algebra (Adams graded, c.d.g.a.). In weight $p$, its cohomology groups are the
higher Chow groups of $X$: 
\[
\HH^k(\mc N_X(p))=\CH^p(X,2p-k)_{\Q},
\]  
where $\CH^p(X,2p-k)_{\Q}$ stands for $\CH^p(X,2p-k)\otimes {\Q}$.
\end{thm} 
Moreover $\Nge{X}{\bullet}$ turns into a sub-Adams graded, c.d.g.a. Note that, in the graded algebra context, \emph{commutative}
 always means \emph{graded commutative}. 

One has natural flat pull-backs and proper push-forwards on $\cNg{X}{\bullet}$
(and on $\Nge{X}{\bullet}$). Comparison with higher Chow groups also gives on
the cohomology groups both $\A^1$-homotopy invariance and the long exact
sequence associated to an open and its closed complement. Writing $\ps$ as $\A^1
\sm \{ 0,1 \}$, one obtains the following description of
$\HH^*(\cNg{\ps}{\bullet}(p))$ :
\begin{multline*}
\HH^k(\mc N_{\ps}^{\bullet}(p))\simeq \HH^k(\mc N_{\Q}^{\bullet}(p))\oplus \\
\HH^{k-1}(\mc N_{\Q}^{\bullet}(p-1))\otimes \Q \Lc_0 
\oplus \HH^{k-1}(\mc N_{\Q}^{\bullet}(p-1))\otimes \Q \Lc_1,
\end{multline*}
where $\Lc_0$ and $\Lc_1$ are in cohomological degree $1$ and weight $1$ (that is of
codimension $1$). Their  explicit description will be given later on. 

Comparing the
situation over $\ps$ and over $\A^1$ comes as an important idea in our project
as the desired cycles over $\ps$ need to admit a
natural specialization at $1$.  In particular, we will need to work with
equidimensional cycle and some of their properties are given in the next
subsection.

\subsection{Equidimensional cycles}
The following result given in \cite{SouMPCC}
essentially follows from the definition and makes it easy to compare both situations.   
\begin{prop}\label{equiopen}
Let $X_0$ be an open dense subset of $X$ an irreducible smooth variety and let $j : X_0
\lra X$ be the inclusion. Then the restriction of cycles from $X$ to $X_0$
induces a morphism of  c.d.g.a. preserving the weight (that is the Adams grading)
\[
j^* : \Nge{X}{\bullet} \lra \Nge{X_0}{\bullet}.
\]

Moreover, let $C$ be in $\cNg{X_0}{\bullet}$ and  write $C$ in terms of
the generators of the group $\oplus\Zc^{\star}(X_0,\bullet)$ as
\[
C=\sum_{i\in I} q_i Z_i, \qquad q_i \in \Q
\] 
where $I$ is a finite set. 
Assume that,  for any $i$, the Zariski
closure $\ol{Z_i}$ of $Z_i$ in $X \times \square^{n_i}$ intersected with any
face $X\times F$ of
$X \times \square^{n_i}$ is equidimensional over $X$ of relative dimension
$\dim(F)-p_i$. Define $C'$ as 
\[
C'=\sum_{i\in I} q_i \ol{Z_i},
\]
then
\[
C' \in \Nge{X}{\bullet} \qquad \mx{and}\qquad  C=j^*(C')\in \Nge{X_0}{\bullet}.
\]
\end{prop}

Below, we
describe the main geometric fact that allows the construction of our cycles:
pulling back by the multiplication induces a homotopy between identity and the
zero section on the cycle algebra over $\A^1$.

Let $m : \A^1 \times \A^1 \lra \A^1$ be the multiplication map sending $(x,y)$
to $xy$ and let $\tau : \square^1=\p^1\sm\{1\}\lra \A^1$ be the isomorphism
sending the affine coordinate $u$ to $\frac{1}{1-u}$. The map $\tau$ sends
$\infty$ to $0$, $0$ to $1$ and extends as a map from $\p^1$ to $\p^1$ sending
$1$ to $\infty$.

The maps $m$ and $\tau$ are in particular flat and equidimensional of relative
dimension $1$ and $0$, respectively.

Consider the following commutative diagram for a positive integer $n$
\[
\begin{tikzpicture}
\matrix (m) [matrix of math nodes,
row sep=2.5em, column sep=9.5em, 
 text height=1.5ex, text depth=0.25ex]
{\A^1\times \square^1 \times \square^n & \A^1 \times\square^n \\
\A^1\times \square^1 & \A^1 \\
\A^1 & \\};
\path[->]
(m-1-1) edge node[mathsc,auto]{(m\circ(\id_{\A^1}\times \tau))\times \id_{\square^n}} (m-1-2)
(m-1-1) edge node[mathsc,left]{p_{\A^1\times \square^1}} (m-2-1)
(m-1-2) edge node[mathsc,right]{p_{\A^1}} (m-2-2)
(m-2-1) edge node[mathsc,auto]{m\circ(\id_{\A^1}\times \tau)} (m-2-2)
(m-2-1) edge node[mathsc,left]{p_{\A^1}} (m-3-1);
\end{tikzpicture}
\]

\begin{prop}[multiplication and equidimensionality]\label{multequi}
In the following statement,  $p$, $k$ and $n$ will denote positive integers
subject to the relation $n=2p-k$
\begin{itemize}
\item the composition $ \wt{m}=(m \circ (\id_{\A^1}\times \tau ))\times
  \id_{\square^n}$ induces  a group morphism
\[
\Ze[p](\A^1,n)\st{\wt{m}^*}{\lra} \Ze[p](\A^1\times \square^1,n)
\] 
which extends to a morphism of complexes for any $p$
\[
\Nge{\A^1}{\bullet}(p)\st{\wt{m}^*}{\lra} \Nge{\A^1\times\square^1}{\bullet}(p).
\]
\item Moreover, one has a natural group morphism
\[
h_{\A^1,n}^p : \Ze[p](\A^1\times \square^1,n) \lra  \Ze[p](\A^1,n+1)
\] 
given by regrouping the $\square^1$ factors (as $\square^n=(\square^1)^n$). 
\item The composition $\mu^* = h_{\A^1,n}^p \circ \wt{m}^*$ gives a
  linear map
\[
\mu^* : \Nge{\A^1}{k}(p) \lra \Nge{\A^1}{k-1}(p)
\]
sending equidimensional cycles with  empty fiber at $0$ to equidimensional cycles
with empty fiber at $0$.
\item Let $\theta : \A^1 \lra \A^1$ be the involution sending the natural affine
  coordinate $t$ to $1-t$. Twisting the map $\wt{m}$ by $\theta$ gives a map $\wh{m}$ via
\[
\begin{tikzpicture}
\matrix (m) [matrix of math nodes,
row sep=2.5em, column sep=2.5em, 
 text height=1.5ex, text depth=0.25ex]
{\A^1\times \square^1 \times \square^n & \A^1 \times\square^n \\
\A^1\times \square^1 \times \square^n & \A^1 \times\square^n \\};
\path[->,font=\scriptsize]
(m-2-1) edge node[mathsc,auto]{\wt{m}} (m-2-2)
(m-1-1) edge node[mathsc,left]{\theta \times\id _{\square^{n+1}}} (m-2-1)
(m-2-2) edge node[mathsc,right]{\theta \times\id _{\square^{n}}} (m-1-2);
\path[->,dashed, font=\scriptsize]
(m-1-1) edge node[mathsc,auto]{\wh m} (m-1-2);
\end{tikzpicture}
\]
and induces a linear map
\[
\nu^* : \Nge{\A^1}{k}(p) \lra \Nge{\A^1}{k-1}(p)
\]
sending equidimensional cycles with  empty fiber at $1$ to equidimensional cycles
with empty fiber at $1$.
\end{itemize}
\end{prop}
\begin{proof}
It is enough to work with generators of $\Ze[p](\A^1,n)$. Let $Z$ be an
irreducible subvariety of $\A^1\times \square^n$ such that for any face $F$ of
$\square^n$, the first projection 
\[
p_{\A^1} : Z \cap (\A^1 \times F) \lra \A^1
\]
is equidimensional of relative dimension $\dim(F)-p$ or empty. Let $F$ be a face
of 
$\square^n$. We  want first to show that 
under the projection $\A^1 \times \square^1 \times \square^n \lra \A^1 \times
\square^1$, 
\[
\wt{m}^{-1}(Z)\cap (\A^1\times \square^1 \times F) \lra \A^1 \times
\square^1 \]
 is equidimensional of relative dimension $\dim(F)-p$ or empty. This follows
 from the  
 fact that $Z \cap (\A^1 \times F)$ is equidimensional over $\A^1$ and $m$ is
 flat and equidimensional of relative 
 dimension $1$ (hence are $m\times \tau$ and $\wt{m}$). The map $\wt{m}$ is the
 identity on the $\square^n$ factor, 
 thus for $Z\subset \A^1\times \square^n$ as above and a codimension $1$
 face $F$ of $\square^n$, $\wt{m}^{-1}(Z)$ satisfies
\[
\wt{m}^{-1}(Z)\cap (\A^1 \times \square^1\times F)=\wt{m}^{-1}(Z\cap(\A^1 \times 
F))
\]
which makes $\wt{m}^*$ into a morphism of complexes.

Moreover, assuming  that the fiber of $Z$ at $0$ is empty,  as $\wt{m}$
restricted to  
\[
\{0\} \times \square^1 \times \square^n
\]
factors through the inclusion
$\{0\}\times \square^n \lra \A^1 \times \square^n$, the intersection 
\[
\wt{m}^{-1}(Z)\cap \left(
\{0\}\times \square^1 \times \square^n
\right)
\]
is empty. Hence the fiber of $\wt{m}^{-1}(Z)$ over $\{0\}\times \square^1$
 by $p_{\A^1\times \square^1}$ is empty and the same holds for the fiber over
 $\{0\}$ by $p_{\A^1}\circ p_{\A^1\times \square^1} $.

Now, let $Z$ be an irreducible subvariety of $\A^1\times\square^1\times
\square^n$ such that for any face $F$ of $\square^n$
\[
Z\cap (\A^1\times \square^1 \times F) \lra \A^1 \times
\square^1 \]
is equidimensional of relative dimension $\dim(F)-p$. Let $F'$ be a face of 
\[
\square^{n+1}=\square^1\times \square^n.
\]
The face $F'$ is either of the form $\square^1\times F$ or of the form
$\{\ve\}\times F$ with $F$ a face of $\square^n$  and $\ve \in \{0,\infty\}$. If
$F'$ is of the first type, as 
\[
Z\cap (\A^1\times \square^1 \times F) \lra \A^1 \times \square^1
\] 
is equidimensional and, as  $\A^1
\times \square^1 \lra \A^1$  is equidimensional of relative dimension $1$, the
projection 
\[
Z\cap (\A^1\times \square^1 \times F) \lra \A^1
\]
is equidimensional of relative dimension
\[
\dim(F)-p+1=\dim(F')-p.
\]

If $F'$ is of the second type, by symmetry of the role of $0$ and $\infty$, we can
assume that $\ve=0$. Then the intersection 
\[
Z\cap (\A^1\times \{0\} \times F)
\] 
is nothing but the fiber of $Z\cap (\A^1\times \square^1 \times F)$
over $\A^1\times \{0\}$. Hence, it has pure dimension $\dim(F)-p+1$.

Moreover, denoting with a subscript the fiber, the composition 
\[
Z\cap (\A^1\times \{0\} \times F)=
\left(Z\cap(\A^1\times \square^1 \times F)\right)_{\A^1\times\{0\}} \lra 
\A^1\times \{0\} \lra \A^1
\]
is equidimensional 
of relative dimension  
\[
\dim(F)-p=\dim(F')-p.
\] 
This shows that $h_{\A^1,n}^p$ gives a well defined morphism and that it 
preserves the fiber at a point $x$ in $\A^1$; in particular, if $Z$ has an empty
fiber at $0$, so does $h_{\A^1,n}^p (Z)$. 

Finally, the last part of the proposition is deduced from the fact that
$\theta$ exchanges the role of $0$ and $1$.
\end{proof}
\begin{rem}\label{emptyinfty} We have remarked that $\wt{m}$ sends cycles with
  empty fiber at $0$ 
  to cycles with empty fiber at any point in $\{0\}\times \square^1$. Similarly
  $\wt{m}$ sends cycles with empty fiber at $0$ to cycles that also
  have an empty fiber at any point in $\A^1\times \{\infty\}$.
\end{rem}
From the proof of Levine's Proposition 4.2 in \cite{LevBHCG}, we deduce that
$\mu^*$ gives a homotopy between $p_{0}^*\circ i_0^*$ and $ 
\id  $ where $i_0$ is the zero section $\{0\} \ra \A^1$ and $p_0$ the projection
onto the point $\{0\}$.
\begin{prop}\label{multhomo}Notations are the ones from Proposition
  \ref{multequi} above. Let $i_0$ (resp. $i_1$) be  the inclusion of $0$ (resp. $1$)
  in $\A^1$:
\[
i_0 : \{0\} \lra \A^1\, , \qquad  \quad i_{1} : \{1\} \lra \A^1.
\]  
Let $p_0$ and $p_1$ be the corresponding projections $p_{\ve} : \A^1 \lra
\{\ve\}$ for $\ve=0,1$.

Then $\mu^*$ provides a homotopy between 
\[
p_{0}^*\circ i_0^* \mx{ and } 
\id : \Nge{\A^1}{\bullet} \lra \Nge{\A^1}{\bullet}
\] 
and similarly $\nu^*$ provides a homotopy between
\[
p_{1}^*\circ i_1^* \mx{ and } 
\id : \Nge{\A^1}{\bullet} \lra \Nge{\A^1}{\bullet}.
\] 
In other words, one has 
\[
\da[\A^1]\circ \mu^* + \mu^*\circ \da[\A^1] = 
\id - p_{0}^*\circ i_0^* 
\quad \mx{and} \quad 
\da[\A^1]\circ \nu^* + \nu^*\circ \da[\A^1] = 
\id-  p_{1}^*\circ i_1^*.
\]
\end{prop}
The proposition follows from computing the different compositions
  involved and the relation between the differential on $\Nge{\A^1\times
    \square^1}{\bullet}$ and the one on $\Nge{\A^1}{\bullet}$ via the map
  $h_{\A^1,n}^p$.

\begin{proof} 
We denote by $i_{0,\square}$ and $i_{\infty, \square}$ the zero section and
the infinity section $\A^1 \lra \A^1 \times \square^1$.
The action of $\theta$ only exchanges the role of $0$ and $1$ in $\A^1$, hence it
is enough to prove the statement for $\mu^*$.
As previously, in order to obtain the proposition for $\Nge{\A^1}{k}(p)$, it is
enough to work on the generators of $\Ze[p](\A^1,n)$ with $n=2p-k$.

By the previous proposition \ref{multequi}, the morphism $\wt{m}^*$ commutes with the
 differential on $\Ze[p](\A^1,\bullet)$ and on $\Ze[p](\A^1\times \square^1,
 \bullet)$.  
 As the morphism $\mu^*$ is defined by  $\mu^*=h_{\A^1,n}^p\circ \wt{m}^*$, the
 proof relies on computing $\da[\A^1]  \circ h_{\A^1, n}^p$. Let $Z$ be
 a generator of $\Ze[p](\A^1\times \square^1,n)$. In particular,
\[
Z \subset \A^1\times \square^1 \times \square^n
\]  
and $h_{\A^ 1,n}^p(Z)$ is also given by $Z$ but viewed in 
\[
\A^1\times \square^{n+1}.
\] 
The differentials denoted by $\da[\A^1]^{n+1}$  on
$\Ze[p](\A^1,n+1)$ and  $\da[\A^1 \times
\square^1]^n$ on $\Ze[p](\A^1\times \square^1, n)$ 
are both given by intersections with the codimension $1$ faces but the first
$\square^1$ factor in $\square^{n+1}$ gives  two more faces and introduces a
change of sign. Namely, using an extra subscript to indicate in which cycle
groups the intersections take place, one has:
\begin{align*}
\da[\A^1]^{n+1}(h_{\A^1,n}^p(Z))=&
\sum_{i=1}^{n+1} (-1)^{i-1}\left(\dN_{i,\A^1}^0(Z)-\dN_{i,\A^1}^{\infty}(Z)\right) \\ 
=&\dN_{1,\A^1}^0(Z)-\dN_{1, \A^1}^{\infty}(Z)
-\sum_{i=2}^{n+1}(-1)^{i-2} \left(\dN_{i,\A^1}^0(Z)-\dN_{i,\A^1}^{\infty}(Z)\right) \\
=&i_{0,\square}^*(Z)-i_{\infty, \square }^*(Z) 
-\sum_{i=1}^{n}(-1)^{i-1} 
\left(\dN_{i+1,\A^1}^0(Z)-\dN_{i+1,\A^1}^{\infty}(Z)\right).
\end{align*} 
Hence one gets
\begin{multline*}
\da[\A^1]^{n+1}(h_{\A^1,n}^p(Z))=i_{0,\square}^*(Z)-i_{\infty, \square }^*(Z) \\
-\sum_{i=1}^{n}(-1)^{i-1} 
\left(h_{\A^1, n-1}^p \left( \dN_{i,\A^1\times \square^1}^0(Z)-
 \dN_{i,\A^1\times \square^1}^{\infty}(Z)\right)\right)
\end{multline*} 
which can be written has
\[
\da[\A^1]^{n+1}(h_{\A^1,n}^p(Z))
=i_{0,\square}^*(Z)-i_{\infty, \square }^*(Z) - 
h_{\A^1, n-1}^p\circ\da[\A^1 \times \square^1]^n(Z).
\] 
Thus one can compute 
 $
\da[\A^1]\circ \mu^* +\mu^*\circ \da[\A^1] 
$ on $\Ze[p](\A^1, n) $ as
\begin{align*}
\da[\A^1]\circ \mu^*+\mu^*\circ \da[\A^1] =&
\da[\A^1]\circ h_{\A^1,n} \circ \wt{m}^*+
h_{\A^1,n-1} \circ \wt{m}^* \circ \da[\A^1] \\
=&i_{0,\square}^*\circ \wt{m}^*-i_{\infty, \square }^*\circ \wt{m}^*
-h_{\A^1,n-1} \circ\da[\A^1]\circ \wt{m}^* \\
&\hskip 30ex +h_{\A^1,n-1} \circ \da[\A^1] \circ \wt{m}^* \\
=&i_{0,\square}^*\circ \wt{m}^*-i_{\infty, \square }^*\circ \wt{m}^*.
\end{align*}

The morphism $i_{\infty,\square}^*\circ \wt{m}^*$ is induced by
\[
\begin{tikzpicture}
\matrix (m) [matrix of math nodes,
 row sep=0.6em, column sep=2.5em, 
 text height=1.5ex, text depth=0.25ex] 
{\A^1 & \A^1 \times \square^1 & \A^1 \times \A^1 & \A^1 \\
x & (x, \infty) & (x, 0) & 0 \\
};
\path[->]
(m-1-1) edge node[mathsc,auto] {i_{\infty,\square}} (m-1-2)
(m-1-2) edge node[mathsc,auto] {\tau} (m-1-3)
(m-1-3) edge node[mathsc,auto] {m} (m-1-4);
\path[|->]
(m-2-1) edge node[mathsc,auto] {} (m-2-2)
(m-2-2) edge node[mathsc,auto] {} (m-2-3)
(m-2-3) edge node[mathsc,auto] {} (m-2-4);
\end{tikzpicture}
\]
which factors through
\[
\begin{tikzpicture}
\matrix (m) [matrix of math nodes,
 row sep=1.5em, column sep=2.5em, 
 text height=1.5ex, text depth=0.25ex] 
{\A^1 & \A^1 \times \square^1 & \A^1 \times \A^1 & \A^1 \\
\{0\} &  &  & \A^1 \\
};
\path[->]
(m-1-1) edge node[mathsc,auto] {i_{\infty,\square}} (m-1-2)
(m-1-2) edge node[mathsc,auto] {\tau} (m-1-3)
(m-1-3) edge node[mathsc,auto] {m} (m-1-4)
(m-1-1) edge node[mathsc,auto] {p_0} (m-2-1)
(m-2-1) edge node[mathsc,auto] {i_0} (m-2-4)
(m-1-4) edge node[mathsc,auto] {\id_{\A^1}} (m-2-4);
\end{tikzpicture}
\]
Thus, 
\[
i_{\infty,\square}^*\circ \wt{m}^*=(i_0\circ p_0)^*=p_0^*\circ i_0^*.
\]
 Similarly  $i_{0, \square }^*\circ \wt{m}^*$ is induced by 
\[
\begin{tikzpicture}
\matrix (m) [matrix of math nodes,
 row sep=0.6em, column sep=2.5em, 
 text height=1.5ex, text depth=0.25ex] 
{\A^1 & \A^1 \times \square^1 & \A^1 \times \A^1 & \A^1 \\
x & (x, 0) & (x, 1) & x \\
};
\path[->]
(m-1-1) edge node[mathsc,auto] {i_{\infty,\square}} (m-1-2)
(m-1-2) edge node[mathsc,auto] {\tau} (m-1-3)
(m-1-3) edge node[mathsc,auto] {m} (m-1-4);
\path[|->]
(m-2-1) edge node[mathsc,auto] {} (m-2-2)
(m-2-2) edge node[mathsc,auto] {} (m-2-3)
(m-2-3) edge node[mathsc,auto] {} (m-2-4);
\end{tikzpicture}
\]
which factors through $\id_{\A^1} : \A^1 \lra \A^1$ and one has 
\[
i_{0, \square }^*\circ \wt{m}^*=\id
\]
which concludes the proof of the proposition.
\end{proof}
 
\subsection{Weight $1$, weight $2$ and polylogarithm cycles}
\label{subsec:weight12Lin}
For now on, we set  $X=\ps$. 
\subsubsection{Two weight $1$  cycles generating the $\HH^1$}
As mentioned before, there is a decomposition of
$\HH^1(\cNg{X}{\bullet}(p))$ as 
\[
\HH^1(\cNg{X}{\bullet}(p))\simeq \HH^1(\cNg{\Q}{\bullet}(p))\oplus
\HH^{0}(\cNg{\Q}{\bullet}(p-1))\otimes \Q \Lc_0 
\oplus \HH^{0}(\cNg{\Q}{\bullet}(p-1))\otimes \Q \Lc_1
\] 
and $\Lc_0$ and $\Lc_1$ (which are in weight $1$ and degree $1$)
  generates the $\HH^*(\cNg{X}{\bullet})$ relatively to
  $\HH^*(\cNg{\Q}{\bullet})$. Explicit expression for $\Lc_0$ and
  $\Lc_1$ are given below.

In Example \ref{exmL0:A1}, a cycle $\ol{\Lc_0}$ was constructed using the graph
of $t \longmapsto t$ from $\A^1 \lra \A^1$. Taking its restriction to $X \times
\square^1$, and using the same convention, one gets a cycle
\begin{equation}\label{exmL0:X}
\Lc_0 = [t; t] \subset X \times \square^1, \qquad \Lc_0\in \cNg{X}{1}(1).
\end{equation}
Similarly, using the graph of $t \longmapsto 1-t$, one gets 
\begin{equation} \label{exmL1:X}
\Lc_1 = [t; 1-t] \subset X \times \square^1, \qquad \Lc_1\in \cNg{X}{1}(1).
\end{equation}
One notices that the cycles $\Lc_0$ and $\Lc_1$ are both equidimensional over $X=\ps$
but not equidimensional over $\A^1$.

Moreover, as 
\[
\Lc_0 \cap (X\times\{\ve\})=\Lc_0 \cap (\ps \times\{\ve\})=\emptyset
\]
for $\ve=0, \infty$, the
above intersection tells us that $\dN(\Lc_0)=0$. Similarly, one shows that
$\dN(\Lc_1)=0$. Thus $\Lc_0$ and $\Lc_1$ give two well defined classes in
$\HH^1(\cN[\bullet](1))$. 

In order to show that they are non-trivial and that they give the above decomposition of
the $\HH^1(\cNg{X}{1})$, one shows that, in the localization
sequence, their images under the boundary map 
\[
\HH^1(\Nc (1)) \st{\delta}{\lra} 
\HH^0(\Nc[\{0\}](0))\oplus  \HH^0(\Nc[\{1\}](0))
\] 
are non-zero. It is enough to treat the
case of $\Lc_0$. 
Recall that $\ol{\Lc_0}$ is the closure of $\Lc_0$ in $\A^1\times \square^1$ and  is
given by the parametrized cycle 
\[
\ol{\Lc_0}=[t; t] \subset \A^1\times \square^1.
\]
Its intersection with the face $u_1=0$ is of codimension $1$ in
$\A^1\times\{0\}$ and the intersection with $u_1=\infty$ is empty. Hence
$\ol{\Lc_0}$ is admissible.

Thus, considering the definition of $\delta$, $\delta(\Lc_0)$ is given by the
intersection of the differential 
of $\ol{\Lc_0}$ with $\{0\}$ and $\{1\}$ on the first and second
factor, respectively. The above discussion on the admissibility of $\ol{\Lc_0}$
tells us that 
$\delta(\Lc_0)$ is non-zero on the factor $\HH^0(\Nc[\{0\}](0)) $ and $0$ on the
other factor as the admissibility condition is trivial for $\Nc[\{0\}](0)
$ and the restriction of $\ol{\Lc_0}$ to $1$ is empty. The situation is reverse
for $\Lc_1$ using its closure $\ol{\Lc_1}$ in $\A^1\times \square^1$. 

Hence, even if  the differentials of $\Lc_0$ and $\Lc_1$ are $0$ in
$\cNg{X}{\bullet}$, the differentials of their closure in $\A^1$ are non-zero in
$\cNg{\A^1}{\bullet}$ and 
have a 
particular behavior when multiplied by an equidimensional cycle (see Lemma
\ref{dA1-L0L1} below and Equation \eqref{dA1-L0L01} for an example). We consider here
only equidimensional cycles as it is needed to work with such cycles in order to
pull-back by the multiplication. We use below notations of propositions
\ref{multequi} and \ref{multhomo}.
\begin{lem}\label{dA1-L0L1}
Let $C$ be an element in $\Nge{\A^1}{\bullet}$, then
\[
\da[\A^1](\ol{\Lc_0}) \, C =C|_{t=0}
\qquad \mx{and} \qquad 
\da[\A^1](\ol{\Lc_1})\,C=C|_{t=1} 
\]
where the notation $C|_{t=0}$ (resp. $C|_{t=1}$) denotes, as in Definition
\ref{def:equicycle}  the (image under the projector $\Alt$ of the) fiber at $0$
 (resp. $1$) of the  irreducible closed subvarieties composing the formal sum
 that defines $C$.
\end{lem}
\begin{proof}
It is enough to assume that $C$ is given by  $C=\Alt(Z)$ where $Z$ is  an 
irreducible closed subvariety of $ \A^1 \times \square^n$ such
  that for any face
  $F$ of $\square^n$, the intersection $Z\cap (X\times F) $ is empty or the
  restriction of $p_1 : \A^1 \times \square^n \lra \A^1$ to  
\[
 Z\cap (\A^1\times F) \lra \A^1
\]
is equidimensional of relative dimension $\dim(F)-p$.

Remark that for $\ve=0,1$ the cycle  $\da[\A^1](\ol{\Lc_{\ve}})$ is given by the
point
\[
\{\ve\} \in \A^1
\]
which is of codimension $1$ in $\A^1$. In order to compute the product
$\da[\A^1](\ol{\Lc_{\ve}})\,C$, one considers first the product in $\A^1 \times
\A^1 \times \square^n$:
\[
\{\ve\}\times Z \subset \A^1 \times \A^1 \times \square^n.
\]  
Let $\Delta$ denote the image of the diagonal $\A^1 \lra \A^1 \times \A^1$. The
equidimensionality of $Z$ insures that for any face $F$ of $\square^n$
\[
\left(\{\ve\}\times Z \right)\bigcap
\left( \Delta \times F\right)
\simeq \left(Z \cap (\{\ve\}\times \square^n)\right) \bigcap
\left(\A^1\times F \right)
\]
is of codimension $p+1$. Thus the product $\da[\A^1](\ol{\Lc_{\ve}})\,C$ is simply
the image under $\Alt$ of 
\[
Z\cap (\{\ve\}\times \square^n)=Z|_{t=\ve} \subset \A^1 \times \square^n.
\]
\end{proof}
\subsubsection{A weight $2$ example: the Totaro cycle}
One considers the linear combination
\[
b=\Lc_{0}\cdot \Lc_{1} \in \cN[2](2).
\]
It is given as a parametrized cycle by 
\[
b=[t; t, 1-t] \subset X \times
\square^2
\]
or in terms of defining equations by
\[
T_1V_1-U_1T_2= 0 \qquad \mbox{and} \qquad U_1V_2+U_2V_1=V_1V_2
\]
where $T_1$ and $T_2$ denote the homogeneous coordinates on $X=\ps$ and $U_i$, $V_i$
the homogeneous coordinates on each factor $\square^1=\p^1\sm \{1\}$ of $\square^2$.
One sees that the intersection of $b$ with faces $U_i$ or $V_i=0$ for 
$i=1,2$ is empty because $T_1$ and $T_2$ are different from $0$ in $X$ and
because $U_i$ is different from $V_i$ in $\square^1$. Thus it tells us  that 
\[
\dN(b)=0.
\] 

Now, let $\ol b$ denote the algebraic closure of $b$ in $\A^1\times \square^2$. 
 As previously, its expression as parametrized cycle is 
\[
\ol{b}=\ol{\Lc_0}\, \ol{\Lc_1}=[t; t, 1-t] \subset \A^1 \times
\square^2
\]
and the intersection with
$\A^1\times F$ for any codimension $1$ face  $F$ of $\square^2$ is empty. 
Writing, as before, $\dN_{\A^1}$ for the differential in $\mc N_{\A^1}$, one has
$\dN_{\A^1}(\ol b)=0$.

As $\ol{\Lc_0}$ (resp. $\ol{\Lc_1}$) is equidimensional over $\A^1 \sm \{0\}$
(resp. over $\A^1 \sm \{1\}$), the cycle $\ol{b}$ is equidimensional over
$\A^1\sm\{0,1\}$. Moreover, as $\ol{\Lc_0}$ (resp. $\ol{\Lc_1}$) has an empty
fiber at $1$ (resp. at $0$), $\ol{b}$ has empty fiber at both $0$ and $1$.
So $\ol{b }$ is equidimensional over $\A^1$ with empty fibers at $0$ and $1$. Following
notations of Proposition \ref{multhomo}, one defines two elements in $\cNg{\A^1}
1 (2)$ by
 pull back by the multiplication ( resp. twisted multiplication): 
\begin{equation} \label{exmL01:A1}
\ol{\Lc_{01}}=\mu^*(\ol{b}) \qquad \mx{and} \qquad 
\ol{\Lcu_{01}}=\nu^*(\ol{b}).
\end{equation}
One also defines their restrictions to $X$
\begin{equation} \label{exmL01:X}
\Lc_{01}=j^*(\ol{\Lc_{01}}) \qquad \mx{and} \qquad
\Lcu_{01}=j^*(\ol{\Lcu_{01}}).
\end{equation}

Now, direct application of Proposition \ref{multhomo} shows that
\[
\dN_{\A^1}(\ol{\Lc_{01}})=-\mu^*(\da[\A^1](\ol b))+\ol b -p_0^*\circ i_0^*(b) 
=-0+\ol{\Lc_0} \, \ol{\Lc_1}-0
\]
because $\ol b$ has empty fiber at $0$ and is $0$ under $\da[\A^1]$. More
generally, as $j^*$ is a morphism of c.d.g.a., Proposition \ref{multhomo} gives
the following.  
\begin{lem} Cycles $\Lc_{01}$, $\ol{\Lc_{01}}$, $\Lcu_{01}$ and
  $\ol{\Lcu_{01}}$ satisfy the following properties 
\begin{enumerate}
\item $\Lc_{01}$ and  $\Lcu_{01}$ (resp. $\ol{\Lc_{01}}$ and $\ol{\Lcu_{01}}$)
  are equidimensional over $X$, that is elements in $\Nge{X} 1 (2)$
  (resp. equidimensional over
  $\A^1$).
\item They satisfy the following differential equations 
\[
\dN(\Lc_{01})=\dN(\Lcu_{01})=b=\Lc_0 \, \Lc_1
\]
and $\da[\A^1](\ol{\Lc_{01}})=\da[\A^1](\ol{\Lcu_{01}})=\ol{b}=\ol{\Lc_0}\, \ol{\Lc_1}$.
\item By the definition given in Equation
  \eqref{exmL01:X}, the cycle $\ol{\Lc_{01}}$ (resp. $\ol{\Lcu_{01}}$) extends $\Lc_{01}$
  (resp. $\Lcu_{01}$)  over $\A^1$  and  has an empty   fiber at $0$ (resp. at $1$). 
\end{enumerate}
\end{lem}
Moreover, one can explicitly compute the two pull-backs and obtain parametric
representations 
\begin{equation} \label{L01:para}
\Lc_{01} =[t; 1-\frac{t}{x}, x, 1-x], \qquad \Lcu_{01} =[t; \frac{x-t}{x-1}, x, 1-x].
\end{equation}
The multiplication map inducing $\mu^*$ is given by

\[ \begin{tikzpicture}[baseline={([yshift=0.5em] current bounding box.south)}]
 \matrix (m) [matrix of math nodes,
 row sep=0.6em, column sep=2em, 
text height=1.5ex, text depth=0.25ex] 
 {\A^1 \times \square^1 \times \square^{2} 
& A^1 \times  \square^{2}  \\
};
\path[->,font=\scriptsize]
(m-1-1) edge node[auto] {} (m-1-2);
\end{tikzpicture}\!,
\quad
\begin{tikzpicture}[baseline={([yshift=0.5em] current bounding box.south)}]
 \matrix (m) [matrix of math nodes,
 row sep=0.6em, column sep=2em, 
text height=1.5ex, text depth=0.25ex] 
 {
{[t ; u_1, u_2, u_{3}]}  &
{[\frac{t}{1-u_1} ; u_2, u_{3}]}   \\
 };
\path[|->,font=\scriptsize]
(m-1-1) edge node[auto] {$$} (m-1-2);
 \end{tikzpicture}. \]
In order to compute the pull-back, one should remark that if $u=1-t/x$ then 
\[
\frac{t}{1-u} = x.
\]
Computing the pull-back by $\mu^*$ is then just rescaling the new $\square^1$
factor which arrives in first position. The case of $\nu^*$ is similar but using
the fact that for $u=\frac{x-t}{x-1}$ one has
\[
\frac{t-u}{1-u}=x.
\]
\begin{rem}
The cycle $\Lc_{01}$ is nothing but Totaro's cycle \cite{Totaro}, already described in
\cite{BKMTM, BlochLie}.

Moreover, $\Lc_{01}$ corresponds to the function $t \mapsto Li^{\C}_2(t)$ as shown in
\cite{BKMTM}.

One recovers the value $\zeta(2)$ by specializing at $t=1$ using the extension
of $\Lc_{01}$ to $\A^1$.
\end{rem}

\subsubsection{Polylogarithm cycles}
By induction one can build cycles $\on{Li}^{cy}_n=\Lc_{0\cdots 01}$ ($n-1$
zeros and one $1$). We define
$\on{Li}^{cy}_1$ to be equal to $\Lc_1$.
\begin{lem}\label{lem:cycleLi_k}
For any integer $n \geqs 2$ there exists an equidimensional cycle  over $X$,
$\on{Li}^{cy}_{n}$ in $\Nge{X}{1}(n)\subset \cNg{X}{1}(n)$ satisfying
\begin{enumerate}
\item There is an equidimensional cycle over $\A^1$, $\ol{\on{Li}^{cy}_n}$ in
    $\Nge{\A^1}{1}(n)$, such that $\on{Li}^{cy}_{n}=j^*(\ol{\on{Li}^{cy}_n})$ (it has
  in particular a well defined fiber at $1$).
\item The cycle $\ol{\on{Li}^{cy}_{n}}$ has empty fiber at $0$.  
\item The cycles $\on{Li}^{cy}_{n}$ and $\ol{\on{Li}^{cy}_{n}}$ satisfy the
  differential equations 
\[
\dN(\on{Li}^{cy}_{n})=\Lc_0\cdot \on{Li}^{cy}_{n-1}
\qquad \mx{and} \qquad
\dN(\ol{\on{Li}^{cy}_{n}})=\ol{\Lc_0}\cdot \ol{\on{Li}^{cy}_{n-1}}.
\]
\item $\on{Li}^{cy}_{n}$ is explicitly given as a parametrized cycle by
\[
[t; 1-\frac{t}{x_{n-1}}, x_{n-1},1-\frac{x_{n-1}}{x_{n-2}}, x_{n-2}, 
\ldots, 1-\frac{x_2}{x_{1}}, x_1, 1-x_1 ] 
\subset X \times \square^{2n-1}.
\] 
\end{enumerate}
\end{lem}
\begin{proof}
For $n=2$, we have already defined $\on{Li}^{cy}_2=\Lc_{01}$ satisfying the
expected properties.

Assume that one has built the cycles $\on{Li}^{cy}_k$ for $2\leqs k< n$.
One considers in $\cNg{\A^1}{2}(n)$ the product
\[
\ol b=\ol{\Lc_{0}}\cdot \ol{\on{Li}^{cy}_{n-1}}=[t; t,  1-\frac{t}{x_{n-2}},
x_{n-2},1-\frac{x_{n-2}}{x_{n-3}}, x_{n-3},  
\ldots, 1-\frac{x_2}{x_{1}}, x_1, 1-x_1].
\]
As $\ol{\Lc_{0}}$ is equidimensional over $\A^1\sm \{0\}$ and
as $\ol{\on{Li}^{cy}_{n-1}}$ is equidimensional over $\A^1$,  $\ol b$ is
equidimensional over $\A^1\sm\{0\}$. Moreover, as  $\ol{\on{Li}^{cy}_{n-1}}$ has empty
fiber at $0$, $\ol b$ is equidimensional over $\A^1$ with empty fiber at $0$.

Computing the differential with the Leibniz rule and Lemma \ref{dA1-L0L1}, one gets 
\[
\dN_{\A^1}(\ol{b})=\ol{\on{Li}^{cy}_{n-1}}|_{t=0}
-\ol{\Lc_{0}}\cdot \ol{\Lc_{0}}\cdot \ol{\on{Li}^{cy}_{n-2}}=0.
\] 
One concludes using Proposition \ref{multhomo}. The same argument used to obtain
the parametrized representation for $\Lc_{01}$ at Equation \eqref{L01:para}
shows that 
\[
\on{Li}^{cy}_n=[t; 1-\frac{t}{x_{n-1}},x_{n-1},1- \frac{x_{n-1}}{x_{n-2}}, x_{n-2},\ldots, 1
-\frac{x_2}{x_1} ,x_1,1-x_1 ] \subset \A^1 \times \square^{2n-1}.
\]
\end{proof}
\begin{rem}~

\begin{itemize}
\item One retrieves the expression given in \cite{BKMTM}.
\item Moreover, $\on{Li}^{cy}_n$ corresponds to the function $t \mapsto
  Li_n^{\C}(t)$ as shown in 
\cite{BKMTM} (or in \cite{GanGonLev05}).
\item $\ol{\Lc_0}$ having an empty fiber at $1$, one can also pull-back by the
  twisted multiplication and obtain similarly cycles $\Lcu_{0\cdots 01}$
  satisfying $\dN(\Lcu_{0\cdots 01})=\dN(\Lc_{0\cdots 01})$. In some sense, 
  they correspond to $\Lc_{0\cdots 01}-p^*\circ i_1^*(\Lc_{0\cdots
    01})$ which in terms of integrals corresponds to $Li^{\C}_n(t)-\zeta(n)$. 
\end{itemize}
\end{rem}

\subsection{Some higher weight examples for multiple polylogarithm cycles}
\subsubsection{Weight $3$}

The cycle $\Lc_{01}$ was defined previously, so was the cycle
$\Lc_{001}=\on{Li}^{cy}_3$ by considering the product
\[
b=\Lc_{0}\cdot \Lc_{01}.
\]

Now, in weight $3$, one could also consider the product
\begin{equation}\label{introL01(1)}
\Lc_{01}\cdot \Lc_{1} \quad \in \cN[2](3).
 \end{equation}
However the above product does not lead by similar arguments to a new
cycle. Before explaining how to follow the strategy used in weight $2$ and for
the polylogarithms in order to obtain another weight $3$ cycle, the author would
like to spend a little time on the obstruction occurring with the product in
Equation \eqref{introL01(1)} 
as it enlightens in particular the need of the cycle $\Lcu_{01}$ previously built.  

Thus let $b=\Lc_{01}\cdot \Lc_{1}$ be the above product in $\cN[2](3)$, 
 given as a parametrized cycle by
\[
b=[t;1-\frac{t}{x_1}, x_1, 1-x_1, 1-t] \quad \subset X\times \square^4.
\]
From this expression, one sees that $b$ is admissible and that  $\dN(b)=0$ 
because $t\in X$ can not be equal to $1$.

Let $\ol b$ be the closure of the defining cycle of $b$ in $\A^1 \times
\square^4$, that is  the image under the projector $\Alt$ of 
\[
\left\{(t,1-\frac{t}{x_1}, x_1, 1-x_1, 1-t) \mbox{ such that } t \in \A^1, \,
x_1 \in \p^1 \right\}\cap \A^1 \times \square^4.
\]
Let $u_i$ denote the coordinate on the $i$-th factor $\square^1$.
As most of the intersections of $\ol b$ with face $\A^1 \times F$ are empty,  in
order to prove  that $\ol b$ is admissible and gives an element in
$\cNg{\A^1}{2}(3)$, it is enough to check the (co)dimension condition on the three 
faces : $u_1=0$, $u_4=0$ and $u_1=u_4=0$. The intersection of $\ol b$ with the
face $u_1=u_4=0$ is empty as $u_2\neq 1$. The intersection $\ol b$ with the face defined
by $u_1=0$ or $u_4=0$ is $1$ dimensional and so of codimension $3$ in
$\A^1\times F$.

 Computing the differential in $\cNg{\A^1}{\bullet}$, using Lemma \ref{dA1-L0L1} or
 the fact that
the intersection with $u_1=0$ is killed by the projector $\Alt$, gives
\begin{equation}\label{dA1-L0L01}
\dN_{\A^1}(\ol b)=\dN_{\A^1}( \ol{\Lc_{01}} \ol{\Lc_{1}})=-\ol{\Lc_{01}}|_{t=1}\neq 0
\end{equation}
and the homotopy trick used previously will not work as it relies (partly) on
beginning with a cycle $\ol b$ satisfying $\da[\A^1](\ol b)=0$.

In order to bypass this, one could introduce the constant cycle
$ \Lc_{01}(1)=p^*\circ i_1^*(\ol{\Lc_{01}})$ and  consider the linear combination
\begin{equation}\label{bL011}
\ol b=(\ol{\Lc_{01}}-\ol{\Lc_{01}}(1))\cdot \ol{\Lc_{1}} \in  \cNg{\A^1}{2}(3).
\end{equation}
and its equivalent in $\cNg{X}{2}(3)$. Now, the correction by
$-\ol{\Lc_{01}}(1)\cdot \ol{ \Lc_1}$ 
insures that $ \dN_{\A^1}(\ol b)=0$.

However, it is still not good enough as the use of the homotopy property for the
pull-back by the multiplication requires to work with equidimensional cycles
which is not the case for $\ol b$ (the problem comes from the fiber at $1$).

The fact that $\ol{\Lc_1}$ is not equidimensional over $\A^1$ but equidimensional on
$\A^1 \sm\{1\}$ requires to multiply it by a cycle with an empty fiber at
$1$ which insures that the fiber of the product at $1$ is empty. 
Thus one considers the product in $\Nge{\A^1}{2}(3)$ 
\[
\ol b =\ol{\Lcu_{01}}\,\ol{\Lc_1}=-\ol{\Lc_1}\, \ol{\Lcu_{01}}
\]
which has an empty fiber at $0$ and $1$. Moreover the Leibniz rule and Lemma
\ref{dA1-L0L1} imply that 
\[
\da[\A^1](\ol b)=\da[\A^1](\ol{\Lcu_{01}})\, \ol{\Lc_1}
- \ol{\Lcu_{01}}\,\da(\ol{\Lc_1})
=\ol{\Lc_1}\, \ol{\Lc_0} \ol{\Lc_1} -\ol{\Lcu_{01}}|_{t=1}=0.
\]
Thus one defines 
\begin{equation}\label{exmL011:A1}
\ol{\Lc_{011}}=\mu^*(\ol{\Lcu_{01}}\,\ol{\Lc_1})
\qquad \mx{and} \qquad 
\ol{\Lcu_{011}}=\nu^*(\ol{\Lcu_{01}}\,\ol{\Lc_1})
\end{equation}
and their restrictions to $X = \ps$ 
\begin{equation}\label{exmL011:X}
\Lc_{011}=j^*(\ol{\Lc_{011}})
\qquad \mx{and} \qquad 
\Lcu_{011}=j^*(\ol{\Lcu_{011}}).
\end{equation}
As previously, propositions \ref{multhomo} and \ref{multequi} insure the following.
\begin{lem} The cycles $\Lc_{011}$, $\ol{\Lc_{011}}$, $\Lcu_{011}$ and
  $\ol{\Lcu_{011}}$ satisfy the following properties 
\begin{enumerate}
\item $\Lc_{011}$ and  $\Lcu_{011}$ (resp. $\ol{\Lc_{011}}$ and $\ol{\Lcu_{011}}$)
  are  in $\Nge{X} 1 (2)$ (resp. in $\Nge{\A^1} 1 (2)$).
\item They satisfy the following differential equations 
\[
\dN(\Lc_{011})=\dN(\Lcu_{011})=\Lcu_{01} \, \Lc_1=-\Lc_1 \, \Lcu_{01}
\]
and $\da[\A^1](\ol{\Lc_{011}})=\da[\A^1](\ol{\Lcu_{011}})=\ol{\Lcu_{01}}\, \ol{\Lc_1}$.
\item The cycle $\ol{\Lc_{011}}$ (resp. $\ol{\Lcu_{011}}$) has an empty   fiber
  at $0$ (resp. at $1$).  
\end{enumerate}
\end{lem}
\subsubsection{Weight $4$}In weight $4$ the first linear combination
appears. The situation in weight $4$ is given by the following Lemma
 \begin{lem} \label{weight4}Let $W$ be one of the Lyndon words $0001$, $0011$
   or $0111$. There 
   exist cycles $\Lc_{W}$, $\Lcu_{W}$ in $\Nge{X} 1 (4)$ and cycles 
$\ol{\Lc_{W}}$,
   $\ol{\Lcu_{W}}$ in $\Nge{\A^1} 1 (4)$ 
 which  satisfy the following properties 
\begin{enumerate}
\item $\Lc_W=j^*(\ol{\Lc_W})$ and $\Lcu_W=j^*(\ol{\Lcu_W})$
\item $\ol{\Lc_W}$ (resp. $\ol{\Lcu_W}$) has an empty fiber at $0$ (resp. at $1$) 
\item Cycles $\Lc_W$ and $\Lcu_W$ for $W=0001$,  $0011$ and $0111$ satisfy the
  following differential equations derived from the differential equations
  satisfied by $\ol{\Lc_W}$ and $\ol{\Lcu_W}$
\begin{equation}\label{exmL0001:X}
\dN(\Lc_{0001})=\dN(\Lcu_{0001})=\Lc_0 \, \Lc_{001},
\end{equation}
\begin{equation}\label{exmL0011:X}
\dN(\Lc_{0011})=\dN(\Lcu_{0011})=\Lc_0 \, \Lc_{011}+\Lcu_{001}\,
\Lc_1 -\Lc_{01} \, \Lcu_{01}
\end{equation}
and
\begin{equation}\label{exmL0111:X}
\dN(\Lc_{0111})=\dN(\Lcu_{0111})=\Lcu_{011} \, \Lc_{1}.
\end{equation}
\end{enumerate}
\end{lem}
\begin{proof}
The proof goes as before as the main difficulty is to ``guess'' the
differential equations. The case of $\Lc_{0001}=\on{Li}^{cy}_4$ and
$\Lcu_{0001}$ has already
been treated in Lemma \ref{lem:cycleLi_k} and the remark afterward. The case of
$\Lc_{0111}$ and $\Lcu_{0111}$ is extremely similar to the case of
$\Lc_{011}$. We will only describe the case of $\Lc_{0011}$. Let $\ol b$ be the
element in $\cNg{\A^1}{2}(4)$ defined by:
\[
\ol b=\ol{\Lc_0} \, \ol{\Lc_{011}}+\ol{\Lcu_{001}}\,
\ol{\Lc_1} -\ol{\Lc_{01}} \, \ol{\Lcu_{01}}.
\]
All the cycles involved are equidimensional over $\A^1 \sm \{0,1\}$. As the
products in the above equation always involve a cycle with empty fiber at 
$0$ and one with empty fiber at $1$, the product has empty fiber at $0$ and $1$
and is equidimensional over $\A^1$. 

This shows that  $\ol b$ is equidimensional over $\A^1$ with empty fiber at $0$
and $1$. One computes $\da[\A^1](\ol b)$ using
the Leibniz rule, Lemma \ref{dA1-L0L1} and the previously obtained differential
equations:
\begin{align*}
\da[\A^1](\ol b) = -
\ol{\Lc_0} \, \ol{\Lcu_{01}} \, \ol{\Lc_1}
+\ol{\Lc_0}\, \ol{\Lc_{01}} \, \ol{\Lc_{1}} 
-\ol{\Lc_0}\,\ol{\Lc_{1}}\, \ol{\Lcu_{01}} +
\ol{\Lc_{01}}\, \ol{\Lc_0}\,\ol{\Lc_{1}}=0
\end{align*}
One can thus define
\[
\ol{\Lc_{0011}}=\mu^*(\ol b) 
\qquad \mx{and} \qquad 
\ol{\Lcu_{0011}}=\nu^*(\ol b) 
\]
and conclude with propositions \ref{multhomo} and \ref{multequi}.
\end{proof}
\subsubsection{General statement and a weight $5$ example}
In weight $5$ there are six Lyndon words and the combinatorics of equation
\eqref{ED-T} leads to six cycles with empty fiber at $0$ and
six  cycles with empty fiber at $1$. 
The general statement proved in \cite{SouMPCC} is given below.

\begin{thm}\label{thm:cycleLcLcu}  For any Lyndon word $W$ of length $p$ greater or
  equal to $2$, there exist
  two cycles $\Lc_W$ and $\Lcu_W$ in $\cN (p)$ such that :
\begin{itemize}
\item $\Lc_W$, $\Lcu_W$ are elements in $\Nge{X}{1}(p)$.
\item There exist cycles $\ol{\Lc_W}$, $\ol{\Lcu_W}$ in $\Nge{\A^1} 1 (p)$ such
  that
\[
\Lc_W=j^*(\ol{\Lc_W}) \qquad \mx{and} \qquad \Lcu_W=j^*(\ol{\Lcu_W}). 
\]
\item The restriction of $\ol{\Lc_W}$ (resp. $\ol{\Lcu_W}$) to the fiber $t=0$
  (resp. $t=1$)
  is empty. 
\item The cycle $\Lc_W$  satisfies the equation
\begin{equation}\label{ED-L}
\dN(\Lc_W)=\sum_{U<V} a_{U,V}^W\Lc_U \Lc_V + \sum_{U,V}b_{U,V}^W\Lc_{U}\Lcu_V
\end{equation}
and resp. $\Lcu_W$ satisfies 
\begin{equation}\label{ED-Lc}
\dN(\Lcu_W)=\sum_{0<U<V} \ap_{U,V}^W\Lcu_U \Lcu_V +
\sum_{U,V}\bp_{U,V}^W\Lc_{U}\Lcu_V+
\sum_{V}\ap_{0,V}\Lc_0\Lc_V 
\end{equation}
 and the same holds for their extensions
  $\ol{\Lc_W}$ and $\ol{\Lcu_W}$ to $\Nge{\A^1}{1}$. In the above
  equations  $U$ and $V$ are Lyndon words of smaller length than
  $W$ and the coefficients $a_{U,V}^W$, $b_{U,V}^W$, $\ap_{U,V}^w$ and
  $\bp_{U,V}^W$ are integers derived from equation \eqref{ED-T}.
\end{itemize}  
\end{thm}
\begin{rem}\label{rem:thm}Without giving a proof which works by induction on the
  length of $W$, the author would like to stress that the
  construction of the cycles $\ol{\Lc_W}$ (resp. $\ol{\Lcu_W}$) relies on a
  geometric argument that has already been described and used here: the pull-back by
  the (twisted) multiplication $\mu^*$ (resp. $\nu^*$) gives a homotopy between the
  identity and $p^* \circ i_0^*$ (resp. $p^* \circ i_1^*$).  Thus, defining
\[
\ol{A_W}=\sum_{U<V} a_{U,V}^W\ol{\Lc_U}\ol{ \Lc_V} + \sum_{U,V}b_{U,V}^W
\ol{\Lc_{U}}\ol{\Lcu_V}
\]
and 
\[
\ol{A_W^1}=\sum_{0<U<V} \ap_{U,V}^W\ol{\Lcu_U}\ol{ \Lcu_V} +
\sum_{U,V}\bp_{U,V}^W\ol{\Lc_{U}}\ol{\Lcu_V}+
\sum_{V}\ap_{0,V}\ol{\Lc_0}\ol{\Lc_V},
\]
the cycle $\ol{\Lc_W}$ and $\ol{\Lcu_{W}}$ are defined by 
\begin{equation}\label{eqdef:LW}
\ol{\Lc_W}=\mu^*(\ol{A_W}) 
\qquad \mx{and} \qquad 
\ol{\Lcu_{W}}=\nu^*(\ol{A_W^1}).
\end{equation}
The fact that $\ol{A_W}$ (resp. $\ol{A_{W}^1}$) is equidimensional over $\A^1$
with empty fiber at $0$ (resp. $1$) is essentially a consequence of the
induction. The main problem is to show that
$\da[\A^1](\ol{A_W})=\da[\A^1](\ol{A_W^1})=0$ which in \cite{SouMPCC} is deduced
after a long preliminary work from the combinatorial situation given by the
trees $\T{W}$. 
\end{rem}

In weight $5$ appears the need of two distinct differential equations and the
first example with coefficient different from $\pm 1$.
\begin{exm}\label{exmL01011} The two cycles associated to the Lyndon word
  $01011$ satisfy 
\begin{align}
 \dN(\Lc_{01011})=&-\Lc_{01}\, \Lc_{011}- \Lc_{1}\, \Lcu_{0011}- 2\Lc_{011}\,\Lcu_{01} \\
\dN(\Lcu_{01011})=&\Lcu_{01}\, \Lcu_{011}-\Lc_{011}\, \Lcu_{01}
 -\Lc_{01}\, \Lcu_{011} -\Lc_{1}\, \Lcu_{0011}.
\end{align}
The factor $2$ in the last term of $\dN(\Lc_{01011})$ is related to the factor
$2$ appearing in $\dc(\T{01011})$ presented in Equation \eqref{exm:dT01011}. %
The term 
\[2 \Lc_{011}\Lc_0\Lc_1
\]
which is equal to $\dN(-2\Lc_{011}\Lcu_{01})$ cancels with one term in
$-\Lc_0\Lc_1\Lc_{011}$ coming from $\dN(-\Lc_{01}\cdot \Lc_{011})$ and one term
in $\Lc_1\Lc_0\Lc_{011}$ coming from $\dN(\Lc_{1}\cdot \Lc_{0011})$. The
whole computation can in fact be done over $\A^1$ and $\ol{\Lc_{01011}}$ is
defined as previously as the pull-back by $\mu^*$ of 
\[
\ol b= -\ol{\Lc_{01}}\, \ol{ \Lc_{011}}-\ol{\Lc_{1}}\, \ol{\Lcu_{0011}}
-2\ol{\Lc_{011}}\,\ol{\Lcu_{01}}.
\]
The cycle $\Lc_{01011}$ is then its restriction to $X$. The above linear
combination has an empty fiber at $0$ (which allows the use of $\mu^*$). However
its fiber at $1$ is nonempty and given by 
\[
-\ol{\Lc_{01}}|_{t=1}\, \ol{ \Lc_{011}}|_{t=1}
\]
and its  pull-back by the twisted multiplication $\nu^*$ satisfies 
\[
\da[\A^1](\nu^*(\ol b))=\ol b + p^*\circ i_1^*(\ol b)\neq \ol b.
\]
That is why we have  introduced the linear combination 
\[
\Lcu_{01}\, \Lcu_{011}-\Lc_{011}\, \Lcu_{01}
 -\Lc_{01}\, \Lcu_{011} -\Lc_{1}\, \Lcu_{0011}
\]
whose extension to $\A^1$ has empty fiber at $1$ (but not at $0$). This allows
us to define 
\[
\ol{\Lcu_{01011}}=\nu^*\left(
\Lcu_{01}\, \Lcu_{011}-\Lc_{011}\, \Lcu_{01}
 -\Lc_{01}\, \Lcu_{011} -\Lc_{1}\, \Lcu_{0011}
\right).
\]
\end{exm} 
\section[Tree with colored edges]{Parametric and combinatorial representation for the cycles: trees with colored
  edges}
\newcommand{\Tc}{\mf T}
\newcommand{\Tcu}{\Tc^{1}}
One can give a combinatorial approach to describe cycles $\Lc_W$
and $\Lcu_W$ as parametrized cycles using trivalent trees with two types of
edge.

\begin{defn}Let $\Tdr$ be the $\Q$ vector space spanned by rooted trivalent trees such
  that
\begin{itemize}
\item the edges can be of two types: 
 $\begin{tikzpicture}[%
baseline={([yshift=-0.5ex]current bounding box.center)},scale=1]
\node(root) {}
[level distance=1.5em,sibling distance=3ex]
child {node{}};
\end{tikzpicture}$ or $\begin{tikzpicture}[%
baseline={([yshift=-0.5ex]current bounding box.center)},scale=1]
\node(root) {}
[level distance=1.5em,sibling distance=3ex]
child {node{}
edge from parent [Reda]};
\end{tikzpicture}$;
\item the root vertex is decorated by $t$ ;
\item other external vertices are decorated by $0$ or $1$.
\end{itemize}
We say that such a tree is a rooted colored tree or simply a colored tree.
\end{defn}

We define two bilinear maps $\Tdr\otimes \Tdr \lra \Tdr$ as follows on the
colored trees:
\begin{itemize}
\item Let $T_1 \prac T_2$ be the colored tree given by joining the two roots of
  $T_1$ and $T_2$ and adding a new root and a new edge of type  $\begin{tikzpicture}[%
baseline={([yshift=-0.5ex]current bounding box.center)},scale=1]
\node(root) {}
[level distance=1.5em,sibling distance=3ex]
child {node{}};
\end{tikzpicture}$ :
\[
T_1 \prac T_2 =
\begin{tikzpicture}[%
baseline={(current bounding box.center)}]
\tikzstyle{every child node}=[mathscript mode,minimum size=0pt, inner sep=0pt]
\node[roots](root) 
{}
[level distance=1.2em,sibling distance=3ex]
child {node[fill, circle, minimum size=2pt,inner sep=0pt]{}[level distance=1.5em]
  child{ node[leaf]{T_1}edge from parent [Nd]}
  child{ node[leaf]{T_2}edge from parent [Nd]}
edge from parent [N]};
\fill (root.center) circle (1/1*\lbullet) ;
\node[mathsc, xshift=-1ex] at (root.west) {t};
\end{tikzpicture}
\]
where the dotted edges denote either type of edges.
\item  Let $T_1 \pracd T_2$ be the colored tree given by joining the two roots of
  $T_1$ and $T_2$ and adding a new root and a new edge of type  $\begin{tikzpicture}[%
baseline={([yshift=-0.5ex]current bounding box.center)},scale=1]
\node(root) {}
[level distance=1.5em,sibling distance=3ex]
child {node{}edge from parent [Reda]};
\end{tikzpicture}$ :
\[
T_1 \pracd T_2 =
\begin{tikzpicture}[%
baseline={(current bounding box.center)}]
\tikzstyle{every child node}=[mathscript mode,minimum size=0pt, inner sep=0pt]
\node[roots](root) 
{}
[level distance=1.2em,sibling distance=3ex]
child {node[fill, circle, minimum size=2pt,inner sep=0pt]{}[level distance=1.5em]
  child{ node[leaf]{T_1}edge from parent [Nd]}
  child{ node[leaf]{T_2}edge from parent [Nd]}
edge from parent [Reda]};
\fill (root.center) circle (1/1*\lbullet) ;
\node[mathsc, xshift=-1ex] at (root.west) {t};
\end{tikzpicture}
\]
where the dotted edges denote either type of edges.
\end{itemize}

\begin{defn} Let $\Tc_0$ and $\Tc_1$ be the colored trees defined by
\[\Tc_0=
\begin{tikzpicture}[baseline=(current bounding box.center)]
\tikzstyle{every child node}=[intvertex]
\node[roots](root) {}
[deftree]
child {node(1){} 
edge from parent [Reda]}
;
\fill (root.center) circle (1/1*\lbullet) ;
\node[mathsc, xshift=-1ex] at (root.west) {t};
\node[labf] at (1.south){0};
\end{tikzpicture} %
\quad \mx{ and } \quad 
\Tc_1=
\begin{tikzpicture}[baseline=(current bounding box.center)]
\tikzstyle{every child node}=[intvertex]
\node[roots](root) {}
[deftree]
child {node(1){} 
}
;
\fill (root.center) circle (1/1*\lbullet) ;
\node[mathsc, xshift=-1ex] at (root.west) {t};
\node[labf] at (1.south){1};
\end{tikzpicture}. %
\]

For any Lyndon word $W$ of length greater or equal to $2$, let $\Tc_W$
(resp. $\Tcu_W$) be the
linear combination of colored trees given by 
\[
\Tc_W=\sum_{U<V} a_{U,V}^W\Tc_U \prac \Tc_V + \sum_{U,V}b_{U,V}^W\Tc_{U}\prac \Tcu_V, 
\]
and respectively by
\[
\Tcu_W=\sum_{0<U<V} \ap_{U,V}^W\Tcu_U \pracd \Tcu_V +
\sum_{U,V}\bp_{U,V}^W\Tc_{U}\pracd \Tcu_V+
\sum_{V}\ap_{0,V}\Tc_0\pracd \Tc_V .
\]  
where the coefficients appearing are the ones from Theorem \ref{thm:cycleLcLcu}.
\end{defn}

To a colored tree $T$ with $p$ external leaves and a root, one associates a
function $f_T : X \times {(\p^1)}^{p -1} \lra X \times (\p^1)^{2p-1}$ as follows :
\begin{itemize}
\item Endow $T$ with its natural order as trivalent tree.
\item This induces a numbering of the edges of $T$ : $(e_1, e_2, \ldots ,
  e_{2p-1})$.
\item The edges being oriented away from the root, the numbering of the edges
  induces  a numbering of the
  vertices $(v_1, v_2, \ldots , v_{2p})$ such that the root is $v_1$.
\item Associate variables $x_1, \ldots, x_{p-1}$ to each internal vertices such
  that  the numbering of the variables is opposite to the order induced by the
  numbering of the vertices (first internal vertex has variable $x_{p-1}$, second
  internal vertex has variable $x_{p-2}$ and so on).
\item For each edge 
$e_i = \begin{tikzpicture}[baseline={([yshift=-1ex]current bounding box.center)}]
\tikzstyle{every child node}=[intvertex]
\node[intvertex](root) {}
[deftree]
child {node(1){} 
edge from parent [Nd]}
;
\node[mathsc, xshift=-1ex] at (root.west) {a};
\node[labf] at (1.south){b};
\end{tikzpicture}$ oriented from $a$ to $b$,  define a function 
\[
f_i(a,b)=\left\{
\begin{array}{ll}
\ds 1-\frac{a}{b}& \mx{if }e_i \mx{ is of type } 
\begin{tikzpicture}[%
baseline={([yshift=-0.5ex]current bounding box.center)},scale=1]
\node(root) {}
[level distance=1.5em,sibling distance=3ex]
child {node{}};
\end{tikzpicture},\\[1em]
\ds \frac{b-a}{b-1}& \mx{if }e_i \mx{ is of type }
\begin{tikzpicture}[%
baseline={([yshift=-0.5ex]current bounding box.center)},scale=1]
\node(root) {}
[level distance=1.5em,sibling distance=3ex]
child {node{}
edge from parent [Reda]};
\end{tikzpicture}. \\[1em]
\end{array}
\right.
\]
\item Finally $f_T : X\times (\p^1)^{p-1}  \lra X\times(\p^1)^{2p-1}$ is defined by 
\[
f_T(t, x_1, \ldots x_{p-1})=(t, f_1, \ldots, f_{2p-1}).
\]
\end{itemize}
Let $\Gamma(T)$ be the intersection of the image of $f_T$ with 
$X \times \square^{2p-1}$. 
One extends the definition of
$\Gamma$ to $\Tdr$ by linearity and thus obtains a \emph{twisted forest cycling
map} similar to the one defined by Gangl, Goncharov and Levin in \cite{GanGonLev05}. 

 The map $\Gamma$ satisfies:
\begin{itemize}
\item $\Alt(\Gamma(\Tc_0))=\Lc_0$ and $\Alt(\Gamma(\Tc_1))=\Lc_1$.
\item For any Lyndon word of length $p\geqs 2$, 
\[
\Alt(\Gamma(\Tc_W))=\Lc_W \qquad \mx{and} \qquad
\Alt(\Gamma(\Tcu_W))=\Lcu_W.
\]
\end{itemize}

 The fact that $\Gamma(\Tc_0)$ (resp. $\Gamma(\Tc_1)$) is the graph
  of $t \mapsto t$ (resp. $t \mapsto 1-t$) follows from the definition. Thus one
  already has  $\Gamma(\Tc_0)$ (resp. $\Gamma(\Tc_1)$) in $\Ze[1](X,1)$ and 
\[
\Alt(\Gamma(\Tc_0))=\Lc_0 \qquad \mx{and} \qquad
\Alt(\Gamma(\Tc_1))=\Lc_1.
\]

Then the above property is deduced by induction. Recall that the defining
equation \eqref{eqdef:LW} for the cycle $\ol{\Lc_W}$ is 
\[
\ol{\Lc_W}=\mu^*\left(
\sum_{U<V} a_{U,V}^W\ol{\Lc_U}\ol{ \Lc_V} + \sum_{U,V}b_{U,V}^W
\ol{\Lc_{U}}\ol{\Lcu_V}
\right).
\]
As already remarked in
Example  \ref{L01:para}, in order to compute the pull-back by $\mu^*$ one sets
the former parameter $t$ to a new variable $x_n$ and parametrizes 
the new $\square^1$ factor arriving in first position by $1 - \frac{t}{x_n}$; 
$t$ is again the  parameter over $X$ or $\A^1$ depending if one considers 
cycles over $\A^1$ or their restriction to $X=\ps$. Thus the expression of $\Lc_W$,
restriction of $\ol{\Lc_W}$ to $X$ is exactly given by 
\[
\Lc_W=\Alt(\Gamma(\Tc_W)).
\] 
The case of $\nu^*$ is similar but parametrizing the new
$\square^1$ factor by $\frac{x_n-t}{x_n-1}$.

For the previously built examples, we give below the corresponding colored trees
and expressions as parametrized 
cycles (omitting the projector $\Alt$). We also recall the corresponding
differential equations as given by Theorem \ref{thm:cycleLcLcu}.

 \newcommand{\locscale}{1}
 \begin{exm}[Weight $1$]
\[
\Tc_0=
\begin{tikzpicture}[baseline=(current bounding box.center)]
\tikzstyle{every child node}=[intvertex]
\node[roots](root) {}
[deftree]
child {node(1){} 
edge from parent [Reda]}
;
\fill (root.center) circle (1/1*\lbullet) ;
\node[mathsc, xshift=-1ex] at (root.west) {t};
\node[labf] at (1.south){0};
\end{tikzpicture} %
,\quad 
\Tc_1=
\begin{tikzpicture}[baseline=(current bounding box.center)]
\tikzstyle{every child node}=[intvertex]
\node[roots](root) {}
[deftree]
child {node(1){} 
}
;
\fill (root.center) circle (1/1*\lbullet) ;
\node[mathsc, xshift=-1ex] at (root.west) {t};
\node[labf] at (1.south){1};
\end{tikzpicture}\quad \mx{and }\dN(\Lc_0)=\dN(\Lc_1)=0.
\]
We recall below how cycles $\Lc_0$ and $\Lc_1$ are expressed in terms of
parametrized cycles:
\[
\Lc_0=[t;t] \quad \subset X \times \square^1
\qquad \mx{and} \qquad
\Lc_1=[t;1-t] \quad \subset X \times \square^1.
\]
\end{exm}
\begin{exm}[Weight $2$]
\[
\Tc_{01}=
\begin{tikzpicture}[baseline=(current bounding box.center)]
\tikzstyle{every child node}=[intvertex]
\node[roots](root) {}
[deftree]
child {node{}  
   child {node(1){}  
    edge from parent [Reda]}
   child {node(2){} 
    edge from parent [N]}
edge from parent [N]}
;
\fill (root.center) circle (1/1*\lbullet) ;
\node[mathsc, xshift=-1ex] at (root.west) {t};
\node[labf] at (1.south){0};
\node[labf] at (2.south){1};
\end{tikzpicture} %
,\quad 
\Tcu_{01}=
\begin{tikzpicture}[baseline=(current bounding box.center)]
\tikzstyle{every child node}=[intvertex]
\node[roots](root) {}
[deftree]
child {node{}  
   child {node(1){}  
    edge from parent [Reda]}
   child {node(2){} 
    edge from parent [N]}
edge from parent [Reda]}
;
\fill (root.center) circle (1/1*\lbullet) ;
\node[mathsc, xshift=-1ex] at (root.west) {t};
\node[labf] at (1.south){0};
\node[labf] at (2.south){1};
\end{tikzpicture} %
\quad \mx{and }\dN(\Lc_{01})=\dN(\Lcu_{01})=\Lc_0\Lc_1
\]
We have seen in Equation \eqref{L01:para} that cycles $\Lc_{01}$ and $\Lcu_{01}$
are given (in $X \times \square^3$) by
\[
\Lc_{01} =[t; 1-\frac{t}{x_1}, x_1, 1-x_1]
 \qquad \mx{and} \qquad 
 \Lcu_{01} =[t; \frac{x_1-t}{x_1-1}, x_1, 1-x_1].
\]
\end{exm}
\begin{exm}[Weight $3$]
\[
\dN(\Lc_{001})=\dN(\Lcu_{001})=\Lc_0\Lc_{01},
\quad
\dN(\Lc_{011})=\dN(\Lcu_{011})=-\Lc_1\Lcu_{01}. 
\]
\[
\Tc_{001}=
\begin{tikzpicture}[baseline=(current bounding box.center)]
\tikzstyle{every child node}=[intvertex]
\node[roots](root) {}
[deftree]
child{node{}
  child{node(1){}
    edge from parent [Reda]}
  child {node{}  
     child {node(2){}  
      edge from parent [Reda]}
     child {node(3){} 
      edge from parent [N]}
  edge from parent [N]}
edge from parent [N]}
;
\fill (root.center) circle (1/1*\lbullet) ;
\node[mathsc, xshift=-1ex] at (root.west) {t};
\node[labf] at (1.south){0};
\node[labf] at (2.south){0};
\node[labf] at (3.south){1};
\end{tikzpicture} %
,\quad 
\Tcu_{001}=
\begin{tikzpicture}[baseline=(current bounding box.center)]
\tikzstyle{every child node}=[intvertex]
\node[roots](root) {}
[deftree]
child{node{}
  child{node(1){}
    edge from parent [Reda]}
  child {node{}  
     child {node(2){}  
      edge from parent [Reda]}
     child {node(3){} 
      edge from parent [N]}
  edge from parent [N]}
edge from parent [Reda]}
;
\fill (root.center) circle (1/1*\lbullet) ;
\node[mathsc, xshift=-1ex] at (root.west) {t};
\node[labf] at (1.south){0};
\node[labf] at (2.south){0};
\node[labf] at (3.south){1};
\end{tikzpicture} %
,\quad 
\Tc_{011}=-
\begin{tikzpicture}[baseline=(current bounding box.center)]
\tikzstyle{every child node}=[intvertex]
\node[roots](root) {}
[deftree]
child{node{}
  child{node(1){}
    edge from parent [N]}
  child {node{}  
     child {node(2){}  
      edge from parent [Reda]}
     child {node(3){} 
      edge from parent [N]}
  edge from parent [Reda]}
edge from parent [N]
}
;
\fill (root.center) circle (1/\locscale*\lbullet) ;
\node[mathsc, xshift=-1ex] at (root.west) {t};
\node[labf] at (2.south){0};
\node[labf] at (3.south){1};
\node[labf] at (1.south){1};
\end{tikzpicture}
,\quad 
\Tcu_{011}=-
\begin{tikzpicture}[baseline=(current bounding box.center)]
\tikzstyle{every child node}=[intvertex]
\node[roots](root) {}
[deftree]
child{node{}
  child{node(1){}
    edge from parent [N]}
  child {node{}  
     child {node(2){}  
      edge from parent [Reda]}
     child {node(3){} 
      edge from parent [N]}
  edge from parent [Reda]}
edge from parent [Reda]
}
;
\fill (root.center) circle (1/\locscale*\lbullet) ;
\node[mathsc, xshift=-1ex] at (root.west) {t};
\node[labf] at (2.south){0};
\node[labf] at (3.south){1};
\node[labf] at (1.south){1};
\end{tikzpicture}
\]
The corresponding expression as parametrized cycles are given below (following
our ``twisted forest cycling map''):
\begin{multline*}
\Lc_{001}=[t;1-\frac{t}{x_2}, x_2, 1-\frac{x_2}{x_1}, x_1, 1-x_1]
\quad \subset X \times \square^5, \\
\Lcu_{001}=[t;\frac{x_2-t}{x_2-1}, x_2, 1-\frac{x_2}{x_1}, x_1, 1-x_1]
\quad \subset X \times \square^5
\end{multline*}
and
\begin{multline*}
\Lc_{011}=-[t;1-\frac{t}{x_2}, 1-x_2,\frac{x_1-x_2}{x_1-1}, x_1, 1-x_1]
\quad \subset X \times \square^5, \\
\Lcu_{011}=-[t;\frac{x_2-t}{x_2-1},1- x_2,\frac{x_1-x_2}{x_1-1}, x_1, 1-x_1]
\quad \subset X \times \square^5
\end{multline*}
\end{exm}
\begin{exm}[Weight $4$] The differential equations satisfied by the weight $4$
  cycles are:
\begin{align*}
\dN(\Lc_{0001})=\dN(\Lcu_{0001})=&\Lc_0\Lc_{001}
\\
\dN(\Lc_{0011})=\dN(\Lcu_{0011})=&\Lc_0\Lc_{011}-\Lc_1\Lcu_{001}-\Lc_{01}\Lcu_{01}
 \\
\dN(\Lc_{0111})=\dN(\Lcu_{0111})=&-\Lc_1\Lcu_{011}
\end{align*}
The corresponding colored trees are given by:
\[
\Tc_{0001}=
\begin{tikzpicture}[baseline=(current bounding box.center)]
\tikzstyle{every child node}=[intvertex]
\node[roots](root) {}
[deftree]
child{node{}
  child{node(1){}
    edge from parent [Reda]}
  child{node{}
    child{node(2){}
      edge from parent [Reda]}
    child {node{}  
       child {node(3){}  
        edge from parent [Reda]}
       child {node(4){} 
        edge from parent [N]}
    edge from parent [N]}
  edge from parent [N]}
edge from parent [N]
}
;
\fill (root.center) circle (1/\locscale*\lbullet) ;
\node[mathsc, xshift=-1ex] at (root.west) {t};
\node[labf] at (1.south){0};
\node[labf] at (2.south){0};
\node[labf] at (3.south){0};
\node[labf] at (4.south){1};
\end{tikzpicture}
,\quad 
\Tcu_{0001}=
\begin{tikzpicture}[baseline=(current bounding box.center)]
\tikzstyle{every child node}=[intvertex]
\node[roots](root) {}
[deftree]
child{node{}
  child{node(1){}
    edge from parent [Reda]}
  child{node{}
    child{node(2){}
      edge from parent [Reda]}
    child {node{}  
       child {node(3){}  
        edge from parent [Reda]}
       child {node(4){} 
        edge from parent [N]}
    edge from parent [N]}
  edge from parent [N]}
edge from parent [Reda]
}
;
\fill (root.center) circle (1/\locscale*\lbullet) ;
\node[mathsc, xshift=-1ex] at (root.west) {t};
\node[labf] at (1.south){0};
\node[labf] at (2.south){0};
\node[labf] at (3.south){0};
\node[labf] at (4.south){1};
\end{tikzpicture}
,\quad 
\Tc_{0111}=
\begin{tikzpicture}[baseline=(current bounding box.center)]
\tikzstyle{every child node}=[intvertex]
\node[roots](root) {}
[deftree]
child{node{}
  child{node(1){}
    edge from parent [N]}
  child{node{}
    child{node(2){}
      edge from parent [N]}
    child {node{}  
       child {node(3){}  
        edge from parent [Reda]}
       child {node(4){} 
        edge from parent [N]}
    edge from parent [Reda]}
  edge from parent [Reda]
  } 
edge from parent [N]
}
;
\fill (root.center) circle (1/\locscale*\lbullet) ;
\node[mathsc, xshift=-1ex] at (root.west) {t};
\node[labf] at (3.south){0};
\node[labf] at (4.south){1};
\node[labf] at (2.south){1};
\node[labf] at (1.south){1};
\end{tikzpicture}
,\quad
\Tcu_{0111}=
\begin{tikzpicture}[baseline=(current bounding box.center)]
\tikzstyle{every child node}=[intvertex]
\node[roots](root) {}
[deftree]
child{node{}
  child{node(1){}
    edge from parent [N]}
  child{node{}
    child{node(2){}
      edge from parent [N]}
    child {node{}  
       child {node(3){}  
        edge from parent [Reda]}
       child {node(4){} 
        edge from parent [N]}
    edge from parent [Reda]}
  edge from parent [Reda]
  } 
edge from parent [Reda]
}
;
\fill (root.center) circle (1/\locscale*\lbullet) ;
\node[mathsc, xshift=-1ex] at (root.west) {t};
\node[labf] at (3.south){0};
\node[labf] at (4.south){1};
\node[labf] at (2.south){1};
\node[labf] at (1.south){1};
\end{tikzpicture},
\]
\[
\Tc_{0011}=-
\begin{tikzpicture}[baseline=(current bounding box.center)]
\tikzstyle{every child node}=[intvertex]
\node[roots](root) {}
[deftree]
child{node{}
  child{node(1){}
    edge from parent [Reda]}
  child{node{}
    child{node(2){}
      edge from parent [N]}
    child {node{}  
       child {node(3){}  
        edge from parent [Reda]}
       child {node(4){} 
        edge from parent [N]}
    edge from parent [Reda]}
  edge from parent [N]
  }
edge from parent [N]
}
;
\fill (root.center) circle (1/\locscale*\lbullet) ;
\node[mathsc, xshift=-1ex] at (root.west) {t};
\node[labf] at (1.south){0};
\node[labf] at (3.south){0};
\node[labf] at (4.south){1};
\node[labf] at (2.south){1};
\end{tikzpicture}
-
\begin{tikzpicture}[baseline=(current bounding box.center)]
\tikzstyle{every child node}=[intvertex]
\node[roots](root) {}
[deftree]
child{node{}
  child{node(4){}
    edge from parent [N]}
  child{node{}
    child{node(1){}
      edge from parent [Reda]}
    child {node{}  
       child {node(2){}  
        edge from parent [Reda]}
       child {node(3){} 
        edge from parent [N]}
    edge from parent [N]}
  edge from parent [Reda]
  }
edge from parent [N]
}
;
\fill (root.center) circle (1/\locscale*\lbullet) ;
\node[mathsc, xshift=-1ex] at (root.west) {t};
\node[labf] at (1.south){0};
\node[labf] at (2.south){0};
\node[labf] at (3.south){1};
\node[labf] at (4.south){1};
\end{tikzpicture} %
-
\begin{tikzpicture}[baseline=(current bounding box.center)]
\tikzstyle{every child node}=[intvertex]
\node[roots](root) {}
[deftree]
child{node{}
  child[edgesp=2] {node{}  
     child[edgesp=1] {node(1){}  
      edge from parent [Reda]}
     child[edgesp=1] {node(2){} 
      edge from parent [N]}
  edge from parent [N]
  }
  child[edgesp=2] {node{}  
     child[edgesp=1] {node(3){}  
      edge from parent [Reda]}
     child[edgesp=1] {node(4){} 
      edge from parent [N]}
  edge from parent [Reda]
  }
edge from parent [N]
}
;
\fill (root.center) circle (1/\locscale*\lbullet) ;
\node[mathsc, xshift=-1ex] at (root.west) {t};
\node[labf] at (1.south){0};
\node[labf] at (2.south){1};
\node[labf] at (3.south){0};
\node[labf] at (4.south){1};
\end{tikzpicture}
\]
and
\[
\Tcu_{0011}=-
\begin{tikzpicture}[baseline=(current bounding box.center)]
\tikzstyle{every child node}=[intvertex]
\node[roots](root) {}
[deftree]
child{node{}
  child{node(1){}
    edge from parent [Reda]}
  child{node{}
    child{node(2){}
      edge from parent [N]}
    child {node{}  
       child {node(3){}  
        edge from parent [Reda]}
       child {node(4){} 
        edge from parent [N]}
    edge from parent [Reda]}
  edge from parent [N]
  }
edge from parent [Reda]
}
;
\fill (root.center) circle (1/\locscale*\lbullet) ;
\node[mathsc, xshift=-1ex] at (root.west) {t};
\node[labf] at (1.south){0};
\node[labf] at (3.south){0};
\node[labf] at (4.south){1};
\node[labf] at (2.south){1};
\end{tikzpicture}
-
\begin{tikzpicture}[baseline=(current bounding box.center)]
\tikzstyle{every child node}=[intvertex]
\node[roots](root) {}
[deftree]
child{node{}
  child{node(4){}
    edge from parent [N]}
  child{node{}
    child{node(1){}
      edge from parent [Reda]}
    child {node{}  
       child {node(2){}  
        edge from parent [Reda]}
       child {node(3){} 
        edge from parent [N]}
    edge from parent [N]}
  edge from parent [Reda]
  }
edge from parent [Reda]
}
;
\fill (root.center) circle (1/\locscale*\lbullet) ;
\node[mathsc, xshift=-1ex] at (root.west) {t};
\node[labf] at (1.south){0};
\node[labf] at (2.south){0};
\node[labf] at (3.south){1};
\node[labf] at (4.south){1};
\end{tikzpicture} %
-
\begin{tikzpicture}[baseline=(current bounding box.center)]
\tikzstyle{every child node}=[intvertex]
\node[roots](root) {}
[deftree]
child{node{}
  child[edgesp=2] {node{}  
     child[edgesp=1] {node(1){}  
      edge from parent [Reda]}
     child[edgesp=1] {node(2){} 
      edge from parent [N]}
  edge from parent [N]
  }
  child[edgesp=2] {node{}  
     child[edgesp=1] {node(3){}  
      edge from parent [Reda]}
     child[edgesp=1] {node(4){} 
      edge from parent [N]}
  edge from parent [Reda]
  }
edge from parent [Reda]
}
;
\fill (root.center) circle (1/\locscale*\lbullet) ;
\node[mathsc, xshift=-1ex] at (root.west) {t};
\node[labf] at (1.south){0};
\node[labf] at (2.south){1};
\node[labf] at (3.south){0};
\node[labf] at (4.south){1};
\end{tikzpicture} .
\]

The expressions as parametrized cycles of $\Lc_{0001}$, $\Lcu_{0001}$,
$\Lc_{0111}$ and $\Lcu_{0111}$ are given below (in $X\times
\square^7$):
\begin{align*}
\Lc_{0001}&=[t;1-\frac{t}{x_3}, x_3,1-\frac{x_3}{x_2}, x_2, 
1-\frac{x_2}{x_1}, x_1, 1-x_1],
\\
\Lcu_{0001}&=[t;\frac{x_3-t}{x_3-1}, x_3,1-\frac{x_3}{x_2}, x_2, 
1-\frac{x_2}{x_1}, x_1, 1-x_1] ,\\
\intertext{}
\Lc_{0111}&=[t;1-\frac{t}{x_3},1-x_3, \frac{x_2-x_3}{x_2-1}
, 1-x_2,\frac{x_1-x_2}{x_1-1}, x_1, 1-x_1],
\\
\Lcu_{0111}&=[t;\frac{x_3-t}{x_3-1},1- x_3,\frac{x_2-x_3}{x_2-1}
,1- x_2,\frac{x_1-x_2}{x_1-1}, x_1, 1-x_1],
\end{align*}
while the expressions for $\Lc_{0011}$ and $\Lcu_{0011}$ involved linear
combinations:
\begin{multline*}
\Lc_{0011}=-[t;1-\frac{t}{x_3}, x_3,
1-\frac{x_3}{x_2},1-x_2,\frac{x_1-x_2}{x_1-1},x_1,1-x_1] \\ 
-[t;1-\frac{t}{x_3},1-x_3,\frac{x_2-x_3}{x_2-1}
,x_2,1-\frac{x_2}{x_1},x_1,1-x_1] \\
-[t;1- \frac{t}{x_3},1-\frac{x_3}{x_2}, x_2,1-x_2,\frac{x_1-x_3}{x_1-1}, x_1,1-x_1]
\end{multline*}
and 
\begin{multline*}
\Lcu_{0011}=-[t;\frac{x_3-t}{x_3-1}, x_3,
1-\frac{x_3}{x_2},1-x_2,\frac{x_1-x_2}{x_1-1},x_1,1-x_1] \\ 
-[t;\frac{x_3-t}{x_3-1},1-x_3,\frac{x_2-x_3}{x_2-1}
,x_2,1-\frac{x_2}{x_1},x_1,1-x_1] \\
-[t;\frac{x_3-t}{x_3-1},1-\frac{x_3}{x_2}, x_2,1-x_2,\frac{x_1-x_3}{x_1-1}, x_1,1-x_1].
\end{multline*}
\end{exm}
\begin{exm}[Weight $5$] The differential equations satisfied by $\Lc_{01011}$ and
  $\Lcu_{01011}$ are:
 \begin{align*}
\dN(\Lc_{01011})=&-\Lc_{01}\cdot \Lc_{011}-\Lc_{1}\Lcu_{0011}-2\Lc_{011}\Lcu_{01} \\
\dN(\Lcu_{01011})=&\Lcu_{01}\cdot \Lcu_{011}-\Lc_{011}\cdot \Lcu_{01}
-\Lc_{01}\cdot \Lcu_{011} -\Lc_{1}\cdot \Lcu_{0011}.
 \end{align*}
The corresponding colored trees are given by:
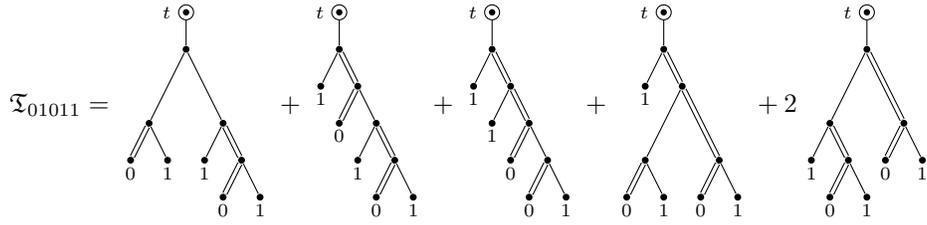
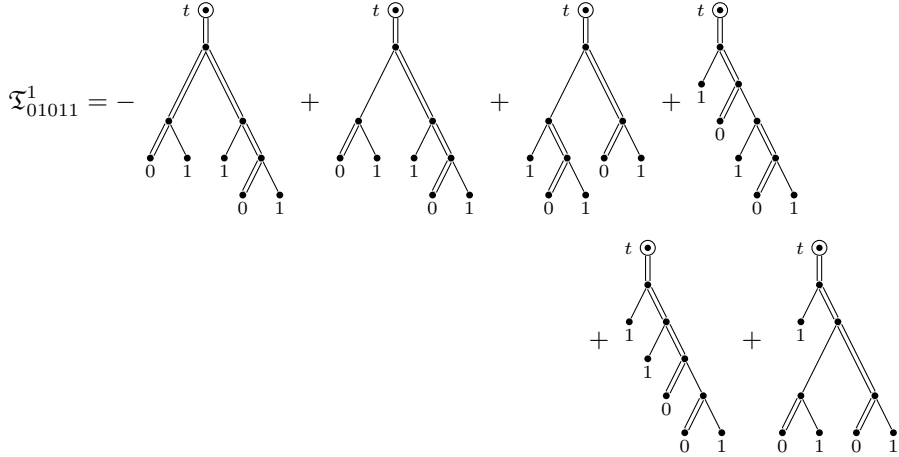
\[
\Tc_{01011}=
\begin{tikzpicture}[baseline=(current bounding box.center)]
\tikzstyle{every child node}=[intvertex]
\node[roots](root) {}
[deftree]
child{node{}
  child[edgesp=2] {node{}  
     child[edgesp=1] {node(1){}  
      edge from parent [Reda]}
     child[edgesp=1] {node(2){} 
      edge from parent [N]}
  edge from parent [N]
  }
  child[edgesp=2]{node{}
    child[edgesp=1]{node(3){}
      edge from parent [N]}
    child[edgesp=1] {node{}  
       child {node(4){}  
        edge from parent [Reda]}
       child {node(5){} 
        edge from parent [N]}
    edge from parent [Reda]}
  edge from parent [N]
  }
edge from parent [N]
}
;
\fill (root.center) circle (1/\locscale*\lbullet) ;
\node[mathsc, xshift=-1ex] at (root.west) {t};
\node[labf] at (1.south){0};
\node[labf] at (2.south){1};
\node[labf] at (3.south){1};
\node[labf] at (4.south){0};
\node[labf] at (5.south){1};
\end{tikzpicture}
+
\begin{tikzpicture}[baseline=(current bounding box.center)]
\tikzstyle{every child node}=[intvertex]
\node[roots](root) {}
[deftree]
child{node{}
  child{node(1){}
    edge from parent [N]}
  child{node{}
    child{node(2){}
      edge from parent [Reda]}
    child{node{}
      child{node(3){}
        edge from parent [N]}
      child {node{}  
         child {node(4){}  
          edge from parent [Reda]}
         child {node(5){} 
          edge from parent [N]}
      edge from parent [Reda]}
    edge from parent [N]
    }
  edge from parent [Reda]
  }
edge from parent [N]
}
;
\fill (root.center) circle (1/\locscale*\lbullet) ;
\node[mathsc, xshift=-1ex] at (root.west) {t};
\node[labf] at (1.south){1};
\node[labf] at (2.south){0};
\node[labf] at (4.south){0};
\node[labf] at (5.south){1};
\node[labf] at (3.south){1};
\end{tikzpicture}
+
\begin{tikzpicture}[baseline=(current bounding box.center)]
\tikzstyle{every child node}=[intvertex]
\node[roots](root) {}
[deftree]
child{node{}
  child{node(1){}
    edge from parent[N]}  
  child{node{}
    child{node(5){}
      edge from parent [N]}
    child{node{}
      child{node(2){}
        edge from parent [Reda]}
      child {node{}  
         child {node(3){}  
          edge from parent [Reda]}
         child {node(4){} 
          edge from parent [N]}
      edge from parent [N]}
    edge from parent [Reda]
    }
  edge from parent [Reda]
  }
edge from parent [N]
}
;
\fill (root.center) circle (1/\locscale*\lbullet) ;
\node[mathsc, xshift=-1ex] at (root.west) {t};
\node[labf] at (1.south){1};
\node[labf] at (2.south){0};
\node[labf] at (3.south){0};
\node[labf] at (4.south){1};
\node[labf] at (5.south){1};
\end{tikzpicture} %
+
\begin{tikzpicture}[baseline=(current bounding box.center)]
\tikzstyle{every child node}=[intvertex]
\node[roots](root) {}
[deftree]
child{node{}
  child{node(1){}
    edge from parent[N]}
  child{node{}
    child[edgesp=2] {node{}  
       child[edgesp=1] {node(2){}  
        edge from parent [Reda]}
       child[edgesp=1] {node(3){} 
        edge from parent [N]}
    edge from parent [N]
    }
    child[edgesp=2] {node{}  
       child[edgesp=1] {node(4){}  
        edge from parent [Reda]}
       child[edgesp=1] {node(5){} 
        edge from parent [N]}
    edge from parent [Reda]
    }
  edge from parent [Reda]
  }
edge from parent [N]
}
;
\fill (root.center) circle (1/\locscale*\lbullet) ;
\node[mathsc, xshift=-1ex] at (root.west) {t};
\node[labf] at (1.south){1};
\node[labf] at (2.south){0};
\node[labf] at (3.south){1};
\node[labf] at (4.south){0};
\node[labf] at (5.south){1};
\end{tikzpicture}
+2
\begin{tikzpicture}[baseline=(current bounding box.center)]
\tikzstyle{every child node}=[intvertex]
\node[roots](root) {}
[deftree]
child{node{}
  child[edgesp=2]{node{}
    child[edgesp=1]{node(3){}
      edge from parent [N]}
    child[edgesp=1] {node{}  
       child {node(4){}  
        edge from parent [Reda]}
       child {node(5){} 
        edge from parent [N]}
    edge from parent [Reda]}
  edge from parent [N]
  }
  child[edgesp=2] {node{}  
     child[edgesp=1] {node(1){}  
      edge from parent [Reda]}
     child[edgesp=1] {node(2){} 
      edge from parent [N]}
  edge from parent [Reda]
  }
edge from parent [N]
}
;
\fill (root.center) circle (1/\locscale*\lbullet) ;
\node[mathsc, xshift=-1ex] at (root.west) {t};
\node[labf] at (1.south){0};
\node[labf] at (2.south){1};
\node[labf] at (3.south){1};
\node[labf] at (4.south){0};
\node[labf] at (5.south){1};
\end{tikzpicture} 
\]
and
\begin{multline*}
\Tcu_{01011}=
-
\begin{tikzpicture}[baseline=(current bounding box.center)]
\tikzstyle{every child node}=[intvertex]
\node[roots](root) {}
[deftree]
child{node{}
  child[edgesp=2] {node{}  
     child[edgesp=1] {node(1){}  
      edge from parent [Reda]}
     child[edgesp=1] {node(2){} 
      edge from parent [N]}
  edge from parent [Reda]
  }
  child[edgesp=2]{node{}
    child[edgesp=1]{node(3){}
      edge from parent [N]}
    child[edgesp=1] {node{}  
       child {node(4){}  
        edge from parent [Reda]}
       child {node(5){} 
        edge from parent [N]}
    edge from parent [Reda]}
  edge from parent [Reda]
  }
edge from parent [Reda]
}
;
\fill (root.center) circle (1/\locscale*\lbullet) ;
\node[mathsc, xshift=-1ex] at (root.west) {t};
\node[labf] at (1.south){0};
\node[labf] at (2.south){1};
\node[labf] at (3.south){1};
\node[labf] at (4.south){0};
\node[labf] at (5.south){1};
\end{tikzpicture}
+
\begin{tikzpicture}[baseline=(current bounding box.center)]
\tikzstyle{every child node}=[intvertex]
\node[roots](root) {}
[deftree]
child{node{}
  child[edgesp=2] {node{}  
     child[edgesp=1] {node(1){}  
      edge from parent [Reda]}
     child[edgesp=1] {node(2){} 
      edge from parent [N]}
  edge from parent [N]
  }
  child[edgesp=2]{node{}
    child[edgesp=1]{node(3){}
      edge from parent [N]}
    child[edgesp=1] {node{}  
       child {node(4){}  
        edge from parent [Reda]}
       child {node(5){} 
        edge from parent [N]}
    edge from parent [Reda]}
  edge from parent [Reda]
  }
edge from parent [Reda]
}
;
\fill (root.center) circle (1/\locscale*\lbullet) ;
\node[mathsc, xshift=-1ex] at (root.west) {t};
\node[labf] at (1.south){0};
\node[labf] at (2.south){1};
\node[labf] at (3.south){1};
\node[labf] at (4.south){0};
\node[labf] at (5.south){1};
\end{tikzpicture}
+
\begin{tikzpicture}[baseline=(current bounding box.center)]
\tikzstyle{every child node}=[intvertex]
\node[roots](root) {}
[deftree]
child{node{}
  child[edgesp=2]{node{}
    child[edgesp=1]{node(3){}
      edge from parent [N]}
    child[edgesp=1] {node{}  
       child {node(4){}  
        edge from parent [Reda]}
       child {node(5){} 
        edge from parent [N]}
    edge from parent [Reda]}
  edge from parent [N]
  }
  child[edgesp=2] {node{}  
     child[edgesp=1] {node(1){}  
      edge from parent [Reda]}
     child[edgesp=1] {node(2){} 
      edge from parent [N]}
  edge from parent [Reda]
  }
edge from parent [Reda]
}
;
\fill (root.center) circle (1/\locscale*\lbullet) ;
\node[mathsc, xshift=-1ex] at (root.west) {t};
\node[labf] at (1.south){0};
\node[labf] at (2.south){1};
\node[labf] at (3.south){1};
\node[labf] at (4.south){0};
\node[labf] at (5.south){1};
\end{tikzpicture} 
+
\begin{tikzpicture}[baseline=(current bounding box.center)]
\tikzstyle{every child node}=[intvertex]
\node[roots](root) {}
[deftree]
child{node{}
  child{node(1){}
    edge from parent [N]}
  child{node{}
    child{node(2){}
      edge from parent [Reda]}
    child{node{}
      child{node(3){}
        edge from parent [N]}
      child {node{}  
         child {node(4){}  
          edge from parent [Reda]}
         child {node(5){} 
          edge from parent [N]}
      edge from parent [Reda]}
    edge from parent [N]
    }
  edge from parent [Reda]
  }
edge from parent [Reda]
}
;
\fill (root.center) circle (1/\locscale*\lbullet) ;
\node[mathsc, xshift=-1ex] at (root.west) {t};
\node[labf] at (1.south){1};
\node[labf] at (2.south){0};
\node[labf] at (4.south){0};
\node[labf] at (5.south){1};
\node[labf] at (3.south){1};
\end{tikzpicture}
\\
+
\begin{tikzpicture}[baseline=(current bounding box.center)]
\tikzstyle{every child node}=[intvertex]
\node[roots](root) {}
[deftree]
child{node{}
  child{node(1){}
    edge from parent[N]}  
  child{node{}
    child{node(5){}
      edge from parent [N]}
    child{node{}
      child{node(2){}
        edge from parent [Reda]}
      child {node{}  
         child {node(3){}  
          edge from parent [Reda]}
         child {node(4){} 
          edge from parent [N]}
      edge from parent [N]}
    edge from parent [Reda]
    }
  edge from parent [Reda]
  }
edge from parent [Reda]
}
;
\fill (root.center) circle (1/\locscale*\lbullet) ;
\node[mathsc, xshift=-1ex] at (root.west) {t};
\node[labf] at (1.south){1};
\node[labf] at (2.south){0};
\node[labf] at (3.south){0};
\node[labf] at (4.south){1};
\node[labf] at (5.south){1};
\end{tikzpicture} %
+
\begin{tikzpicture}[baseline=(current bounding box.center)]
\tikzstyle{every child node}=[intvertex]
\node[roots](root) {}
[deftree]
child{node{}
  child{node(1){}
    edge from parent[N]}
  child{node{}
    child[edgesp=2] {node{}  
       child[edgesp=1] {node(2){}  
        edge from parent [Reda]}
       child[edgesp=1] {node(3){} 
        edge from parent [N]}
    edge from parent [N]
    }
    child[edgesp=2] {node{}  
       child[edgesp=1] {node(4){}  
        edge from parent [Reda]}
       child[edgesp=1] {node(5){} 
        edge from parent [N]}
    edge from parent [Reda]
    }
  edge from parent [Reda]
  }
edge from parent [Reda]
}
;
\fill (root.center) circle (1/\locscale*\lbullet) ;
\node[mathsc, xshift=-1ex] at (root.west) {t};
\node[labf] at (1.south){1};
\node[labf] at (2.south){0};
\node[labf] at (3.south){1};
\node[labf] at (4.south){0};
\node[labf] at (5.south){1};
\end{tikzpicture}
\end{multline*}
The corresponding expression as parametrized cycles are given below (in $X\times
\square^9$):
\begin{multline*}
\Lc_{01011}=[t;1-\frac{t}{x_4},1-\frac{x_4}{x_3},x_3,1-x_3,
1-\frac{x_4}{x_2},1-x_2,\frac{x_1-x_2}{x_2-1},x_1,1-x_1] \\
+[t;1-\frac{t}{x_4},1-x_4,\frac{x_3-x_4}{x_3-1}, x_3,
1-\frac{x_3}{x_2},1-x_2,\frac{x_1-x_2}{x_1-1},x_1,1-x_1] \\ 
+[t;1-\frac{t}{x_4},1-x_4,\frac{x_3-x_4}{x_3-1},1-x_3,\frac{x_2-x_3}{x_2-1}
,x_2,1-\frac{x_2}{x_1},x_1,1-x_1] \\
+[t;1-\frac{t}{x_4},1-x_4,\frac{x_3-x_4}{x_3-1},1-\frac{x_3}{x_2}, x_2,1-x_2,
\frac{x_1-x_3}{x_1-1}, x_1,1-x_1] \\
+2[t; 1-\frac{t}{x_4},1-\frac{x_4}{x_3},1-x_3,\frac{x_2-x_3}{x_2-1},x_2,1-x_2,
\frac{x_1-x_4}{x_1-1},x_1,1-x_1],
\end{multline*}
and
\begin{multline*}
\Lcu_{01011}=-[t;\frac{x_4-t}{x_4-1},\frac{x_3-x_4}{x_3-1},x_3,1-x_3,
\frac{x_2-x_4}{x_2-1},1-x_2,\frac{x_1-x_2}{x_2-1},x_1,1-x_1] \\
+[t; \frac{x_4-t}{x_4-1}, 1-\frac{x_4}{x_3},x_3,1-x_3,
\frac{x_2-x_4}{x_2-1},1-x_2,\frac{x_1-x_2}{x_1-1},x_1,1-x_1]\\
+[t; \frac{x_4-t}{x_4-1},1-\frac{x_4}{x_3},1-x_3,\frac{x_2-x_3}{x_2-1},x_2,1-x_2,
\frac{x_1-x_4}{x_1-1},x_1,1-x_1] \\
+[t;\frac{x_4-t}{x_4-1},1-x_4,\frac{x_3-x_4}{x_3-1}, x_3,
1-\frac{x_3}{x_2},1-x_2,\frac{x_1-x_2}{x_1-1},x_1,1-x_1] \\ 
+[t;\frac{x_4-t}{x_4-1},1-x_4,\frac{x_3-x_4}{x_3-1},1-x_3,\frac{x_2-x_3}{x_2-1}
,x_2,1-\frac{x_2}{x_1},x_1,1-x_1] \\
+[t;\frac{x_4-t}{x_4-1},1-x_4,\frac{x_3-x_4}{x_3-1},1-\frac{x_3}{x_2}, x_2,1-x_2,
\frac{x_1-x_3}{x_1-1}, x_1,1-x_1] .
\end{multline*}
 \end{exm}
\section{Bar construction settings}
In the cycle motives setting, a motive
over $X$ is a comodule on the $\HH^0$ of the bar construction over $
\cNg{X}{\bullet}$ modulo shuffle products. For more details, one can look at the works of Bloch and Kriz \cite{BKMTM},
Spitzweck \cite{SpitzweckSCVTM, OAMSpit} (i.e. as presented in \cite{KTMMLevine}) and
 Levine \cite{LEVTMFG}.

In this context, the cycles constructed above, which are expected to correspond
to multiple polylogarithms (as outlined in Section \ref{sec:int}), induce
elements in this $\HH^0$ 
and naturally gives rise to an associated comodule, thus to mixed Tate motives
corresponding to multiple polylogarithms.

Before, giving explicit expressions for the induced elements in the bar
construction, the beginning of the section is devoted to a short review of the
bar construction.
\subsection{Bar construction}
 \label{sec:bar}
As there does not seem to exist a global sign
convention for the various operations on the bar construction,  
the main definitions in the cohomological setting are recalled
below following the (homological) description given in \cite{LodValAO}.

Let $A$ be a commutative differential graded algebra (c.d.g.a.) with
augmentation $\ve : A \lra \Q$, with product 
$\mu_A$ and let $A^+$ be the augmentation ideal $A^+=\ker(\ve)$. Note again that
\emph{commutative} in this context stand for \emph{graded commutative}.  

In order to understand the sign convention below and the ``bar grading'', one
should think of the bar construction as built on the tensor coalgebra over 
the shifted (suspended) graded vector space $A^+[1]$. 
\begin{defn} The bar construction $B(A)$ over $A$ is the tensor coalgebra over
  the suspension of $A^+$. 
\begin{itemize}
\item In particular, as vector space
  $B(A)$ is given by :
\[
B(A)=T( A^+)=\bigoplus_{n \geqs 0} ( A^+)^{\otimes n}.
\]

\item A homogeneous element $\mathbf{a}$ of tensor degree $n$ is denoted using
  the \emph{bar} notation, that is 
\[
\mathbf{a}=[a_1| \ldots| a_n]
\]
and its degree is 
\[
\deg_B(\mathbf{a})=\sum_{i=1}^n\left(\deg_A(a_i)-1\right).
\]
\item The coalgebra structure comes from the natural deconcatenation coproduct,
  that is 
\[
\Delta([a_1| \ldots| a_n])=\sum_{i=0}^n
[a_1| \ldots| a_i]\otimes [a_{i+1}| \ldots| a_n].
\]
\end{itemize}
\end{defn}
\begin{rem} This construction can be seen as a simplicial total complex
  associated to the complex $A$ (Cf. \cite{BKMTM}). The augmentation makes it
  possible to use directly $A^+$ without referring to the tensor coalgebra over $A$
  and without the need of killing the degeneracies.

However this simplicial presentation usually masks the need of working with the
shifted complex.  
\end{rem}
We associate to any bar element $[a_1| \ldots | a_n]$ the function $\eta(i)$
giving its ``partial'' degree
\[
\eta(i)=\sum_{k=1}^{i}(\deg_{A}(a_k)-1).
\]
The original differential $d_A$ induces a differential $D_1$ on $B(A)$ given
by
\[
D_1([a_1| \ldots | a_n])=-\sum_{i=1}^n(-1)^{\eta(i-1)}
[a_1| \ldots |d_A(a_i)| \ldots| a_n]
\]
where the initial minus sign comes from the fact that the differential on the shifted
complex $A[1]$ is $-d_A$.
Moreover, the multiplication on $A$ induces another differential $D_2$ on $B(A)$ given by
\[
D_2([a_1| \ldots| a_n])=-\sum_{i=1}^{n-1}(-1)^{\eta(i)}
[a_1| \ldots |\mu_A(a_i,a_{i+1})| \ldots | a_n]
\]
where the signs are coming from Koszul commutation rules (due to the
shifting). One checks that the two 
differentials anticommute providing $B(A)$ with a total differential.
\begin{defn} The total differential on $B(A)$ is defined by
\[
d_{B(A)}=D_1+D_2.
\]
\end{defn}
The last structure arising with the bar construction is the graded shuffle
product
\[
[a_1 | \ldots | a_n]\sha [a_{n+1} | \ldots | a_{n+m}]=
\sum_{\sigma \in sh(n,m)}(-1)^{\ve_{gr}(\sigma)}[a_{\sigma(1)}, \ldots, a_{\sigma(n+m)}]
\]
where $sh(n,m)$ denotes the permutation of $\{1, \ldots , n+m\}$ such that if
$1\leqs i<j \leqs n$ or $n+1\leqs i<j \leqs n+m$ then $\sigma(i)<\sigma(j)$. The
sign is the graded signature of the permutation (for the degree in $A^+[1]$) given by 
\[
\ve_{gr}(\sigma)=\sum_{\substack{i<j \\ \sigma(i)>\sigma(j)}}
(\deg_{A}(a_i)-1)(\deg_{A}(a_j)-1).
\]  
With these definitions, one can explicitly check the following
\begin{prop} Let $A$ be a (Adams/weight graded) c.d.g.a.
The operations $\Delta$, $d_{B(A)}$ and $\sha$ together with the obvious unit and counit
give $B(A)$ a structure of (Adams graded) commutative graded differential Hopf algebra. 

In particular, these operations induce on $\HH^0(B(A))$, and more generally on
$\HH^*(B(A))$, a (Adams graded) commutative Hopf algebra structure. This (Adams
graded) algebra 
is cohomologically graded in the case of $\HH^*(B(A))$ and cohomologically 
graded concentrated in degree $0$ in
the case of $\HH^0(B(A))$.
\end{prop}
We recall that the set of indecomposable elements of an augmented c.d.g.a. is
defined as the augmentation ideal $I$ modulo products, that is $I/I^2$. Applying a
general fact about Hopf algebras, the
coproduct structure on $\HH^0(B(A))$ induces a coLie
algebra structure on its set  of indecomposable elements.


\subsection{Bar elements}
Considering the bar construction over $\cNg{X}{\bullet}$, part of the issue
is to associate to any cycle $\Lc_W$ and $\Lcu_W$ a corresponding element in
$\HH^0(B(\cNg{X}{\bullet}))$. 

As the weight $1$ cycles $\Lc_0$ and $\Lc_1$ have $0$ differential in
$\cNg{X}{\bullet}$, there are obvious corresponding bar elements:
\begin{equation}\label{L0L1:bar}
\Lc_0^B=[\Lc_0] \qquad \mx{and} \qquad \Lc_1^B=[\Lc_1].
\end{equation} 

 Let $\mc M_X$ denote the indecomposable elements of
 $\HH^0(B(\cNg{X}{\bullet}))$ and let $\tau$ be
 the morphism exchanging the two factors of $\HH^0(B(\cNg{X}{\bullet}))\otimes
 \HH^0(B(\cNg{X}{\bullet}))$. We denote by
  \[
d_{\Delta}=\frac{1}{2}\left(\Delta-\tau \Delta\right)
\]
the differential  on the coLie algebra $\mc
  M_X$ induced by the 
  coproduct on $\HH^0(B(\cNg{X}{\bullet}))$.
In general, one should have the following.
\begin{claim} For any Lyndon word $W$ (of length greater or equal to $2$), 
there exist elements $\Lc^B_W$ and ${\Lcub_W}$ in $B(\cNg{X}{\bullet})$ of bar
degree $0$ satisfying:
\begin{itemize}
\item Let $d_B$ denotes the  total bar differential
  $d_B=d_{B(\cNg{X}{\bullet})}$. Then one has:
\[
d_B(\Lc^B_W)=d_B({\Lcub_W})=0.
\]
\item The tensor degree $1$ part of $\Lc^B_W$ (resp. ${\Lcub_W}$) is given by
  $[\Lc_W]$ (resp. $[\Lcu_W]$). 
\item The elements  $\Lc^B_W$
  (resp. ${\Lcub_W}$) satisfy the differential equation \eqref{ED-L}
  (resp. \eqref{ED-Lc}) in $\mc M_X$. That is 
\[
d_{\Delta}(\Lc_W^B)=-\left(\sum_{U<V} a_{U,V}^W\Lc_U^B \Lc_V^B + 
\sum_{U,V}b_{U,V}^W\Lc_{U}^B{\Lcub_V}\right) \qquad \in \mc M_X \wedge \mc M_X
\]
and 
\begin{multline*}
d_{\Delta}({\Lcub_W})=-\left(\sum_{0<U<V} \ap_{U,V}^W{\Lcub_U}{ \Lcub_V} +
\sum_{U,V}\bp_{U,V}^W\Lc_{U}^B{\Lcub_V}+
\sum_{V}\ap_{0,V}\Lc_0^B\Lc_V^B\right)  
\\ \quad \in \mc M_X \wedge \mc M_X
\end{multline*}
where the overall minus sign is due to shifting reasons.
\end{itemize}
\end{claim}

  The obstruction for proving the general
statement lies in the control of the global combinatorics relating $D_1$, $D_2$
and the two systems of differential equations. 

Below, one finds some elements $\Lc_W^B$ and ${\Lcub_W}$ corresponding
to the previously described examples together with some relations among those
elements. Once the element $\Lc_W^B$ are explicitly described, it is a
straightforward computation to check that it lies in the kernel of $d_B$ and
this verification will be omitted. 

Note that all cycles $\Lc_W$ and $\Lcu_W$ are
in cohomological degree $1$, that is in $\cNg{X}{1}$. Thus, signs appearing in
the operations on the bar construction are much simpler as all terms in
$\deg_{A}(a_i)-1$ are $0$.

\begin{exm}[Weight $2$]
Cycles $\Lc_{01}$ and $\Lcu_{01}$ satisfy
$\dN(\Lc_{01})=\dN(\Lcu_{01})=\Lc_0\Lc_1$. Thus one can define
\begin{equation}\label{01B}
\Lc_{01}^B=[\Lc_{01}]-\frac{1}{2}\left([\Lc_{0}|\Lc_1]-[\Lc_{1}|\Lc_0] \right)
\quad \mx{and} \quad 
\Lcub_{01}=[\Lcu_{01}]-\frac{1}{2}\left([\Lc_{0}|\Lc_1]-[\Lc_{1}|\Lc_0] \right).
\end{equation}\label{u01B}

Remark that, looking at things modulo products, that is in $\mc M_X$,  the tensor
degree $2$ involves some choices. Instead of 
\[
-\frac{1}{2}\left([\Lc_{0}|\Lc_1]-[\Lc_{1}|\Lc_0] \right),
\]
we could have used
\[
-[\Lc_0|\Lc_1]
\qquad \mx{or} \qquad 
[\Lc_1|\Lc_0]
\]
and obtained the same elements in $\mc M_X$ as 
\[
-\frac{1}{2}\left([\Lc_{0}|\Lc_1]-[\Lc_{1}|\Lc_0] \right)=
-[\Lc_0|\Lc_1]+\frac{1}{2} \Lc_0^B\sha \Lc_1^B = 
[\Lc_1|\Lc_0]-\frac{1}{2} \Lc_0^B\sha \Lc_1^B.
\]
The above choice reflects in some sense that there is no preferred choice for either
\[
\dN(\Lc_{01})=\dN(\Lcu_{01})=\Lc_0\Lc_1
\quad \mx{or} \quad \dN(\Lc_{01})=\dN(\Lcu_{01})=-\Lc_1\Lc_0.
\] 

Recall that we have defined a cycle $\Lc_{01}(1)$ in $\cNg{X}{1}$ by 
\[
\Lc_{01}(1)=j^*(p^*\circ i_1^*(\ol{\Lc_{01}}))
\]
Building the cycle $\Lc_{011}$, we have introduced the cycle $\Lcu_{01}$ instead
of using the difference $\Lc_{01}-\Lc_{01}(1)$ in order to keep working with
equidimensional cycles. The ``correspondence'' 
\[
\Lcu_{01} \leftrightarrow \Lc_{01}-\Lc_{01}(1)
\]
becomes an equality in $\HH^0(B(\cNg{X}{\bullet}))$.

More precisely, using either the commutation of the above morphisms with the
differential or the expression of $\Lc_{01}(1)$ as parametrized cycle, one sees
that $\dN(\Lc_{01}(1))=0$ and one defines
\[
\Lc_{01}^B(1)=[\Lc_{01}(1)].
\]

A direct computation shows that 
\[
\Lc_{01}(1)=[t; 1-\frac{1}{x_1}, x_1,1-x_1].
\]
Now, from the expressions of $\Lc_{01}$, $\Lc_{01}(1)$ and $\Lcu_{01}$ as
parametrized cycles, one checks that in $\cNg{X}{1}$
\[
 \Lc_{01}-\Lc_{01}(1)=\Lcu_{01} +\dN(C_{01}) 
\]
where $C_{01}$ is the element of $\cNg{X}{0}$ defined by 
\[
C_{01}=-[t;y,\frac{y-\frac{x_1-t}{x_1}}{y-\frac{x_1-1}{x_1}},x_1,1-x_1] \subset X
\times \square^4.
\]
The bar element $C_{01}^B=[C_{01}]$ is of bar degree $-1$ and gives in
$B(\cNg{X}{\bullet})$
\[
\Lc_{01}^B-\Lc_{01}^B(1)={\Lcub_{01}}-d_B(C_{01}^B)
\]
and thus, the  equality  $\Lc_{01}^B-\Lc_{01}^B(1)={\Lcub_{01}} $ in the $\HH^0$.

For these weight $2$ examples , computing the deconcatenation coproduct is
trivial and gives the expected relation
\[
d_{\Delta}(\Lc_{01}^B)=d_{\Delta}({\Lcub_{01}})=-\Lc_0^B\wedge \Lc_1^B.
\]
Finally the  motive corresponding to $\Lc_{01}$ is the comodule generated by
$\Lcb_{01}$, that is the subvector space of $\mc M_X$ spanned by $\Lcb_{01}$,
$\Lcb_0$ and $\Lcb_1$.

\end{exm}
\begin{exm}[Weight $3$]\label{exm:001B011B}
The differentials of $\Lc_{001}$, $\Lcu_{001}$, $\Lc_{011}$, $\Lcu_{011}$ allow us
to easily write down the corresponding tensor degree $1$ and $2$. The
expressions below try to keep a symmetric presentation for the part in
 tensor degree $3$.

In the equations below, cycles $\Lc_W$ are simply denoted by $W$ and cycles
$\Lcu_W$ simply by $\ol{W}$. We will also use this abuse of notation later on
in weight $4$. One defines

\begin{align*}
\Lc_{001}^B&= [001]-\frac{1}{2}\left([0|01]-[01|0] \right)
+\frac{1}{4}\left([0|0|1] -  [0|1|0] + [1|0|0] \right),\\
{\Lcub_{001}}&=[\ol{001}]-\frac{1}{2}\left([0|01]-[01|0] \right)
+\frac{1}{4}\left([0|0|1] -  [0|1|0] + [1|0|0] \right)
\end{align*}
and 
\begin{align*}
\Lc_{011}^B&= [011]-\frac{1}{2}\left([\ol{01}|1]-[1|\ol{01}] \right)
+\frac{1}{4}\left([0|1|1] -  [1|0|1] + [1|1|0] \right),\\
{\Lcub_{011}}&=[\ol{011}]-\frac{1}{2}\left([\ol{01}|1]-[1|\ol{01}] \right)
+\frac{1}{4}\left([0|1|1] -  [1|0|1] + [1|1|0] \right).
\end{align*} 

As the cycles $\Lc_{001}$ and $\Lcu_{001}$ (resp. $ \Lc_{011}$ and $\Lcu_{011}$)
differ only by their first $\square^1$ factors, the arguments used to compare
$\Lc_{01}^B$ and $\Lcub_{01}$ apply here and give:
\[
\Lcb_{001}-\Lcb_{001}(1)=\Lcub_{001}
\quad \mx{ and } \quad 
\Lcb_{001}-\Lcb_{001}(1)=\Lcub_{001}
\qquad \in \mc M_X.
\]

The ``correction'' cycles giving the explicit relations between the cycles are
\[
C_{001}=-[t;s,\frac{s-\frac{x_2-t}{x_2}}{s-\frac{x_2-1}{x_2}}
, x_2, 1-\frac{x_2}{x_1}, x_1, 1-x_1]
\]
and 
\[
C_{011}=[t;s,\frac{s-\frac{x_2-t}{x_2}}{s-\frac{x_2-1}{x_2}}
, 1-x_2,\frac{x_1-x_2}{x_1-1}, x_1, 1-x_1].
\]

Now, computing the reduced coproduct $\Delta'=\Delta - 1\otimes \id-\id \otimes
1$ of $\Lcb_{001}$ gives:
\begin{multline*}
\Delta'(\Lcb_{001})=-\frac{1}{2}\left( [0]\otimes [01]-[01]\otimes [0]\right) \\
+\frac{1}{4}\left( [0]\otimes [0|1] - [0]\otimes [1|0] + [1]\otimes [0|0]
\right. \\
\left.
+[0|0]\otimes [1] -[0|1]\otimes [0] + [1|0]\otimes [0]\right).
\end{multline*}
As $[0|0]=1/2 \Lcb_0\sha \Lcb_0$, one has modulo products
\[
\Delta'(\Lcb_{001})=-\frac{1}{2}\left(
[0]\otimes \left([01]-\frac{1}{2}([0|1]-[1|0])\right)
- \left([01]-\frac{1}{2}([0|1]-[1|0])\right) \otimes [0] 
\right).
\] 
Similar computations apply to $\Lcub_{001}$, $\Lcb_{011}$ and $\Lcub_{011}$ and
give in $\mc M_X \wedge \mc M_X$:
\[
d_{\Delta}(\Lcb_{001})=d_{\Delta}(\Lcub_{001})=-\Lcb_0\w \Lcb_{01}
\quad \mx{and} \quad 
d_{\Delta}(\Lcb_{011})=d_{\Delta}(\Lcub_{011})=-\Lcub_{01} \w \Lcb_1.
\]

One should remark that the equality $\Lcub_{01}=\Lcb_{01}-
\Lcb_{01}(1)$ in the $\HH^0$ implies
\begin{equation}\label{011BeqT}
d_{\Delta}(\Lcb_{011})=-\left(\Lcb_{01}-\Lcb_{01}(1)\right) \w \Lcb_1
\end{equation}
which is (up to a global minus sign) the equation satisfied by $\T{011}$ as shown at Example \ref{exm:dT}.

Finally, the corresponding comodules giving motives associated to the cycle
$\Lc_{001}$ and $\Lc_{011}$ are the subvector spaces of $\mc M_X$ generated
respectively by 
\[
\left<\Lcb_{001},\, \Lcb_{01},\,\Lcb_{0},\,\Lcb_{1}\right>
\]
and 

\[
\left<\Lcb_{011},\, \Lcb_{01},\, \Lcb_{01}(1),\,\Lcb_{0},\,\Lcb_{1}\right>.
\]
 
\end{exm}
The above arguments apply similarly in weight $4$. Hence, we will describe below
the case of $\Lc_{0011}$ as it gives a ``preview'' of the combinatorial
difficulties related to the bar construction context. 
\begin{exm}[Weight $4$: $\Lcb_{0011}$]
We give below an element $\Lcb_{0011}$ in the bar construction with zero
differential and with tensor degree $1$ part equal to $\Lc_{0011}$:
\begin{align*}
\Lcb_{0011}=&[0011]-\frac{1}{2}\left([0|011]-[0|011]+[\ol{001}|1]-[1|\ol{001}]
-[01|\ol{01}]+[\ol{01}|01] \right) \\
&+\frac{1}{4}\left(-[0|1|\ol{01}]+[0|1|01]-[01|0|1]+[\ol{01}|0|1]
-[\ol{01}|1|0]+[01|1|0]\right. \\
&\left. -[1|0|01]+[1|0|\ol{01}]+[1|\ol{01}|0]
+[1|01|0]+[0|\ol{01}|1]+[0|01|1]\right) \\
&-\frac{1}{2}\left([0|0|1|1]
-[1|1|0|0]
\right).
\end{align*}
Identifying the reduced coproduct of $\Lcb_{0011}$ with 
\begin{multline}\label{res-D-0011}
-\frac{1}{2}\left(
\Lcb_{0}\otimes \Lcb_{011}-\Lcb_{011}\otimes \Lcb_0 
+\Lcub_{001}\otimes \Lcb_1 -\Lcb_1 \otimes \Lcub_{001}
\right. \\
\left.
-\Lcb_{01}\otimes \Lcub_{01}  +\Lcub_{01}\otimes \Lcb_{01} 
\right)
\end{multline}
is more difficult than in the previous cases. First of all, one remarks that in the
above expression the terms in $\left(\cNg{X}{1}\right)^{\otimes 2} \otimes 
\left(\cNg{X}{1}\right)^{\otimes 2}$
are coming only from $-\Lcb_{01}\otimes \Lcub_{01} $ and $\Lcub_{01}\otimes
\Lcb_{01}  $ and cancel each other. Thus the expression \eqref{res-D-0011} has
no term in $\left(\cNg{X}{1}\right)^{\otimes 2} \otimes 
\left(\cNg{X}{1}\right)^{\otimes 2}$. On the other hand, the terms in
$\left(\cNg{X}{1}\right)^{\otimes 2} \otimes 
\left(\cNg{X}{1}\right)^{\otimes 2}$ coming from
$\Delta'(\Lcb_{0011})$ are given by
\begin{multline*}
-\frac{1}{2}\left([0|0]\otimes[1|1]-[1|1]\otimes[0|0]\right)= \\
-\frac{1}{8}\left(\Lcb_0\sha\Lcb_0\otimes \Lcb_1\sha\Lcb_1-
\Lcb_1\sha\Lcb_1\otimes \Lcb_0\sha\Lcb_0\right)
\end{multline*}
and thus are zero in $\mc M_X \w \mc M_X$.

The terms in $\cNg{X}{1}\otimes \cNg{X}{1}$ in the above expression
\eqref{res-D-0011} obviously agree with the corresponding terms of
$\Delta'(\Lcb_{0011})$ as the term of tensor degree $2$ of $\Lcb_{0011}$ is
written down that way. 

Computations below are done in $B(\cNg{X}{\bullet})\otimes
B(\cNg{X}{\bullet})$. They will induce the expected relation in $\mc M_X\w \mc
M_X$ after going to the $\HH^0$ and taking the quotient modulo shuffle product.
Let $\pi_n : B(\cNg{X}{\bullet}) \lra \left(\cNg{X}{\bullet}\right)^{\otimes
  n}$ be the projection to the $n$-th tensor factor. From the above discussion
it is enough to compute 
$\Delta'(\pi_3(\Lcb_{0011}))$ and part of $\Delta'(\pi_{4}(\Lcb_{0011}))$.

First the definition of the coproduct gives 
\begin{align*}
4\Delta'(\pi_3(\Lcb_{0011}))=&
\left(
[0]\otimes [\ol{01}|1]+[0|01]\otimes [1]-[01]\otimes [0|1]-[0|1]\otimes [\ol{01}]
\right)\\
&+
\left(
-[0]\otimes [1|\ol{01}]-[01|0]\otimes [1]+[01]\otimes [1|0]+[1|0]\otimes [\ol{01}]
\right)\\
&+(-1)
\left(
 [\ol{01}|1]\otimes [0]+[1]\otimes [0|01]-[0|1]\otimes [01]- [\ol{01}]\otimes[0|1]
\right)\\
&+(-1)
\left(
- [1|\ol{01}]\otimes [0]-[1]\otimes [01|0]+[1|0]\otimes [01]+[\ol{01}]\otimes [1|0]
\right)\\
&+
\left(
[0]\otimes \left( [01|1]+[1|01]\right)+\left( [01|1]+[1|01]\right) \otimes [0]
\right)\\
&+
\left(
[1]\otimes \left( [\ol{01}|0]+[0|\ol{01}]\right)
+\left( [\ol{01}|0]+[0|{01}]\right) \otimes [0]
\right).
\end{align*}

The four first lines of the above equality correspond to $4$ times the terms of Equation
\eqref{res-D-0011} in $\cNg{X}{\bullet}\otimes{\cNg{X}{\bullet}}^{\otimes
  2}\oplus {\cNg{X}{\bullet}}^{\otimes 2}\otimes \cNg{X}{\bullet}$. The two last
lines can be written as 
\[
[0]\otimes \left([01]\sha [1]\right)+\left([01]\sha [1]\right)\otimes [0] +
[0]\otimes \left([\ol{01}]\sha [0]\right)
+\left([\ol{01}]\sha [0]\right)\otimes [1] 
\]
which covers many terms of a shuffle between terms in the $\HH^0$ of the bar
construction. As an example,
\[
[0]\otimes \left([01]\sha [1]\right)
\]
covers many terms of
\[
[\Lcb_0]\otimes \left(\Lcb_{01}\sha \Lcb_1\right)=[0]\otimes \left(
[01]\sha[1]-\frac{1}{2}([0|1]-[1|0])\sha[1]
\right).
\]

Computing the reduced coproduct of $\pi_4(\Lcb_{0011})$ gives
\begin{multline*}
\Delta'(\pi_4(\Lcb_{0011}))=\left(-\frac{1}{2}\right)\left(\frac{1}{4}\right)
\left( 4[0]\otimes [0|1|1]+4[0|0]\otimes[1|1]+4[0|0|1]\otimes[1] \right.
\\ \left.
-4[1]\otimes [1|0|0]+4[1|1]\otimes[0|0]-4[1|1|0]\otimes[0]
\right).
\end{multline*}
where the factors $1/4$ and $4$ make it easier to relate
$\Delta'(\pi_4(\Lcb_{0011}))$ with shuffle products and the corresponding terms
in Equation \eqref{res-D-0011}.

We have already remarked that the terms $[0|0]\otimes [1|1]$ and $[1|1]\otimes
[0|0]$ in the equation above can be expressed as shuffles. The four other terms are 
similar. Hence we will only discuss the case of $[0]\otimes [0|1|1]$.
One can write 
\begin{align*}
4[0|1|1]&=2[0|1|1] +2[1|1|0]+2[0|1|1] -2[1|1|0] \\
&=[0|1|1] +[1|1|0]-[1|0|1]+[0|1|1]+[1|0|1]+[1|1|0]+2[0|1|1] -2[1|1|0].
\end{align*}
Now, one remarks that 
\[
[0|1|1]+[1|0|1]+[1|1|0]=\frac{1}{2}[0]\sha[1]\sha [1]
=\frac{1}{2}\Lcb_0\sha \Lcb_1 \sha \Lcb_1
\]
and that the tensor degree $3$ part of $-\frac{1}{2}\Lcb_{01}\sha\Lcb_1$ is equal to
\[
\frac{-1}{2}\,\frac{-1}{2}\left(
[0|1]\sha[1]-[1|0]\sha[1]
\right)=\frac{1}{4}\left( 2[0|1|1]-2[1|1|0] \right).
\]
Then one can conclude that 
\begin{align*}\label{Dred0011}
\Delta'(\Lcb_{0011})=&
-\frac{1}{2}\left(
\Lcb_{0}\otimes \Lcb_{011}-\Lcb_{011}\otimes \Lcb_0 
+\Lcub_{001}\otimes \Lcb_1 -\Lcb_1 \otimes \Lcub_{001}\right.\\
&\hskip 40ex \left.
-\Lcb_{01}\otimes \Lcub_{01} +\Lcub_{01}\otimes \Lcb_{01} 
\right)\\
&+\frac{1}{4}\left(
\Lcb_{01}\sha\Lcb_1\otimes \Lcb_0+\Lcub_{01}\sha\Lcb_0\otimes \Lcb_1
+\Lcb_0 \otimes \Lcb_{01}\sha\Lcb_1 \right. \\
& \hskip 40ex \left. +\Lcb_1 \otimes\Lcub_{01}\sha\Lcb_0
\right) \\
&\frac{-1}{16}\left(
-\Lcb_0 \sha \Lcb_{1}\sha\Lcb_1\otimes \Lcb_0-\Lcb_{0}\sha\Lcb_0\sha \Lcb_1\otimes
\Lcb_1 \right.\\
&\hskip 10ex \left.
+\Lcb_0 \otimes \Lcb_0 \sha \Lcb_{1}\sha\Lcb_1 +\Lcb_1
\otimes\Lcb_{0}\sha\Lcb_0\sha \Lcb_1
\right).
\end{align*}
In $\mc M_X \w \mc M_X$, one simply gets 
\begin{equation}\label{dbar0011modsha}
d_{\Delta}(\Lcb_{0011})=- \left(\Lcb_{0}\w \Lcb_{011}+\Lcub_{001}\w \Lcb_1 
-\Lcb_{01}\w \Lcub_{01} \right)
\end{equation}
and using the relations between $\Lcb_W$ and $\Lcub_W$, one recovers (up to a
global minus sign) the
differential equation associated to $\T{0011}$
\begin{equation}\label{dbar0011-tree}
d_{\Delta}(\Lcb_{0011})=-\left(\Lcb_{0}\w \Lcb_{011}+(\Lcb_{001}-\Lcb_{001}(1))\w \Lcb_1 
+\Lcb_{01}\w \Lcb_{01}(1)\right). 
\end{equation}

The associated motive is as above the sub-vector space of $\mc M_X$ generated by
\[
\Lcb_{0},\,\Lcub_{001},\,\Lcb_{0011},\,\Lcb_{01},\, \Lcb_{011},\, 
 \Lcub_{01} , \,\Lcb_1. 
\]
\end{exm}
\subsection{Goncharov's motivic coproduct}
In this subsection, we would like to illustrate how the differential equation
satisfied by the elements $\Lcb_W$, written using its ``tree differential form''
(that is using the elements $\Lcb_U(1)$ instead of the elements $\Lcub_U$),
gives another expression for Goncharov's motivic coproduct.

Work of Levine \cite{LEVTMFG} 
insures that the above 
differential coincides with Goncharov's motivic coproduct for motivic iterated
integrals (modulo products). We will not review this theory here but
only recall some of the needed properties satisfied by Goncharov's motivic iterated
integrals \cite{GSFGGon}. 
A short exposition of the combinatorics involved is also recalled in
\cite{GanGonLev05}[Section 8]. 

For our purpose, it is enough to consider motivic iterated integrals as
degree $n$ generating elements $I(a_0;a_1,\ldots, a_n; a_{n+1})$ of a Hopf
algebra with $a_i$ in $\A^1(\Q)$. They are subject to the following relations.
\begin{description}
\item[Path composition] for $x$ in $\A^1(\Q)$, one has 
\[
I(a_0;a_1,\ldots, a_n; a_{n+1})=\sum_{k=0}^nI(a_0;a_1,\ldots, a_k;x)
I(x;a_{k+1},\ldots, a_n; a_{n+1}).
\]
\item[Inversion] which relates $I(a_0;a_1,\ldots, a_n; a_{n+1})$ and 
$I(a_{n+1};a_n,\ldots, a_1; a_{0})$
\[
I(a_{n+1};a_n,\ldots, a_1; a_{0})=(-1)^nI(a_0;a_1,\ldots, a_n; a_{n+1}).
\] 
\item[Unit and neutral identities]
\[
\mbox{for }a\neq b \quad I(a;b)=1 \qquad \mx{and} \qquad
I(a_0;a_1,\ldots, a_n; a_0)=0.
\]
\item[Rescaling] If $a_{n+1}$ and at least one of the $a_i$ is not zero then
\[
I(0;a_1,\ldots, a_n, a_{n+1})=I(0;a_1/a_{n+1},\ldots, a_n/a_{n+1}, 1).
\]
\item[Regularization]
\[
I(0;1;1)=I(0;0;1)=0.
\]
\end{description}
The product is given by the shuffle relations
\[
I(a;a_1,\ldots, a_n; b)I(a;a_{n+1},\ldots, a_{n+m}; b)=
\sum_{\sigma \in sh(n,m)}I(a;a_{\sigma(1)},\ldots, a_{\sigma(n+m)} ;b)
\]
where $sh(n,m)$ denotes the set of permutations preserving the order of the
ordered subset $\{1, \ldots, n\}$ and $\{n+1, \ldots, n+m\}$. Such a motivic
iterated integral corresponds formally to the iterated integral
\[
\int_{\Delta_{a_0,a_{n+1}}} \frac{dt}{t-a_1} \w \cdots \w \frac{dt}{t-a_n}
\]
with $\Delta_{a_0,a_{n+1}}$ the image of the standard simplex induced by a path
from $a_0$ to $a_{n+1}$. The above relations reflect the relations satisfied by
the integrals.

The coproduct is given by the formula
\begin{multline*}
\Delta^M(I(a_0;a_1\ldots,a_n;a_{n+1}))= \\
\sum_{
\substack{
(a_{k_1},\ldots,a_{k_r})\\
\{k_1,\ldots,k_r\}\subset\{1,\ldots ,n\}
}
}I(a_0;a_{k_1},\ldots,a_{k_r};a_{n+1}) \otimes 
\prod_{l=0}^r I(a_{k_l};a_{k_l+1},\ldots, a_{k_{l+1}-1};a_{k_{l+1}})
\end{multline*}
with the convention that $r$ runs from $0$ to $n$ and that $k_0=0$ and
$k_{r+1}=n+1$. 

Now, considering the reduced coproduct $\Delta^{'M}=\Delta^M-(1\otimes \id
+\id\otimes 1)$ on the space of indecomposable elements (that is modulo products), the
above formula reduces to
\begin{multline*}
\Delta^{'M}(I(a_0;a_1\ldots,a_n;a_{n+1}))= \\
\sum_{\substack{k<l \\ k\neq l-1}}
I(a_0;a_{1},\ldots,a_k,a_l,a_{l+1},\ldots a_n;a_{n+1}) \otimes 
I(a_{k};a_{k+1},\ldots, a_{l-1},a_{l}).
\end{multline*}

This formula can be pictured placing the $a_i$ on a semicircle in the order
dictated by their indices. Then a term in the above sums corresponds to a
non-trivial chord between to vertices:
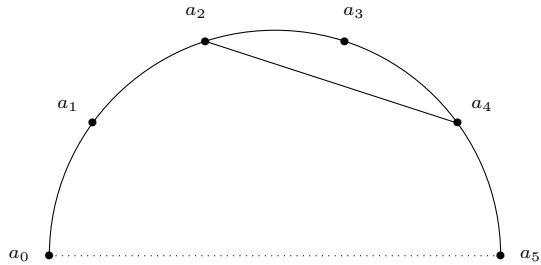
\[
\begin{tikzpicture}
\node[mathsc](origine) at (0,0){};
\foreach \i in {0,1,...,5}
\node[mathsc] (a\i) at (180/5*\i:3) {\bullet};
\draw (a0) arc (0:180:3);
\foreach \k/\a in {0/a_5,1/a_4,2/a_3,3/a_2,4/a_1,5/a_0}
\node[mathsc] (b\k) at (180/5*\k:3.4) {\a};
\draw[dotted] (a0.center) -- (a5.center);
\draw (a1.center)--(a3.center);
\end{tikzpicture}
\]

Considering the relation between multiple polylogarithms and iterated integrals,
we want to related our expression of the differential of $\Lcb_{011}$ at $t$ to
the reduced coproduct for the 
motivic iterated integral $I(0;0,x,x;1)$ for $x=t^{-1}$.  
From the semicircle
representation, one sees that there are five terms to consider:
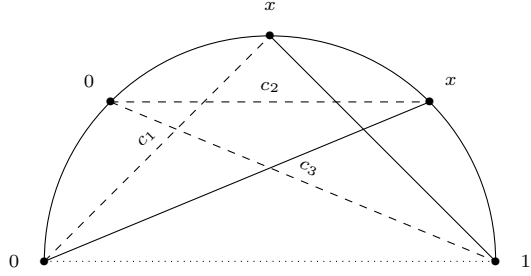
\[
\begin{tikzpicture}
\node[mathsc](origine) at (0,0){};
\foreach \i in {0,1,...,4}
\node[mathsc] (a\i) at (180/4*\i:3) {\bullet};
\draw (a0) arc (0:180:3);
\foreach \k/\a in {0/1,1/x,2/x,3/0,4/0}
\node[mathsc] (b\k) at (180/4*\k:3.4) {\a};
\draw[dotted] (a0.center) -- (a5.center);
\draw[dashed]  (a4.center)--(a2.center) 
node[midway,above,sloped,mathsc]{c_1};
\draw (a4.center)--(a1.center);
\draw[dashed] (a3.center)--(a1.center) 
node[midway,above,sloped,mathsc]{c_2};
\draw[dashed] (a3.center)--(a0.center) 
node[above,midway,sloped,mathsc]{c_3};
\draw (a2.center)--(a0.center);
\end{tikzpicture}
\]

However, the chord $c_1$ gives a zero term modulo products as
$I(0;x,x;1)=\frac{1}{2}I(0;x;1)I(0;x;1)$ and chord $c_2$ and $c_3$
 give terms equal to $0$
using the regularization relations. There are thus only two terms to consider 
\[
I(0;0,x;1)\otimes I(x;x;1)
\qquad \mx{and} \qquad 
I(0;x;1)\otimes I(0;0,x;x).
\]

Using path composition, inversion and  regularization relations, in the set
of indecomposable elements, one has
\[
I(x;x;1)=I(0;x;1)+I(x;x;0)=I(0;x;1)-I(0;x;x)=I(0;x;1).
\]
Thus the first term equals
\[
I(0;0,x;1)\otimes I(x;x;1)=I(0;0,x;1)\otimes I(0;x;1).
\]
From the rescaling relation, the second term equals
\[
I(0;x;1)\otimes I(0;0,1;1)
\]
and one can write modulo products
\begin{equation}\label{DI3}
\Delta^{'M}(I(0;0,x,x;1))=I(0;0,x;1)\otimes I(0;x;1) + I(0;x;1)\otimes I(0;0,1;1).
\end{equation}
Keeping in mind that, for $x=t^{-1}$, $I(0;x;1)$ corresponds to the fiber at $t$
of $\Lcb_1$ ($t\neq 1$) and that $ I(0;0,x;1)$ corresponds to the fiber at $t$
of $\Lcb_{01}$ (any $t$), the above formula \eqref{DI3} corresponds exactly to Equation
\eqref{011BeqT}:
\[
d_{\Delta}(\Lcb_{011})=-\left(\Lcb_{01}-\Lcb_{01}(1)\right) \w \Lcb_1=-\left(
\Lcb_{01}\w \Lcb_1 + \Lcb_1 \w \Lcb_{01}(1)\right).
\]

The fact that, as in the above example, Goncharov's motivic iterated integrals
corresponding to multiple polylogarithms in one variable, satisfy the tree
differential equations \eqref{ED-T} is easily checked for Lyndon words with one
$1$ (that is for the classical polylogarithms) but can also be checked for some
Lyndon words with two $1$'s or more. It seems to be a general behavior. However,
the above example and the 
example of $I(0;0,0,x,x;1)$ corresponding to $\Lcb_{0011}$ below show that it
involves using all the relations in order to pass from Goncharov's formula to the
shape of Equation \eqref{ED-T}.

The case of $\Lcb_{0011}$ involves more computations but works essentially as
the case of $\Lcb_{011}$. The reduced coproduct for $I(0;0,0,x,x;1)$ modulo
products gives nine terms corresponding to the nine chords below:
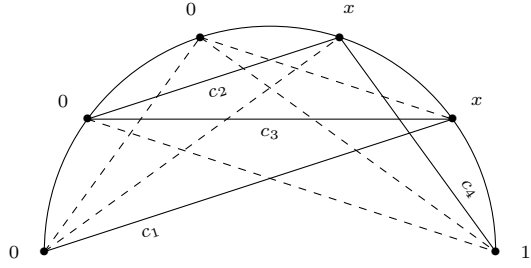
\[
\begin{tikzpicture}
\node[mathsc](origine) at (0,0){};
\foreach \i in {0,1,...,5}
\node[mathsc] (a\i) at (180/5*\i:3) {\bullet};
\draw (a0) arc (0:180:3);
\foreach \k/\a in {0/1,1/x,2/x,3/0,4/0,5/0}
\node[mathsc] (b\k) at (180/5*\k:3.4) {\a};
\draw[dashed] (a5.center)--(a3.center);
\draw[dashed] (a5.center)--(a2.center);
\draw[dashed] (a3.center)--(a1.center);
\draw[dashed] (a3.center)--(a0.center);
\draw[black] (a5.center)--(a1.center) node [near start,below,mathsc,sloped]  {c_1};
\draw[dashed] (a4.center)--(a0.center);
\draw[black] (a4.center)--(a1.center) node [midway,below,mathsc,sloped]  {c_3};
\draw[black] (a4.center)--(a2.center) node [midway,below,mathsc,sloped]  {c_2};
\draw[black] (a2.center)--(a0.center) node [near end,above,mathsc,sloped]  {c_4};
\end{tikzpicture}
\]

The five dashed chords give terms equal to $0$ for one of the following reasons:
$I(a;\ldots;a)=0$, regularization relations or shuffle relations. Hence we are
left with four terms.
The chord $c_1$ gives, using the rescaling relation,
\[
I(0;x;1)\otimes I(0;0,0,1;1)
\]
corresponding to $\Lcb_1 \w \Lcb_{001}(1)$. The chord $c_2$ gives a term in 
\[
I(0;0,x,x;1)\otimes I(0;0;x)
\]
corresponding to $\Lcb_0\w \Lcb_{011}$. 
The chord $c_3$ gives, using the
rescaling relation, a term in
\[
I(0;0,x;1)\otimes I(0;0,x;x)=I(0;0,x;1)\otimes I(0;0,1;1)
\]
corresponding to $\Lcb_{01}\w \Lcb_{01}(1)$. Finally the chord $c_4$ gives,
using the path composition and regularization relations 
\begin{align*}
I(0;0,0,x;1)\otimes I(x;x;1)&=I(0;0,0,x;1)\otimes I(0;x;1)+I(0;0,0,x;1)\otimes
I(x;x;0) \\
&=I(0;0,0,x;1)\otimes I(0;x;1).
\end{align*}

Finally, $ \Delta^{'M}(I(0;0,0,x,x;1))$ can be written as 
\begin{multline}\label{DI4}
\Delta^{'M}(I(0;0,0,x,x;1))=
I(0;0,x,x;1)\otimes I(0;0;x)+
I(0;0,0,x;1)\otimes I(0;x;1)+ \\
I(0;x;1)\otimes I(0;0,0,1;1) 
+I(0;0,x;1)\otimes I(0;0,1;1).
\end{multline}
This expression corresponds to Equation \eqref{dbar0011-tree}:
\[
d_{\Delta}(\Lcb_{0011})=-\left(\Lcb_{0}\w \Lcb_{011}+(\Lcb_{001}-\Lcb_{001}(1))\w \Lcb_1 
+\Lcb_{01}\w \Lcb_{01}(1)\right). 
\]

\section{Integrals and multiple zeta values}\label{sec:int}
We present here a sketch of how to associate an integral to cycles
$\Lc_{01}$, $\Lcu_{01}$ and $\Lc_{011}$. The author will directly follow the
algorithm described in \cite{GanGonLev05}[Section 9] and put in detailed practice in
\cite{GanGonLev06}.  There will be no general review of the direct Hodge realization
from Bloch-Kriz  motives \cite{BKMTM}[Section 8 and 9]. Gangl, Goncharov and
Levin's construction seems to consist in setting particular choices of
representatives in the intermediate Jacobians for their
algebraic cycles. 

However, the goal of this paper is not to formalize such an idea. That is why
computations below are only outlined. In particular, the 
lack of precise knowledge of the ``algebraico-topological cycle algebra''
described in \cite{GanGonLev05} makes it difficult to control how
``negligible'' cycles are killed looking at the $\HH^0$ of its bar
construction.

\subsection{An integral associated to $\Lc_{01}$ and $\Lcu_{01}$}
We recall the parametrized cycle expression for $\Lc_{01}$:
\[
\Lc_{01}=[t;1-\frac{t}{x_1}, x_1,1-x_1]\qquad
\subset X\times \square^3.
\]

One wants to bound $\Lc_{01}$ by an algebraic-topological cycle
in a larger bar construction (not described here) introducing topological
variables $s_i$ in real 
simplices  
\[
\Delta_s^n=\{0\leqs s_1 \leqs \cdots \leqs s_n \leqs 1\}.
\]
 Let $d^s : \Delta^n_s \ra \Delta_s^{n-1}$ denotes the simplicial differential
\[
d^s=\sum_{k=0}^{n}(-1)^k i_k^*
\] 
where $i_k : \Delta_s^{n-1} \ra \Delta_s^{n}$ is given by the face
$s_k=s_{k+1}$ in $\Delta_s^n$ with the usual conventions for $k=0,n$.

Defining
\[
C_{01}^{s,1}=-[t; 1-\frac{s_2t}{x_1},x_1,1-x_1]
\] 
for $s_2$ going from $0$ to $1$, one sees that $d^s(C_{01}^{s,1})=\Lc_{01}$ as $s_2=0$
implies that the first cubical coordinate is $1$. The algebraic boundary $\dN$
of $C_{01}^{s,1}$ is given by the intersection with the faces of $\square^3$:
\[
\dN(C_{01}^{s,1})=-[t;s_2t, 1-s_2t] 
\qquad \subset X \times \square^2.
\]

This cycle is part of the boundary of a larger ``simplicial'' algebraic cycle
\[
C_{01}^{s,2}=-[t;s_2t,1-s_1t].
\]
Computing the simplicial differential of $C_{01}^{s,2}$ gives
\[
d^s(C_{01}^{s,2})=[t;s_2 t,1-s_2t]-[t;t,1-s_1t] 
\qquad \subset X \times \square^2
\]
with $0\leqs s_2\leqs 1$ in the first term and $0\leqs s_1 \leqs 1$ in the
second term.

Note that the cycle $[t; t,1-s_1t]$ is negligible as it is a product
\[
[t; t,1-s_1t]=\Lc_0 \, [t;1-s_1t]
\]
and thus can be canceled in the bar construction setting as the multiplicative
boundary of 
\[
-[\Lc_0 | [t, 1-s_1 t] ].
\]

Thus, up to negligible terms,
\[
(d^s+\dN)(C_{01}^{s,1}+C_{01}^{s,2})=\Lc_{01}.
\]

Now, we fix the situation at the fiber $t_0$ and following Gangl, Goncharov and
Levin, we associate to the algebraic cycle 
$\Lc_{01}|_{t=t_0}$ the integral $I_{01}(t_0)$ of the standard volume form
\[
\frac{1}{(2i\pi)^2}\frac{dz_1}{z_1}\w \frac{dz_2}{z_2}
\]
over the simplex given by $C_{01}^{s,2}$. That is :
\begin{align*}
I_{01}(t_0)&=-\frac{1}{(2i\pi)^2}\int_{0\leqs s_1 \leqs s_2  \leqs 1}
 \frac{ds_2}{s_2} \w
\frac{-t_0\, ds_1}{1-t_0s_1} \\
& =\frac{-1}{(2i\pi)^2}\int_{0\leqs s_1 \leqs s_2  \leqs 1}
\frac{ ds_1}{t^{-1}_0-s_1} \w \frac{ds_2}{s_2} =\frac{-1}{(2i\pi)^2}Li^{\C}_{2}(t_0).
\end{align*}
In particular, this expression is valid for $t_0=1$, as is the cycle $\Lc_{01}|_{t=1}$,
and gives $-1/(2\pi^2) \zeta(2)$.

Before presenting the weight $3$ example of $\Lc_{011}$, we describe shortly below the
situation for $\Lcu_{01}$. In the bar construction the element $\Lcub_{01}$
is equal  to the difference $\Lcb_{01}-\Lcb_{01}(1)$. Associating an integral
to $\Lcu_{01}$ works in the same way as the cycle $\Lc_{01}$
but it also reflects the correspondence with $\Lcb_{01}-\Lcb_{01}(1)$. 

The expression of $\Lcu_{01}$ in terms of parametrized cycle is given by 
\[
\Lcu_{01}=[t; \frac{x_1-t}{x_1-1},x_1,1-x_1]
\]
and can be bounded using the ``simplicial'' algebraic cycle
\[
C_{\ol{01}}^{s,1}=-[t;\frac{x_1-s_2t}{x_1-s_2},x_1,1-x_1].
\]
Now, the algebraic boundary of $C_{\ol{01}}^{s,1}$ gives two terms
\[
\dN(C_{\ol{01}}^{s,1})=-[t;s_2t, 1-s_2t]+[t;s_2,1-s_2].
\]
Then one defines $C_{\ol{01}}^{s,2}$ for simplicial variables $0\leqs s_1 \leqs
s_2 \leqs 1$ as 
\[
C_{\ol{01}}^{s,2}=-[t;s_2t,1-s_1t]+[t;s_2,1-s_1]
\]
whose simplicial boundary cancels $\dN(C_{\ol{01}}^{s,1})$ up to negligible
cycles. Again, fixing a fiber $t_0$ the integral associated to
$\Lcu_{01}|_{t=t_0}$ is the integral of the 
standard volume form over the $C_{\ol{01}}^{s,2}$:
\begin{equation*}
I_{\ol{01}}(t_0)=-\frac{1}{(2i\pi)^2}
\left(\int_{0\leqs s_1 \leqs s_2  \leqs 1}
 \frac{ds_2}{s_2} \w
\frac{-t_0\, ds_1}{1-t_0s_1} 
+ 
\int_{0 \leqs s_1 \leqs s_2 \leqs 1}\frac{ ds_2}{s_2} \w \frac{ -ds_1}{1-s_1}
\right).
\end{equation*}

This expression is exactly the difference
\[
I_{\ol{01}}(t_0)=\frac{-1}{(2i \pi)^2}\left(Li^{\C}_2(t_0)-Li^{\C}_2(1)\right). 
\]

\subsection{An integral associated to $\Lc_{011}$}


Let's recall the expression of $\Lc_{011}$ as parametrized cycle:
\[
\Lc_{011}=-[t; 1-\frac{t}{x_2},1-x_2,\frac{x_1-x_2}{x_1-1},x_1,1-x_1].
\]

As previously, one wants to bound $\Lc_{011}$ by an algebraic-topological
cycle.
Hence we define
\[
C_{011}^{s,1}=[t; 1-\frac{s_3t}{x_2},1-x_2,\frac{x_1-x_2}{x_1-1},x_1,1-x_1]
\]    
for $s_3$ going from $0$ to $1$. Then $d^s(C_{011}^{s,1})=\Lc_{011}$ as $s_3=0$
implies that the first cubical coordinate is $1$.

Now the algebraic boundary $\dN$ of $C_{011}^{s,1}$ is given by the intersection with the
codimension $1$ faces of $\square^5$: 
\[
\dN(C_{011}^{s,1})=[t;1-s_3t, \frac{x_1-s_3t}{x_1-1},x_1,1-x_1].
\]
We can again bound this cycle by introducing a new simplicial variable $0\leqs
s_2\leqs s_3$ and the cycle 
\[
C_{011}^{s,2}=[t;1-s_3t,\frac{x_1-s_2t}{x_1-s_2/s_3},x_1,1-x_1].
\]
The intersection with the face of the simplex $\{0\leqs s_2 \leqs s_3 \leqs
1\}$ given by 
$s_2=0$  leads to an empty cycle (as one cubical coordinate
equals $1$) while the intersection with face $s_3=1$ leads to a ``negligible''
cycle. 
Thus, the simplicial boundary of $C_{011}^{s,2}$ satisfies (up to a negligible term)
\[
d^s(C_{011}^{s,2})=-\dN(C_{011}^{s,1})=-[t;1-s_3t,
\frac{x_1-s_3t}{x_1-1},x_1,1-x_1]. 
\]
Its algebraic boundary is given  by 
\[
\dN(C_{011}^{s,2})=-[t;1-s_3t,s_2t,1-s_2t]
+[t;1-s_3t,\frac{s_2}{s_3},1-\frac{s_2}{s_3}].  
\]
Finally, we introduce a last simplicial variable $0\leqs s_1 \leqs s_2$ and a
purely topological cycle
\[
\wt{C_{011}}^{s,3}=-[t; 1-s_3t,
s_2t,1-s_1t] + [t;1-s_3t,\frac{s_2}{s_3},1-\frac{s_1}{s_3}] 
\]
whose simplicial differential is (up to negligible terms) given in one hand by the face
$s_1=s_2$:
\[
[t;1-s_3t,s_2t,1-s_2t]
-[t;1-s_3t,\frac{s_2}{s_3},1-\frac{s_2}{s_3}]
\]
which is equal to $-\dN(C_{011}^{s,2})$; and in the other hand by the face $s_2=s_3$:
\[
-[t;1-s_3t,s_3t,1-s_1t].
\]
In order to cancel this extra term, we defined $C_{011}^{s,3}$ by
\[
C_{011}^{s,3}=\wt{C_{011}}^{s,3}+[t;1-s_2t,s_3t,1-s_1t]
\]
whose algebraic boundary is $0$ (up to negligible terms). 

Finally one has
\[
(d^s+\dN)(C_{011}^{s,1}+C_{011}^{s,2}+C_{011}^{s,3})=\Lc_{011}
\] 
up to negligible terms.

Now, we fix the situation at the fiber $t_0$ and following Gangl, Goncharov and
Levin, we associate to the algebraic cycle 
$\Lc_{011}|_{t=t_0}$ the integral $I_{011}(t_0)$ of the standard volume form
\[
\frac{1}{(2i\pi)^3}\frac{dz_1}{z_1}\wedge \frac{dz_2}{z_2}\w \frac{dz_3}{z_3}
\]
over the simplex given by $C_{011}^{s,3}$. That is :
\begin{multline*}
I_{011}(t_0)=-\frac{1}{(2i\pi)^3}\int_{0\leqs s_1 \leqs s_2 \leqs s_3 \leqs 1}
\frac{t_0\, ds_3}{1-t_0s_3} \w \frac{ds_2}{s_2} \w
\frac{t_0\, ds_1}{1-t_0s_1} \\
+ 
\frac{1}{(2i\pi)^3}\int_{0\leqs s_3 \leqs 1}\frac{t_0\, ds_3}{1-t_0s_3}
\int_{0 \leqs s_1 \leqs s_2 \leqs 1}\frac{ ds_2}{s_2} \w \frac{ ds_1}{1-s_1}\\
+ 
\frac{1}{(2i\pi)^3}\int_{0\leqs s_1 \leqs s_2 \leqs s_3 \leqs 1}
\frac{t_0ds_2}{1-t_0s_2}\w \frac{ds_3}{s_3}\w \frac{t_0ds_1}{1-t_0s_1}
.
\end{multline*}
 
Taking care of the change of sign due to the numbering, the first term  in the
above sum is (for $t_0 \neq 0$ and up to the factor $(2i\pi)^{-3}$) equal to  
\[
Li_{1,2}^{\C}(t_0)=\int_{0\leqs s_1 \leqs s_2 \leqs s_3 \leqs 1}
\frac{ ds_1}{t_0^{-1}-s_1} \w \frac{ds_2}{s_2} \w
\frac{ds_3}{t_0^{-1}-s_3}
\] 
while (up to the same multiplicative factor) the second term is equal to 
\[
-Li_1^{\C}(t_0)Li_2^{\C}(1)
\]
and the third term is equal to 
\[
Li_{2,1}^{\C}(t_0).
\]

Globally the integral is well defined for $t_0=0$ and, which is the
interesting part, also for $t_0=1$ as the divergencies as $t_0$ goes to $1$ cancel
each other in the above sum. A simple computation and the shuffle relation for
$Li_1^{\C}(t_0)Li_2^{\C}(t_0)$ show that  the integral associated to the fiber of
$\Lc_{011}$ at $t_0=1$ is given by
\[
(2i\pi)^{3}I_{011}(1)=-Li_{2,1}^{\C}(1)=-\zeta(2,1).
\]

\providecommand{\bysame}{\leavevmode\hbox to3em{\hrulefill}\thinspace}
\providecommand{\MR}{\relax\ifhmode\unskip\space\fi MR }
\providecommand{\MRhref}[2]{%
  \href{http://www.ams.org/mathscinet-getitem?mr=#1}{#2}
}
\providecommand{\href}[2]{#2}


\begin{thebibliography}{GGL09}

\bibitem[BK94]{BKMTM}
Spencer Bloch and Igor Kriz, \emph{Mixed Tate motives}, Ann. of Math.
  \textbf{140} (1994), no.~3, 557--605.

\bibitem[Blo86]{BlochACHKT}
Spencer Bloch, \emph{Algebraic cycles and higher {$K$}-theory}, Adv. in Math.
  \textbf{61} (1986), no.~3, 267--304. 

\bibitem[Blo91]{BlochLie}
Spencer Bloch, \emph{Algebraic cycles and the {L}ie algebra of mixed {T}ate
  motives}, J. Amer. Math. Soc. \textbf{4} (1991), no.~4, 771--791. 


\bibitem[Blo97]{BlochLMM}
Spencer Bloch, \emph{Lectures on mixed motives}, Algebraic geometry---{S}anta
  {C}ruz 1995, Proc. Sympos. Pure Math., vol.~62, Amer. Math. Soc., Providence,
  RI, 1997, ~329--359. 

\bibitem[GGL07]{GanGonLev06}
H.~Gangl, A.~B. Goncharov, and A.~Levin, \emph{Multiple logarithms, algebraic
  cycles and trees}, Frontiers in number theory, physics, and geometry. {II},
  Springer, Berlin, 2007, 759--774.

\bibitem[GGL09]{GanGonLev05}
\bysame, \emph{Multiple polylogarithms, polygons, trees and algebraic cycles},
  Algebraic geometry---{S}eattle 2005. {P}art 2, Proc. Sympos. Pure Math.,
  vol.~80, Amer. Math. Soc., Providence, RI, 2009, 547--593.

\bibitem[Gon95]{PAGGon}
Alexander~B. Goncharov, \emph{Polylogarithms in arithmetic and geometry},
  Proceedings of the {I}nternational {C}ongress of {M}athematicians, {V}ol.\ 1,
  2 ({Z}\"urich, 1994) (Basel), Birkh\"auser, 1995, 374--387.

\bibitem[Gon05]{GSFGGon}
A.~B. Goncharov, \emph{Galois symmetries of fundamental groupoids and
  noncommutative geometry}, Duke Math. J. \textbf{128} (2005), no.~2, 209--284.

\bibitem[Lev94]{LevBHCG}
Marc Levine, \emph{Bloch's higher {C}how groups revisited}, Ast\'erisque
  (1994), no.~226, 10, 235--320, $K$-theory (Strasbourg, 1992). 

\bibitem[Lev05]{KTMMLevine}
Marc Levine, \emph{Mixed motives}, Handbook of {$K$}-{T}heory (E.M Friedlander
  and D.R. Grayson, eds.), vol.~1, Springer-Verlag, 2005, 429--535.

\bibitem[Lev11]{LEVTMFG}
\bysame, \emph{{Tate motives and the fundamental group.}}, {Srinivas, V. (ed.),
  Cycles, motives and Shimura varieties. Proceedings of the international
  colloquium, Mumbai, India, January 3--12, 2008. New Delhi: Narosa Publishing
  House/Published for the Tata Institute of Fundamental Research. 265-392,
  (2011).}

\bibitem[LV12]{LodValAO}
Jean-Louis Loday and Bruno Vallette, \emph{Algebraic operads}, Vol. 346,
  Springer-Verlag, 2012.

\bibitem[Reu93]{ReuFLA93}
Christophe Reutenauer, \emph{Free {L}ie algebras}, London Mathematical Society
  Monographs. New Series, vol.~7, The Clarendon Press Oxford University Press,
  New York, 1993, Oxford Science Publications.

\bibitem[Sou12]{SouMPCC}
Ismael Soud{\`e}res, \emph{Cycle complex over {$\mathbb P^1$} minus {$3$}
  points : toward multiple zeta values cycles},
  http://arxiv.org/abs/1210.4653, 2012.

\bibitem[Spi]{SpitzweckSCVTM}
Markus Spitzweck, \emph{Some constructions for Voevodsky's triangulated
  categories of motives}, preprint.

\bibitem[Spi01]{OAMSpit}
Markus Spitzweck, \emph{Operads, algebras and modules in model categories and
  motives}, Ph.D. thesis, Bonn Universit{\"a}t,
http://hss.ulb.uni-bonn.de/2001/0241/0241.htm, 2001.

\bibitem[Tot92]{Totaro}
Burt Totaro, \emph{Milnor {$K$}-theory is the simplest part of algebraic
  {$K$}-theory}, $K$-Theory \textbf{6} (1992), no.~2, 177--189.

\end{thebibliography}
\end{document}